\documentclass{amsart}

\usepackage{amsmath}
\usepackage{amssymb}
\usepackage{bibentry}
\usepackage{enumitem}
\usepackage{graphicx}
\usepackage{mathrsfs}
\usepackage{tabu}
\usepackage{tikz}
\usepackage[colorlinks=true,
            linktocpage=true,
            linkcolor=magenta,
            citecolor=magenta,
            urlcolor=magenta]{hyperref}
\PassOptionsToPackage{hyphens}{url}
\usepackage{cleveref}
\usepackage[all]{xypic}

\usetikzlibrary{arrows}
\usetikzlibrary{shapes}
\usetikzlibrary{snakes}

\newtheorem{theorem}{Theorem}[section]
\newtheorem{proposition}[theorem]{Proposition}
\newtheorem{lemma}[theorem]{Lemma}
\newtheorem{corollary}[theorem]{Corollary}
\theoremstyle{definition}
\newtheorem{definition}[theorem]{Definition}

\newtheorem{question}[theorem]{Question}

\DeclareMathOperator{\res}{\upharpoonright}
\newcommand{\ran}{\operatorname{ran}}

\renewcommand{\mod}{\text{ }\textrm{mod}\text{ }}

\newcommand{\seq}[1]{\langle #1 \rangle}

\newcommand{\set}[1]{\{ #1 \}}

\renewcommand{\iff}{\leftrightarrow}

\newcommand{\Lang}{\mathsf{L}}

\newcommand{\RCA}{\mathsf{RCA}}
\newcommand{\ACA}{\mathsf{ACA}}
\newcommand{\WKL}{\mathsf{WKL}}

\newcommand{\B}{\mathsf{B}}
\newcommand{\I}{\mathsf{I}}

\newcommand{\RT}{\mathsf{RT}}
\newcommand{\COH}{\mathsf{COH}}
\newcommand{\CAC}{\mathsf{CAC}}
\newcommand{\ADS}{\mathsf{ADS}}
\newcommand{\SRT}{\mathsf{SRT}}

\newcommand{\TT}{\mathsf{TT}}

\newcommand{\D}{\mathsf{D}}

\newcommand{\GS}{\mathsf{GS}}
\newcommand{\GST}{\mathsf{GST}}
\newcommand{\wGS}{\mathsf{wGS}}
\newcommand{\SGST}{\mathsf{SGST}}
\newcommand{\SubP}{\text{-}\mathsf{Subset}}

\newcommand{\ncred}{\nleq_{\text{\upshape c}}}

\newcommand{\Tabove}{\geq_{\text{\upshape T}}}
\newcommand{\Tred}{\leq_{\text{\upshape T}}}

\newcommand{\Tequiv}{\equiv_{\text{\upshape T}}}

\newcommand{\height}{\mathrm{ht}}

\newcommand{\converges}{\operatorname{\downarrow}}
\newcommand{\diverges}{\operatorname{\uparrow}}

\newcommand{\G}{\mathtt{G}}
\renewcommand{\setminus}{\smallsetminus}
\newcommand{\cl}{\operatorname{cl}}

\newcommand{\topol}[1]{\mathcal{#1}}
\newcommand{\fres}[3]{\ran(#1) \res #2~[#3]}
\newcommand{\Lreq}{\mathcal{R}}
\newcommand{\Preq}{\mathcal{Q}}

\begin{document}

\title{The Ginsburg--Sands theorem and computability theory}

\author[Benham]{Heidi Benham}
\address{Department of Mathematics\\
University of Connecticut\\
Storrs, Connecticut U.S.A.}
\email{heidi.benham@uconn.edu}

\author[DeLapo]{Andrew DeLapo}
\address{Department of Mathematics\\
University of Connecticut\\
Storrs, Connecticut U.S.A.}
\email{andrew.delapo@uconn.edu}

\author[Dzhafarov]{Damir D. Dzhafarov}
\address{Department of Mathematics\\
University of Connecticut\\
Storrs, Connecticut U.S.A.}
\email{damir@math.uconn.edu}

\author[Solomon]{Reed Solomon}
\address{Department of Mathematics\\
University of Connecticut\\
Storrs, Connecticut U.S.A.}
\email{solomon@math.uconn.edu}

\author[Villano]{Java Darleen Villano}
\address{Department of Mathematics\\
University of Connecticut\\
Storrs, Connecticut U.S.A.}
\email{javavill@uconn.edu}

\begin{abstract}
	The Ginsburg--Sands theorem from topology states that every infinite topological space has an infinite subspace homeomorphic to exactly one of the following five topologies on $\omega$: indiscrete, discrete, initial segment, final segment, and cofinite. The original proof is nonconstructive, and features an interesting application of Ramsey's theorem for pairs ($\mathsf{RT}^2_2$). We analyze this principle in computability theory and reverse mathematics, using Dorais's formalization of CSC spaces. Among our results are that the Ginsburg--Sands theorem for CSC spaces is equivalent to $\mathsf{ACA}_0$, while for Hausdorff spaces it is provable in $\mathsf{RCA}_0$. Furthermore, if we enrich a CSC space by adding the closure operator on points, then the Ginsburg--Sands theorem turns out to be equivalent to the chain/antichain principle ($\mathsf{CAC}$). The most surprising case is that of the Ginsburg--Sands theorem restricted to $T_1$ spaces. Here, we show that the principle lies strictly between $\mathsf{ACA}_0$ and $\mathsf{RT}^2_2$, yielding arguably the first natural theorem from outside logic to occupy this interval. As part of our analysis of the $T_1$ case we introduce a new class of purely combinatorial principles below $\mathsf{ACA}_0$ and not implied by $\mathsf{RT}^2_2$ which form a strict hierarchy generalizing the stable Ramsey's theorem for pairs ($\SRT^2_2$). We show that one of these, the $\Sigma^0_2$ subset principle ($\Sigma^0_2\SubP$), has the property that it, together with the cohesive principle ($\COH$), is equivalent over $\RCA_0$ to the Ginsburg--Sands theorem for $T_1$ CSC spaces.
\end{abstract}

\thanks{The authors were partially supported by a Focused Research Group grant from the National Science Foundation of the United States, DMS-1854355. They thank Natasha Dobrinen, Denis Hirschfeldt, Ludovic Levy Patey, Theodore Slaman, and Keita Yokoyama for insightful discussions during the preparation of this article, and the anonymous referee for valuable comments and suggestions for improvement.}

\maketitle


\section{Introduction}\label{sec:intro}

In their 1979 paper ``Minimal infinite topological spaces'' \cite{GS-1979}, Ginsburg and Sands proved the following classification theorem for topological spaces, described by Marron and McMaster \cite[pp.~26--27]{MM-1995} as a ``[significant] application of Ramsey's theorem''. 

\begin{theorem}[Ginsburg and Sands \cite{GS-1979}]\label{thm:mainGS}
	Every infinite topological space has a subspace homeomorphic to one of the following topologies on $\omega$:
	\begin{enumerate}
		\item indiscrete (only $\emptyset$ and $\omega$ are open);
		\item initial segment (the open sets are $\emptyset$, $\omega$, and all intervals $[0,n]$ for $n \in \omega$);
		\item final segment (the open sets are $\emptyset$, $\omega$, and all intervals $[n,\infty)$ for $n \in \omega$);
		\item discrete (all subsets are open);
		\item cofinite (the open sets are $\emptyset$, $\omega$, and all cofinite subsets of $\omega$).
	\end{enumerate}
\end{theorem}

\noindent Since the result is not especially well-known, we include the original proof from \cite{GS-1979} in \Cref{sec:GS_proof} below. As Ginsburg and Sands explain \cite[p.~574]{GS-1979}, any infinite subspace of a topological space homeomorphic to one of the above five is itself homeomorphic to the same, and no two distinct such spaces can be homeomorphic to each other. Thus, the above topologies are ``minimal'' in the same sense that a homogeneous set is for a given coloring of tuples of integers: if a homogeneous set has color $i$ for some coloring then so does every subset, and no homogeneous set can simultaneously have colors $i$ and $j \neq i$.

In this article, we will analyze the effective and proof-theoretic content of the Ginsburg--Sands theorem using the frameworks of computability theory and reverse mathematics.

Computability theory is concerned with the question of which subsets of the natural numbers can be specified by an algorithm, and by extension, seeks to determine the algorithmic properties of mathematical theorems that can be phrased in terms of, or represented by, countable objects. Reverse mathematics, on the other hand, seeks to classify the logical strengths of mathematical theorems by identifying the minimal axioms necessary for their proofs, commonly using subsystems of second-order arithmetic as benchmarks. There is a well-understood relationship between these two endeavors, with notions and results in one often leading to notions and results in the other. This combined perspective often provides deeper insights into the logical underpinnings of mathematical theorems, particularly with respect to the combinatorial tools and methods needed to prove them (see, e.g., \cite{Hirschfeldt-2014}, p.~65).

We refer to Shore \cite{Shore-2010} for a thorough discussion of the interplay between computability and reverse mathematics. For general background on computability, see Soare \cite{Soare-2016} and Downey and Hirschfeldt \cite{DH-2010}. The standard reference on reverse mathematics is Simpson \cite{Simpson-2009}, with newer treatments by Hirschfeldt \cite{Hirschfeldt-2014} and Dzhafarov and Mummert \cite{DM-2022}. 

Ramsey's theorem, and many related principles from combinatorics, model theory, and set theory have been studied extensively in this context. A partial survey appears in \cite[Chapters 8 and 9]{DM-2022}. Topological theorems of various kinds have also been considered, but because of the need for countable representations a lot here depends on how exactly topological spaces are coded. Thus, different approaches exist for studying metric spaces (see \cite{DM-2022}, Sections 10.1--10.5 for an overview), certain uncountable spaces (see Mummert \cite{Mummert-2005} and Mummert and Lempp \cite{MS-2006}), and countable spaces with countably based topologies (see Dorais \cite{Dorais-2011}). Topology has also been studied using higher-order reverse mathematics (Hunter \cite{Hunter-2008}, Normann and Sanders \cite{NS-2019}, and Sanders \cite{Sanders-2020, Sanders-2021}). A more detailed discussion of these various viewpoints appears in \cite[Section 10.8]{DM-2022}.

Our investigation of the Ginsburg--Sands theorem will use Dorais's formalization of countable second-countable (CSC) spaces in classical (second-order) reverse mathematics. The original account \cite{Dorais-2011} of this is unpublished, so we review some of the main definitions and the features most relevant for our purposes in \Cref{sec:CSC}. Additional background can be found in \cite[Section 10.8.1]{DM-2022}. Other recent applications besides our own include the work of Shafer \cite{Shafer-2020} and DeLapo \cite{DeLapo-TA}.


Our goal is to better understand the Ginsburg--Sands theorem in isolation, but also with respect to Ramsey's theorem for pairs, which figures in its proof. For the remainder of this section, then, we recall the statement of Ramsey's theorem, and of some of its consequences. Throughout, our notation will be standard unless otherwise noted. We assume familiarity with the usual benchmark subsystems $\RCA_0$, $\WKL_0$, and $\ACA_0$ of second-order arithmetic. For $n \geq 1$, we use $\mathsf{I}\Sigma^0_n$ to refer to the \emph{$\Sigma^0_n$ induction scheme}, which consists of all formulas of the form
\[
	[\varphi(0) \wedge (\forall x)[\varphi(x) \to \varphi(x+1)]] \to (\forall x)\varphi(x)
\]
where $\varphi(x)$ is a $\Sigma^0_n$ formula in the language of second-order arithmetic. All implications between principles and schemes in the sequel will be assumed to be over $\RCA_0$, unless otherwise noted. We also follow the convention of formulating all definitions in $\RCA_0$ for generality, thereby using $\mathbb{N}$ in place of $\omega$ so as to refer to the set of natural numbers in a possibly nonstandard model.

\begin{definition}
	For $X \subseteq \mathbb{N}$ and $m,n \geq 1$, $[X]^n = \set{\seq{x_0,\ldots,x_{n-1}} : x_0 < \cdots < x_{n-1}}$, and an \emph{$m$-coloring of $[X]^n$} is a map $c : [X]^n \to m = \set{0,\ldots,m-1}$. We write $c(x_0,\ldots,x_{n-1})$ in place of $c(\seq{x_0,\ldots,x_{n-1}})$. A set $H \subseteq X$ is \emph{homogeneous} for $c$ if $c \res [H]^n$ is constant. $\RT^n_m$ is the statement that for every infinite set $X$ and every $c : [X]^n \to m$ there exists an infinite homogeneous set for $c$.
\end{definition}

\noindent It is easy to see that for each $m$, the principle $\RT^1_m$ is provable in $\RCA_0$. Hirst \cite[Theorem 6.4]{Hirst-1987} showed that $(\forall m)[\RT^1_m]$ is equivalent (over $\RCA_0$) to $\mathsf{B}\Sigma^0_2$, the \emph{$\Sigma^0_2$ bounding} scheme. We recall that if $\Gamma$ is any collection of formulas of $\Lang_2$, the scheme $\B\Gamma$ consists of all formulas of the form
\[
	(\forall z)[(\forall x < z)(\exists y)\varphi(x,y) \to (\exists w)(\forall x < z)(\exists y < w) \varphi(x,y)]
\]
for $\varphi(x,y)$ in $\Gamma$. By well-known results of Kirby and Paris \cite{PK-1978}, $\mathsf{B}\Sigma^0_2$ is strictly stronger than $\mathsf{I}\Sigma^0_1$, and equivalent to $\mathsf{B}\Pi^0_1$.

For any fixed, standard $m \geq 2$, the usual collapsing argument shows that $\RT^n_m$ is equivalent over $\RCA_0$ to $\RT^n_2$. By a result of Jockusch \cite[Lemma 5.9]{Jockusch-1972}, $\RT^3_2$ is equivalent over $\RCA_0$ to $\ACA_0$. By contrast, Seetapun (see \cite{SS-1995}, Theorem 2.1) showed that $\RT^2_2$ is strictly weaker than $\ACA_0$, while Liu \cite[Corollary 1.6]{Liu-2012} showed that it does not imply $\WKL_0$. A modern proof of Seetapun's theorem uses the so-called \emph{strong cone avoidance} of $\RT^1_2$ of Dzhafarov and Jockusch \cite[Lemma 3.2]{DJ-2009} (see also \cite{DM-2022}, Section 8.5.1).

Alongside $\WKL_0$, $\RT^2_2$ holds a special position among principles lying below $\ACA_0$, which have collectively come to be called the \emph{reverse mathematics zoo} (\cite{Dzhafarov-zoo}). Namely, virtually all of the known principles in the zoo are a provable consequence of $\WKL_0$ or $\RT^2_2$ or both, with the vast majority of these being a consequence of $\RT^2_2$ alone. While principles strictly weaker than $\ACA_0$ but strictly stronger than $\RT^2_2$ exist, only two natural examples are known to occupy this interval, the Chubb-Hirst-McNicholl tree theorem $(\TT^2_2$), introduced in \cite{CHM-2009}, and Milliken's tree theorem ($\mathsf{MTT}^2_2$), introduced in \cite{Milliken-1979,Milliken-1981}. Interestingly, both of these principles have their roots at least partially in logic, with $\TT^2_2$ being devised specifically in the context of computability theory and reverse mathematics, and $\mathsf{MTT}^2_2$ being originally motivated in part by its applications to consistency results in set theory (see also \cite{Milliken-1975}, Chapter 3, and \cite{HL-1966}, p.~360). (The reverse mathematical analysis of $\mathsf{MTT}^2_2$ was only performed recently, by Angl\`{e}s d'Auriac et al. \cite{ACDMP-2022}.) We will prove that the Ginsburg--Sands theorem for $T_1$ CSC spaces is another example of a principle strictly in-between $\ACA_0$ and $\RT^2_2$, and as such is arguably the first such theorem with no a priori connection to logic.

We mention, for completeness, that Angl\`{e}s d'Auriac et al. \cite[Theorem 4.15]{ACDMP-2022} established the fact that $\mathsf{MTT}^2_2$ does not imply $\ACA_0$. Dzhafarov and Patey \cite[Theorem 4.1]{DP-2017} showed that $\TT^2_2$ does not imply $\ACA_0$, which Chong et al. \cite[Theorem 1.1]{CLLY-2021} improved to show that $\TT^2_2$ does not imply even $\WKL_0$. Patey \cite[Theorem 1.3]{Patey-2016b} proved that $\RT^2_2$ does not imply $\TT^2_2$. It remains open whether $\mathsf{MTT}^2_2$ implies $\WKL_0$, or whether $\TT^2_2$ implies $\mathsf{MTT}^2_2$.

We now discuss two important restrictions of $\RT^2_2$, which are the \emph{stable Ramsey's theorem for pairs} and the \emph{$\Delta^0_2$ subset principle}.

\begin{definition}
	Given a set $X$, a coloring $c : [X]^2 \to 2$ is \emph{stable} if $(\forall x)[\lim_y c(x,y)$ exists$]$. A set $L \subseteq X$ is \emph{limit homogeneous} for $c$ if $(\exists i < 2)(\forall x \in L)[\lim_y c(x,y) = i]$.
	\begin{itemize}
		\item $\SRT^2_2$ is the restriction of $\RT^2_2$ to stable colorings.
		\item $\D^2_2$ is the statement that for every infinite set $X$ and every stable $c : [X]^2 \to 2$ there exists an infinite limit homogeneous set for $c$.
	\end{itemize}
\end{definition}

\noindent An infinite homogeneous set for a stable coloring is trivially limit homogeneous. Conversely, an infinite limit homogeneous set can be easily thinned to an infinite homogeneous set, but this thinning process in general requires $\mathsf{B}\Sigma^0_2$ (see Dzhafarov, Hirschfeldt, and Reitzes~\cite[Definition 6.1 and Proposition 6.2]{DHR-2022}). However, Chong, Lempp, and Yang \cite[Theorem 1.4]{CLY-2010} showed that $\D^2_2$ implies $\B\Sigma^0_2$, and hence that $\RCA_0$ proves $\SRT^2_2 \iff \D^2_2$. An alternative formulation of $\D^2_2$ is as the scheme consisting of all formulas of the form
\[
	(\forall x)[\varphi(x) \iff \neg \psi(x)] \to (\exists Y)[(\forall x)[x \in Y \to \varphi(x)] \vee (\forall x)[x \in Y \to \neg \varphi(x)]], 
\]
where $\varphi$ and $\psi$ range over $\Sigma^0_2$ formulas of the language of second-order arithmetic. In gist, this says that every $\Delta^0_2$-definable set (which may not exist) has an infinite subset in it or its complement (that does exist). The equivalence is provable in $\RCA_0$, and is a straightforward consequence of Shoenfield's limit lemma. We will return to this in \Cref{sec:T1}.

Another important principle in the study of $\RT^2_2$ is the \emph{cohesive principle}.

\begin{definition}
	For a family $\vec{R} = \seq{R_n : n \in \mathbb{N}}$ of subsets of $\mathbb{N}$, a set $Y$ is \emph{$\vec{R}$-cohesive} if $(\forall n)[Y \subseteq^* R_n \vee Y \subseteq^* \overline{R_n}]$. $\COH$ is the statement that for every family $\vec{R}$ of sets there exists an infinite $\vec{R}$-cohesive set.
\end{definition}

Famously, Cholak, Jockusch, and Slaman \cite[Lemma 7.11]{CJS-2001} (see also \cite{Mileti-2004}, Corollary A.1.4) showed that over $\RCA_0$, $\RT^2_2 \iff \SRT^2_2 + \COH$. This discovery allowed the study of $\RT^2_2$ to be divided separately into the study of the comparatively simpler principles $\SRT^2_2$ and $\COH$, and was a catalyst for much advancement in the area. The decomposition of principles into ``stable'' and ``cohesive'' parts has become standard practice, and we will give such a decomposition also for a version of the Ginsburg--Sands theorem below. Ours will have the surprising feature of being a decomposition into a $\Sigma^0_2$ subset principle, analogous to $\D^2_2$ above, and $\COH$.

In \cite[Theorem 9.1]{CJS-2001}, it was shown that $\COH$ does not imply $\SRT^2_2$, but whether or not $\SRT^2_2$ implies $\COH$ (or equivalently, $\RT^2_2$) was a major open question in reverse mathematics for many years. It was finally answered in 2014 by Chong, Slaman, and Yang \cite[Theorem 2.2]{CSY-2014}, who constructed a nonstandard model of $\RCA_0 + \SRT^2_2$ in which $\COH$ does not hold. More recently, Monin and Patey \cite[Theorem 1.4]{MP-2021}, improved this to a separation via an $\omega$-model. See \cite[Section 8.8]{DM-2022} for more on the history of this problem.

Hirschfeldt and Shore \cite{HS-2007} initiated the study of the \emph{chain/antichain principle} and the \emph{ascending/descending sequence principle}, which are both partition results concerning orderings, and are closely related to Ramsey's theorem for pairs.

\begin{definition}
	Given a partial order $(X,\leq_X)$, a subset $Y$ of $X$ is a \emph{chain} if $(\forall x,y \in Y)[x \leq_X y \vee y \leq_X x]$; an \emph{antichain} if $(\forall x,y \in Y)[x \neq y \to x \nleq_X y \wedge y \nleq_X x]$; an \emph{ascending sequence} if it is a chain and for all $x \in Y$, the set $\set{y \in Y : y \leq_X x}$ is finite; and a \emph{descending sequence} if it is a chain and for all $x \in Y$, the set $\set{y \in Y : y \geq_X x}$ is finite.
	\begin{itemize}
		\item $\CAC$ is the statement that for every infinite partial order $(P,\leq_P)$ there exists an infinite chain or antichain under~$\leq_P$.
		\item $\ADS$ is the statement that for every infinite linear order $(L,\leq_L)$ there exists an infinite ascending or descending sequence under $\leq_L$.
	\end{itemize}
\end{definition}

It is easy to see that $\RT^2_2$ implies $\CAC$, and that $\CAC$ implies $\ADS$. Hirschfeldt and Shore \cite[Corollaries 2.16 and 3.12]{HS-2007} exhibited an $\omega$-model satisfying $\RCA_0 + \CAC$ but not $\SRT^2_2$, as well as an $\omega$-model satisfying $\WKL_0$ but not $\ADS$. They also showed (\cite{HS-2007}, Proposition 2.10) that $\ADS$ implies $\COH$. This left open whether or not $\ADS$ implies $\CAC$ over $\RCA_0$, which was finally answered, in the negative and via an $\omega$-model, by Lerman, Solomon, and Towsner \cite[Theorem 1.7]{LST-2013}.

It is worth noting that $\CAC$ and $\ADS$ were also identified by Ginsburg and Sands \cite[p.~575]{GS-1979}, and it is these principles, rather than $\RT^2_2$, that they employed in their proof of \Cref{thm:mainGS}, which we present in the next section.

We follow the usual convention of also thinking of $\Pi^1_2$ principles such as those above as \emph{problems}, each having a class of \emph{instances}, and for each instance, a class of \emph{solutions} (see \cite[Chapter 3]{DM-2022} for a general discussion). Thus, for example, the instances of $\RT^2_2$ are all pairs $\seq{X,c}$ such that $X$ is an infinite set and $c : [X]^2 \to 2$ is a coloring, with the solutions to $c$ being all infinite homogeneous sets $H \subseteq X$.

\section{A proof of the Ginsburg--Sands theorem}\label{sec:GS_proof}

We present Ginsburg and Sands's proof of \Cref{thm:mainGS}, modified only to reflect the terminology defined in the previous section. As mentioned, we do this partly to showcase their argument, which is elementary and elegant. This makes it straightforward to formalize the proof in $\ACA_0$ for the types of topological spaces we will be concerned with. But more significantly, the proof breaks very clearly into three main parts which seem to involve different levels of non-constructivity. We give a schematic of these parts at the end of the section, and then in \Cref{sec:general,sec:T1} show that the distinctions between these parts is reflected in their proof-theoretic strengths when considered separately.

We proceed to the proof, recalling that the aim is to show that every infinite topological space has a subspace homeomorphic to either the indiscrete; initial segment; final segment; discrete; or cofinite topology on $\omega$.

\begin{proof}[Proof of \Cref{thm:mainGS} (\cite{GS-1979}, p.~576)]
	Fix an infinite topological space $X$. By passing to a subspace if needed, we may assume $X$ is countable. For each $x \in X$, let $\cl(x)$ denote the \emph{closure} of $x$ in $X$, i.e., the set of all $y \in X$ such that every open set containing $y$ also contains $x$. Then, for $x,y \in X$, define $x \sim y$ if $\cl(x) = \cl(y)$, which is an equivalence relation on $X$. If there is an $x$ such that the equivalence class $\set{y \in \omega : x \sim y}$ is infinite, then the subspace topology on this class is an indiscrete subspace of $X$. Moving forward, we can therefore assume that all equivalence classes under $\sim$ are finite. By choosing representatives, we may further assume that $\cl(x) \neq \cl(y)$ for all $x \neq y$ in $X$, so that $X$ itself is a $T_0$ space.
	
	We now define a relation $\leq_X$ on $X$, as follows. For $x,y \in X$, set $x \leq_X y$ if $x \in \cl(y)$. Clearly, $\leq_X$ is reflexive and transitive, and our assumption that $X$ is $T_0$ implies that $\leq_X$ is antisymmetric. So, $\leq_X$ is a partial ordering. By $\CAC$, there exists an infinite subset of $X$ that either forms a chain under $\leq_X$ or an antichain under $\leq_X$. By regarding this as a subspace, we may assume this subset is in fact of all of $X$ for simplicity.
	
	First, suppose $X$ is a chain under $\leq_X$. Applying $\ADS$, we can fix an infinite subset $Y$ of $X$ that is either an ascending or descending sequence under $\leq_X$. Say $Y$ is an ascending sequence under $\leq_X$, so that its elements may be listed as $x_0 <_X x_1 <_X \cdots$. We claim that $Y$ under the subspace topology is homeomorphic to the final segment topology on $\omega$ via the map $x_n \mapsto n$. To see this, note first that if $n < m$ then every open set in $X$ containing $x_n$ also contains $x_m$ since $x_n \in \cl(x_m)$. On the other hand, for every $n$ there is an open set containing $x_n$ but not any $x_m$ for $m < n$, since $x_n \notin \cl(x_m)$ for any such $m$. It follows that the open sets in $Y$ are precisely the sets of the form $\set{x_m : m \geq n}$ for $n \in \omega$, which proves the claim. The case where $Y$ is a descending sequence under $\leq_X$ is symmetric, and produces a subspace homeomorphic to the initial segment topology.
	
	We are thus left with the case where $X$ is an antichain. By definition of $\leq_X$, this means that $X$ is a $T_1$ space. Assume $X$ has no infinite subspace with the cofinite topology; we show it has an infinite discrete subspace. Since $X$ is $T_1$, every finite subset of $X$ is closed, and so every cofinite subset is open. Our assumption therefore implies that if $W$ is any infinite subspace of $X$, there exists an open set $U$ such that $U \cap W \neq \emptyset$ and $W \setminus U$ is infinite. Let $U_0$ be any such open set for $W = X$, and let $x_0$ be any element of $U_0$. Next, suppose inductively that, for some $n \in \omega$, we have defined $x_0,\ldots,x_n \in X$, along with open sets $U_0,\ldots,U_n$ satisfying the following: $X \setminus \bigcup_{i \leq n} U_i$ is infinite, and for all $i,j \leq n$, we have that $x_i \in U_j$ if and only if $i = j$. Letting $W = X \setminus \bigcup_{i \leq n} U_i$, there exists an open set $U$ such that $U \cap W \neq \emptyset$ and $W \setminus U$ is infinite. Let $x_{n+1}$ be any element of $U \cap W$, and let $U_{n+1} = U \setminus \set{x_0,\ldots,x_n}$. (Note that $U_{n+1}$ is open because $X$ is~$T_1$.) Clearly, the inductive conditions are maintained. In this way, we thus define an infinite subspace $Y = \set{x_n : n \in \omega}$ of $X$, and this subspace is discrete since for each $n$ we have that $U_n \cap Y = \set{x_n}$ by construction.
\end{proof}

Examining the proof, we can identify the following three main parts:
\begin{enumerate}
	\item the construction of the ordering $\leq_X$;
	\item the application of $\CAC$, followed by $\ADS$ in the chain case;
	\item the construction that handles the $T_1$ case.
\end{enumerate}
We will use this breakdown as a guide for how to organize our study of the Ginsburg--Sands theorem. We will then show that, for the formalization of topological spaces we employ, part (1) is equivalent over $\RCA_0$ to $\ACA_0$; part (2) to $\CAC$; and part (3) has a unique strength with respect to the standard subsystems of second-order arithmetic. The analysis of part (3) will consequently take the most work. As part of this, we will actually give two alternative proofs of this part.

\section{Background on CSC spaces}\label{sec:CSC}

\subsection{Spaces and subspaces}

Classically, a topological space is \emph{second-countable} if it has a countable basis. Restricting to spaces whose underlying set is countable, we obtain the notion of a \emph{countable second-countable (CSC)} space. We review the basic definitions of how such spaces are formalized in second-order arithmetic, following Dorais \cite{Dorais-2011}.

\begin{definition}
	A \emph{countable second-countable (CSC)} space is a tuple $\seq{X,\topol{U},k}$ as follows:
	\begin{enumerate}
		\item $X$ is a subset of $\mathbb{N}$;
		\item $\topol{U} = \seq{U_n : n \in \mathbb{N}}$ is a family of subsets of $X$ such that every $x \in X$ belongs to $U_n$ for some $n\in \mathbb{N}$;
		\item $k : X \times \mathbb{N} \times \mathbb{N} \to \mathbb{N}$ is a function such that for every $x \in X$ and all $m,n \in \mathbb{N}$, if $x \in U_m \cap U_n$ then $x \in U_{k(x,m,n)} \subseteq U_n \cap U_m$.
	\end{enumerate}
\end{definition}

\noindent We refer to $\topol{U}$ above a \emph{basis} for $\seq{X,\topol{U},k}$, and to each $U_n$ as a \emph{basic open set}. We write $U \in \topol{U}$ if $U = U_n$ for some $n$, and say $U$ is an \emph{element} of $\topol{U}$. We say $\seq{X,\topol{U},k}$ is an \emph{infinite} CSC space if $X$ is an infinite set. When arguing computability-theoretically, we say a CSC space $\seq{X,\topol{U},k}$ is \emph{computable} if $X$ is a computable subset of $\omega$, $\topol{U}$ is a uniformly computable sequence of sets, and $k$ is a computable function. For ease of notation, when no confusion can arise, we will identify $\seq{X,\topol{U},k}$ simply with $X$. Thus, we will refer, e.g., to the \emph{basic open sets} in $X$, etc.

The following lemma shows that the $k$ function can be extended to arbitrary finite intersections.

\begin{lemma}\label{lem:kStar}
	The following is provable in $\RCA_0$. Let $\seq{X,\topol{U},k}$ be a CSC space with $\topol{U} = \seq{U_n : n \in \mathbb{N}}$. There exists a function $\overline{k} : X \times \mathbb{N} \to \mathbb{N}$ such that for all $x \in X$ and $m \in \mathbb{N}$, if $F$ is the finite set coded by $m$ and $x \in \bigcap_{n \in F} U_n$, then $x \in U_{\overline{k}(x,m)} \subseteq \bigcap_{n \in F} U_n$.	
\end{lemma}

\begin{proof}
	We first define a function $g : X \times \mathbb{N} \times \mathbb{N} \to \mathbb{N}$ by primitive recursion, as follows. Fix $x \in X$ and $m \in \mathbb{N}$ coding a finite set $F$. From $m$ we can determine whether or not $F$ is empty. If $F = \emptyset$, let $g(x,m,s) = 0$ for all $s$. Otherwise, let the elements of $F$ be $n_0 < \ldots < n_{|F|-1}$. Let $g(x,m,0) = n_0$, and for $s > 0$ let 
	\[
		g(x,m,s) =
		\begin{cases}
			k(x,g(x,m,s-1),n_s) & \text{if } s < |F|,\\
			g(x,m,s-1) & \text{if } s \geq |F|.			
		\end{cases}
	\]
	By induction on $s < |F|$, if $x \in \bigcap_{i \leq s} U_{n_i}$ then $x \in U_{g(x,m,s)} \subseteq \bigcap_{i \leq s} U_{n_i}$. Thus, if we define $\overline{k}(x,m) = g(x,m,|F|)$ we obtain the desired function.
\end{proof}

\noindent We will write $\overline{k}(x,F)$ in place of $\overline{k}(x,m)$ when $F$ is coded by $m$.

The next result allows us to build CSC spaces out of arbitrary collections of sets.

\begin{proposition}[Dorais \cite{Dorais-2011}, Proposition 2.12]\label{prop:dorais_generate}
	The following is provable in $\RCA_0$. Given a set $X \subseteq \mathbb{N}$ and a collection $\seq{V_n : n \in \mathbb{N}}$ of subsets $X$, there exists a CSC space $\seq{X,\topol{U},k}$ with $\topol{U} = \seq{U_n : n \in \mathbb{N}}$ as follows:
	\begin{enumerate}
		\item for every $n \in \mathbb{N}$, $V_n \in \topol{U}$;
		\item for every $m \in \mathbb{N}$, $U_m = \bigcap_{n \in F} V_n$, where $F$ is the finite set coded by $m$.
	\end{enumerate}
\end{proposition}

\noindent We call $\seq{X,\topol{U},k}$ above the \emph{CSC space on $X$ generated by} $\seq{V_n : n \in \mathbb{N}}$. Since every finite set is coded by a number, the basis $\topol{U}$ here is closed under intersections.

As noted by Dorais \cite[Section 2]{Dorais-2011}, $\RCA_0$ cannot prove the existence of arbitrary unions of elements of a basis. Thus, the open sets in a CSC space are defined somewhat indirectly.

\begin{definition}
	Let $\seq{X,\topol{U},k}$ be a CSC space with $\topol{U} = \seq{U_n : n \in \mathbb{N}}$.
	\begin{enumerate}
		\item An \emph{open code} of $\seq{X,\topol{U},k}$ is a function $O : \mathbb{N} \to \mathbb{N}$.
		\item The \emph{open set coded by $O$} refers to $O_{\topol{U}} = \bigcup_{n \in \ran(O)} U_n$.
	\end{enumerate}
\end{definition}

\noindent The phrase ``refers to'' in (2) above is intentional. Since $O_{\topol{U}}$ may not exist as an actual set, we formally understand all references to $O_{\topol{U}}$ as abbreviations. So, for example, by $x \in O_{\topol{U}}$ or $O_{\topol{U}} \subseteq Y$ we really mean $(\exists n \in \ran(O))[x \in U_n]$ and $(\forall x)[(\exists n \in \ran(O))[x \in U_n] \to x \in Y]$, respectively. Of course, if $O_{\topol{U}}$ happens to exist then we may also interpret these in the usual sense.

We move on to the definition of subspaces.

\begin{definition}\label{def:subspace}
	Let $\seq{X,\topol{U},k}$ be a CSC space with $\topol{U} = \seq{U_n : n \in \mathbb{N}}$.
	\begin{enumerate}
		\item For $Y \subseteq X$, let $\topol{U} \res Y = \seq{U_n \cap Y: n \in \mathbb{N}}$.
		\item A \emph{subspace} of $\seq{X,\topol{U},k}$ is a tuple $\seq{Y,\topol{U} \res Y, k}$ for some $Y \subseteq X$.
	\end{enumerate}
\end{definition}

\noindent Note that any subspace of a CSC space is itself a CSC space. Given a subspace $\seq{Y,\topol{U} \res Y, k}$ of a CSC space $\seq{X,\topol{U},k}$, we will typically simply say that \emph{$Y$ is a subspace of $X$}, and omit the notation $\topol{U} \res Y$. For clarity, we will sometimes say that a basic open set of $Y$ (i.e., an element of $\topol{U} \res Y$) is a \emph{basic open set of $Y$ in the subspace topology}.

Classically, a subspace is defined in terms of a restriction of the full topology, rather than just the basis. As a result, the definition of subspace may at first appear too strong. However, given a CSC space $\seq{X,\topol{U},k}$ and a set $Y \subseteq X$, the existence of the subspace $\seq{Y,\topol{U} \res Y, k}$ is easily provable in $\RCA_0$. Thus, we may restrict to subspaces in the sense of \Cref{def:subspace} without any loss of generality.

\subsection{Separation axioms}

We now define the first few degrees of Kolmogorov's classification of topological spaces, as they pertain to CSC spaces in our setting.

\begin{definition}
	Let $\seq{X,\topol{U},k}$ be a CSC space.
	\begin{enumerate}
		\item $X$ is \emph{$T_0$} if for all $x \neq y$ in $X$, there exists $U \in \topol{U}$ such that either $x \in U$ and $y \notin U$, or $y \in U$ and $x \notin U$.
		\item $X$ is \emph{$T_1$} if for all $x \neq y$ in $X$, there exist $U,V \in \topol{U}$ such that $x \in U \setminus V$ and $y \in V \setminus U$.
		\item $X$ is \emph{Hausdorff} (or $T_2$) if for all $x \neq y$ in $X$ there exist $U,V \in \topol{U}$ such that $x \in U$, $y \in V$, and $U \cap V = \emptyset$.
	\end{enumerate}	
\end{definition}

\noindent In $\RCA_0$, it is straightforward to prove the standard fact that every Hausdorff CSC space is $T_1$, and every $T_1$ CSC space is $T_0$.

\begin{proposition}\label{lem:T_1_finite_is_singleton}
	The following is provable in $\RCA_0$. Let $\seq{X,\topol{U},k}$ be a $T_1$ CSC space. Then for all $x \in X$, $\set{x} \in \topol{U}$ if and only if $x \in U$ for some finite $U \in \topol{U}$.
\end{proposition}

\begin{proof}
	The left-to-right implication is trivial. For the right-to-left, let $\topol{U} = \seq{U_n : n \in \mathbb{N}}$. Fix $x$ and suppose $x \in U$ for some finite $U \in \topol{U}$. Since $X$ is $T_1$, for every $y \in U$ different from $x$ there is an $n$ such that $x \in U_n$ and $y \notin U_n$, and $\RCA_0$ can form the set $F$ containing the least such $n$ for each $y$. Thus, $\bigcap_{n \in F} U_n = \set{x}$, and so $x \in U_{\overline{k}(x,F)} \subseteq \bigcap_{n \in F} U_n = \set{x}$, where $\overline{k}$ is the function from \Cref{lem:kStar}. Since $U_{\overline{k}(x,F)} \in \topol{U}$, we are done.
\end{proof}

\subsection{Indiscrete and discrete spaces}

\begin{definition}\label{def:top_disc_indisc}
	Let $\seq{X,\topol{U},k}$ be a CSC space.
	\begin{enumerate}
		\item $X$ is \emph{indiscrete} if $U \in \topol{U}$ if and only if $U$ is $\emptyset$ or $X$.
		\item $X$ is \emph{discrete} if $\set{x} \in \topol{U}$ for every $x \in X$.
	\end{enumerate}
\end{definition}

Because there is in general no effective way, given $x \in X$, to find the index of a $U \in \topol{U}$ equal to $\set{x}$, Dorais \cite[Definition 5.1]{Dorais-2011} introduces the following stronger notion.

\begin{definition}\label{def:top_disc_indisc}
	A CSC space $\seq{X,\topol{U},k}$ with $\topol{U} = \seq{U_n : n \in \mathbb{N}}$ is \emph{effectively discrete} if it is discrete and there exists a function $d : X \to \mathbb{N}$ such that $U_{d(x)} = \set{x}$ for all $x \in X$.
\end{definition}

The following proposition encapsulates the gap in complexity between discrete and effectively discrete spaces. We include it here for general interest because it would seem to be a basic result concerning CSC space. The equivalence of (1) and (2) below is implicit in work of Darais~\cite[Example 7.3]{Dorais-2011}.

\begin{proposition}\label{prop:ACA_disc_to_eff_disc}
	The following are equivalent over $\RCA_0$.
	\begin{enumerate}
		\item $\ACA_0$.
		\item The statement that every discrete CSC space is effectively discrete.
		\item The statement that every infinite discrete CSC space has an infinite effectively discrete subspace.
	\end{enumerate}
\end{proposition}

\begin{proof}
	The implication from (1) to (2) is straightforward, and (2) implies (3) trivially. To prove (3) implies (1), we argue in $\RCA_0$ and assume (3). To begin, we derive $\mathsf{B}\Sigma^0_2$, which we take in the guise of $(\forall m)\RT^1_m$ using Hirst's theorem. So fix $c : \mathbb{N} \to m$ for some $m \in \mathbb{N}$, and suppose towards a contradiction that $c$ has no infinite homogeneous set. For all $x,s \in \mathbb{N}$, define $V_{\seq{x,s}} = \set{x} \cup \set{z > s : c(z) = c(x)}$, and let $\seq{\mathbb{N},\topol{U},k}$ be the CSC space generated by $\seq{V_n : n \in \mathbb{N}}$. By assumption, $\set{z : c(z) = c(x)}$ is bounded by some $s$, meaning that $V_{\seq{x,s}} = \set{x}$. Thus, $\seq{\mathbb{N},\topol{U},k}$ is an infinite discrete space. Now let $Y \subseteq \mathbb{N}$ be an infinite effectively discrete subspace of $\seq{\mathbb{N},\topol{U},k}$, with witness $d : Y \to \mathbb{N}$. That is, if we write $\topol{U} = \seq{U_n : n \in \mathbb{N}}$ then $U_{d(x)} \cap Y = \set{x}$ for all $x \in Y$. We define a sequence $x_0 < \cdots < x_m$ in $Y$ as follows. Let $x_0 \in Y$ be arbitrary, and suppose we have defined $x_i$ for some $i < m$. By \Cref{prop:dorais_generate}~(2), we can decode $d(x_i)$ as a nonempty finite set $F$ such that $U_{d(x_i)} = \bigcap_{n \in F} V_n$. Now, each $n \in F$ is the form $\seq{z,s}$ for some $z$ and $s$, and since $x_i \in V_{\seq{z,s}}$ it follows that every $y > s$ with $c(y) = c(x_i)$ also belongs to $V_{\seq{z,s}}$. Choose the largest $s$ such that $\seq{z,s} \in F$ for some $z$ and let $x_{i+1}$ be the least element of $Y$ larger than $x_i$ and $s$. This completes the definition. By induction on $j \leq m$, $c(x_i) \neq c(y)$ for all $i < j$ and all $y \geq x_j$, so in particular, $c(x_i) \neq c(x_j)$. Hence, the map $x_i \mapsto c(x_i)$ is an injection of a finite set of size $m+1$ into $m$, which contradicts $\mathsf{I}\Sigma^0_1$. We conclude that $c$ has an infinite homogeneous set after all, and hence that $\mathsf{B}\Sigma^0_2$ holds.
	
	We now use (3) and $\mathsf{B}\Sigma^0_2$ to derive $\ACA_0$. To this end, it suffices to prove that the range of every one-to-one function $f : \mathbb{N} \to \mathbb{N}$ exists. Fix such an $f$. For $x,s \in \mathbb{N}$, let $\fres{f}{x}{s} = \ran(f \res s+1) \res x = \set{w < x : (\exists y \leq s)[f(y) = w]}$, and define
	\[
		V_{\seq{x,s}} = \set{x} \cup \set{z > s : \fres{f}{x}{s} \neq \fres{f}{x}{z}}.
	\]
	For each $w < x$ there is a $y$ such that either $w \notin \ran(f)$ or $f(y) = w$. Hence, by $\mathsf{B}\Sigma^0_2$ there is an $s$ such that for each $w < x$ there is such a $y$ below $s$. It follows that for this $s$ we have $\fres{f}{x}{s} = \ran(f) \res x$,
	and therefore $V_{\seq{x,s}} = \set{x}$. Consequently, the CSC space $\seq{\mathbb{N},\topol{U},k}$ generated by $\seq{V_n : n \in \mathbb{N}}$ is discrete. Apply (2) to get an infinite discrete subspace $Y \subseteq \mathbb{N}$ with witnessing function $d$, so that if $\topol{U} = \seq{U_n : n \in \mathbb{N}}$ then $U_{d(x)} \cap Y = \set{x}$ for all $x \in Y$. To complete the proof, we show for all $w$ that $\ran (f) \res w$ is uniformly $\Delta^0_1$ definable in $w$. Indeed, let $w$ be given and fix the least $x > w$ in $Y$. By \Cref{prop:dorais_generate}~(2), $d(x)$ codes a nonempty finite set $F$ such that $U_{d(x)} = \bigcap_{n \in F} V_n$. We claim that $\seq{x,s} \in F$ for some $s$, and that for the largest such $s$ we have $\ran(f) \res w = \fres{f}{w}{s}$. To see this, note that every element of $F$ is of the form $\seq{z,s}$ for some $z$ and $s$, and this has the property that if some $y \neq z$ belongs to $V_{\seq{z,s}}$ then so does every number larger than $y$. Since $V_n \subseteq U_{d(x)}$ for all $n \in F$, if all the $\seq{z,s} \in F$ satisfied $z \neq x$ we would have $[x,\infty) \subseteq U_{d(x)}$, hence $U_{d(x)} \cap Y$ could not be $\set{x}$ since $Y$ is infinite. Thus, $\seq{x,s} \in F$ for some $s$. If, for the largest such $s$, we had $y \in V_{\seq{x,s}}$ for some $y \neq x$, we would also have $[y,\infty) \subseteq V_{\seq{x,t}}$ for all $t \leq s$ since $V_{\seq{x,s}} \subseteq V_{\seq{x,t}}$. Hence, for this $y$, $[\max \set{x,y},\infty)$ would be a subset of $U_{d(x)}$, and so $U_{d(x)}$ would again not be a singleton. We conclude that $V_{\seq{x,s}} = \set{x}$, which means that $\fres{f}{x}{s} = \fres{f}{x}{z}$ for all $z > s$ and hence that $\fres{f}{x}{s} = \ran(f) \res x$. In particular, $(\fres{f}{x}{s}) \res w = \ran(f) \res w$. This proves the claim.
\end{proof}

\subsection{The initial segment and final segment topologies}

\begin{definition}\label{def:top_init_final}
	Let $\seq{X,\topol{U},k}$ be an infinite CSC space.
	\begin{enumerate}
		\item $X$ has the \emph{initial segment} topology if there is a bijection $h : \mathbb{N} \to X$ such that $U \in \topol{U}$ if and only if $U$ is $\emptyset$, $X$, or $\set{h(i) : i \leq j}$ for some $j \in \mathbb{N}$.
		\item $X$ has the \emph{final segment} topology if there is a bijection $h : \mathbb{N} \to X$ such that $U \in \topol{U}$ if and only if $U$ is $\emptyset$, $X$, or $\set{h(i) : i \geq j}$ for some $j \in \mathbb{N}$.
	\end{enumerate}
\end{definition}

\noindent We call the map $h$ in (1) and (2) a \emph{homeomorphism}, as usual.

Classically, one can define the initial segment and final segment topologies purely in terms of open sets, and without needing to explicitly reference a homeomorphism. From the point of view of reverse mathematics we might prefer such a definition because, while a homeomorphism is more convenient to work with, having an arithmetical definition rather than a $\Sigma^1_1$ one is less likely to inflate a principle's strength. To this end, we can define the following weaker forms of the above definitions.

\begin{definition}\label{def:weak_init_final}
	Let $\seq{X,\topol{U},k}$ be an infinite CSC space.
	\begin{enumerate}
		\item $X$ has the \emph{weak initial segment topology} if:
			\begin{enumerate}
			\item every $U \in \topol{U}$ is finite or equal to $X$;
			\item if $U,V \in \topol{U}$ then either $U \subseteq V$ or $V \subseteq U$;
			\item for each $s \in \mathbb{N}$, there is a finite $U \in \topol{U}$ such that $|U| = s$;
			\item each $x \in X$ belongs to some finite $U \in \topol{U}$.
		\end{enumerate}
		\item $X$ has the \emph{weak final segment topology} if:
			\begin{enumerate}
			\item every $U \in \topol{U}$ is cofinite in $X$ or equal to $\emptyset$;
			\item if $U,V \in \topol{U}$ then either $U \subseteq V$ or $V \subseteq U$;
			\item for each $s \in \mathbb{N}$, there is a nonempty $U \in \topol{U}$ such that $|X \setminus U| = s$;
			\item each $x \in X$ belongs to some $U \neq X$ in $\topol{U}$.
		\end{enumerate}
	\end{enumerate}	
\end{definition}

\begin{proposition}
	The following are equivalent over $\RCA_0$.
	\begin{enumerate}
		\item $\ACA_0$.
		\item Every CSC space with the weak initial segment topology has the initial segment topology.
		\item Every CSC space with the weak final segment topology has the final segment topology.
	\end{enumerate}
\end{proposition}

\begin{proof}
	To see that (1) implies (2), let $\seq{X,\topol{U},k}$ be a CSC space with the weak initial segment topology. We define $h : \mathbb{N} \to X$ as follows. By properties (1b) and (1c), $\ACA_0$ can find a sequence of sets $U_0 \subset U_1 \subset \cdots$ with $|U_i| = i+1$ for all $i \in \mathbb{N}$. By property (1b), each finite $U \in \topol{U}$ is equal to $U_i$ for some $i$, so by property (1d), $X = \bigcup_{i \in \mathbb{N}} U_i$. Define $h : \mathbb{N} \to X$ by letting $h(i)$ be the unique element of $U_i \setminus U_{j < i} U_j$. Thus, $h$ is bijective, and by induction, $U_j = \set{h(i) : i \leq j}$ for each $j$. Now fix any $U \in \topol{U}$ other than $\emptyset$ and $X$. By property (1a), $U$ is finite, hence equal to $U_j$ for some $j$, which means $U = \set{h(i) : i \leq j}$. Thus, $h$ is a homeomorphism witnessing that $X$ has the initial segment topology. A similar argument shows that (1) implies (3).
	
	We now show that (2) implies (1); the proof that (3) implies (1) is similar. So, assume (2), and fix a one-to-one $f : \mathbb{N} \to \mathbb{N}$ with the aim of showing that $\ran(f)$ exists. By bounded $\Sigma^0_1$ comprehension, $\ran(f)$ is unbounded. For each $x$, define
	\[
		V_{2x} = \set{w : w \leq x} \cup \set{y : f(y) < x}
	\]
	and
	\[
		V_{2x+1} = \set{w : w \leq x} \cup \set{y : f(y) \leq x}.
	\]
	A straightforward induction in $\RCA_0$ shows that $V_n \subseteq V_{n+1}$ and $|V_{n+1} \setminus V_n| \leq 1$ for all $n$. Since $|V_0| = 1$, it follows that $|V_n| \leq n+1$ for all $n$, and so in particular that each $V_n$ is finite. We claim that for each $s > 0$ there is an $n$ such that $|V_n| = s$. Indeed, since $\ran(f)$ is unbounded there exist $x_0 < \cdots < x_s \in \ran(f)$. Because $f$ is injective, this means $|V_{2x_s+1}| \geq |\set{y : (\exists i \leq s)[f(y) = x_i]}| > s$. Thus, $\RCA_0$ can prove that there is a least $n$ such that $|V_n| > s$. Obviously we must have $s > 1$ and so $n \geq 1$. By minimality of $n$, we have $|V_{n-1}| \leq s$, hence by our observation above $|V_{n-1}| = |V_n|-1$. It follows that $|V_{n-1}| = s$, which proves the claim. Now let $\seq{\mathbb{N},\topol{U},k}$ be the CSC space generated by $\seq{V_n : n \in \mathbb{N}}$. The preceding argument implies that this space has the weak initial segment topology. So apply (2) to get a homeomorphism $h : \mathbb{N} \to \mathbb{N}$, and fix any~$x$. Since $V_{2x}$ and $V_{2x+1}$ are finite, we can find the least $j_0$ and $j_1$ such that $h(j_0+1) \notin V_{2x}$ and $h(j_1+1) \notin V_{2x+1}$, which means that $V_{2x} = \set{h(i) : i \leq j_0}$ and $V_{2x+1} = \set{h(i) : i \leq j_1}$. By construction, $j_0 \leq j_1 \leq j_0+1$, and we have that $x \in \ran(f)$ if and only if $j_1 = j_0+1$.
\end{proof}

We will see in the next section that all our results concerning the initial segment and final segment topologies lie at the level of $\ACA_0$. However, because our interest is in finding \emph{subspaces} with these topologies, the following result assures us that this complexity does not owe to the strength of the preceding equivalences. Thus, we may safely use the (stronger) notions of initial segment and final segment topologies from \Cref{def:top_init_final}.

\begin{proposition}
	The following is provable in $\RCA_0$.
	\begin{enumerate}
		\item Every infinite CSC space $\seq{X,\topol{U},k}$ with the weak initial segment topology has an infinite subspace with the initial segment topology.
		\item Every infinite CSC space $\seq{X,\topol{U},k}$ with the weak final segment topology has an infinite subspace with the final segment topology.
	\end{enumerate}	
\end{proposition}

\begin{proof}
	We prove (1); the proof of (2) is analogous. Fix $\seq{X,\topol{U},k}$ satisfying (1a)--(1d) in \Cref{def:weak_init_final}. (In fact, we will not need to use (1d).) We define sequences $x_0,x_1,\ldots \in X$ and $U_{n_0},U_{n_1},\ldots \in \topol{U}$ such that $x_i \in U_{n_j}$ if and only if $i \leq j$. We claim that $Y = \set{x_0,x_1,\ldots}$ is a subspace of $X$ with the initial segment topology, as witnessed by the map $h : i \mapsto x_i$. Note that $U_{n_j} \cap Y = \set{x_i : i \leq j}$ for each $j$, so we wish to show that the basic open sets in $Y$ in addition to $\emptyset$ and $Y$ are precisely the sets $U_{n_j} \cap Y$ for some $j$. Since $U_{n_j} \in \topol{U}$, it is certainly true that $U_{n_j} \cap Y$ is basic open in $Y$. Conversely, consider any basic open set in $Y$ different from $\emptyset$ and $Y$. Then this equals $U \cap Y$ for some $U \in \topol{U}$ different from $\emptyset$ and $X$. For any $j$ such that $x_j \notin U$ it follows by property (1b) that $U \subseteq U_{n_j}$ and hence also that $U \subseteq U_{n_j} \setminus \set{x_j}$. In particular, $U \cap Y \subseteq (U_{n_j} \setminus \set{x_j}) \cap Y$. We must therefore have $x_0 \in U$ since $U \cap Y \neq \emptyset$ by assumption. Now since $Y \neq X$, it follows by property (1a) that $U$ is finite, and we can consequently fix the largest $j$ such that $x_i \in U$ for all $i \leq j$. We conclude that $U_{n_j} \cap Y \subseteq U \cap Y$. And since $x_{j+1} \notin U$, the preceding argument implies that $U \cap Y \subseteq (U_{n_{j+1}} \setminus \set{x_{j+1}}) \cap Y = U_{n_j} \cap Y$. Thus $U \cap Y = U_{n_j} \cap Y$, desired.
	
	It remains to show how to construct the sequences $x_0,x_1,\ldots$ and $U_{n_0},U_{n_1},\ldots$, which we do by primitive recursion. Let $x_0$ be any element of $X$, and $U_{n_0}$ any element of $\topol{U}$ containing $x_0$. Now fix $j$, and suppose we have defined $x_i$ and $U_{n_i}$ for all $i \leq j$. By property (1b), and the fact that $x_j \in U_{n_j} \setminus U_{n_i}$ for all $i < j$ by inductive hypothesis, we must have $U_{n_i} \subseteq U_{n_j}$ for all $i \leq j$. Let $s = |U_j|$. Using property (1c) let $U \in \topol{U}$ be a finite set with at least $s+1$ many elements. Since $\mathsf{I}\Sigma^0_1$ proves that there is no injection from a larger finite set into a strictly smaller one, we cannot have $U \subseteq U_{n_j}$. Hence, $U_{n_j} \subseteq U$ by property (1b), and this containment must be strict. Fixing any $x \in U \setminus U_{n_j}$ thus means that $x_i \in U$ and $x \notin U_i$ for all $i \leq j$. Now by property (1a), being a finite element of $\topol{U}$ is definable by a $\Sigma^0_1$ formula (with parameters). Thus, the least $U$ and $x \in U$ satisfying this conclusion are definably by a $\Sigma^0_1$ formula, and we let $U_{n_{j+1}} = U$ and $x_{j+1} = x$. This completes the recursive definition and completes the proof.
\end{proof}

One additional remark about spaces with the initial segment and final segment topologies is that $\RCA_0$ understands the standard fact that, in such spaces, every open set is a basic open set.

\begin{proposition}\label{prop:O_init_in_U}
	The following is provable in $\RCA_0$. Let $\seq{X,\topol{U},k}$ be a CSC space that has the initial segment or final segment topology. If $O$ is any open code in $X$, then $O_{\topol{U}}$ exists and belongs to $\topol{U}$.
\end{proposition}

\begin{proof}
	Write $\topol{U} = \seq{U_n : n \in \mathbb{N}}$. Suppose $X$ has the initial segment topology; the case where it has the final segment topology is symmetric. Let $h : \mathbb{N} \to X$ be the witnessing homeomorphism, and consider any open code $O$. If $O_{\topol{U}} = X$ then it clearly belongs to $\topol{U}$, so suppose otherwise. Then every $U_n$ for $n \in \ran(O)$ is either $\emptyset$ or has the form $\set{h(i) : i \leq j_n}$ for some $j_n$, and the numbers $j_n$ must be bounded else $O_{\topol{U}}$ would again be $X$. In $\RCA_0$, we can thus fix the least $m$ such that $j_n \leq j_m$ for all $n \in \ran(O)$ for which $j_n$ is defined. Then $O_{\topol{U}} = U_m \in \topol{U}$.
\end{proof}

\subsection{The cofinite topology}

Finally, we discuss formalizing the cofinal topology.

\begin{definition}\label{def:top_cofinite}
	Let $\seq{X,\topol{U},k}$ be an infinite CSC space. $X$ has the \emph{cofinite topology} if $U \in \topol{U}$ if and only if $U = \emptyset$ or $U \subseteq^* X$.
\end{definition}

We could reasonably also define $X$ to have the cofinite topology if every $U \in \topol{U}$ is either $\emptyset$, $X$, or a cofinite subset of $X$, and if for every $Y \subseteq^* X$ there is an open code $O$ in $X$ with $O_{\topol{U}} = Y$. The following lemma shows that for studying topological aspects in reverse mathematics, we do not lose anything by insisting on the stronger definition above.

\begin{proposition}
	The following is provable in $\RCA_0$. Let $\seq{X,\topol{U},k}$ be a CSC space such that every $U \in \topol{U}$ is either $\emptyset$, $X$, or a cofinite subset of $X$. If $O$ is any open code in $X$, then $O_{\topol{U}}$ exists and is either empty or a cofinite subset of $X$.
\end{proposition}

\begin{proof}
	Write $\topol{U} = \seq{U_n : n \in \mathbb{N}}$.  Suppose $\ran(O)$ contains an $n$ with $U_n \neq \emptyset$, since otherwise $O_{\topol{U}} = \emptyset$. Then $U_n \subseteq^* X$, so we can fix a $b$ such that every $x \geq b$ in $X$ belongs to $U_n$. Using bounded $\Sigma^0_1$ comprehension, we can form $F = \set{x < b : (\exists n \in \ran(O))[x \in U_n]}$, and now we see that $x \in O_{\topol{U}}$ if and only if $x \in F$ or $x \geq b$.
\end{proof}

\noindent It follows that, in $\RCA_0$ we can modify $\seq{X,\topol{U},k}$ as above by adding to $\topol{U}$ all cofinite subsets of $X$, and in this way obtain a new CSC space that has the same open sets and has the cofinite topology in the sense of \Cref{def:top_cofinite}.

We conclude this section with the following proposition, which states that $\RCA_0$ understands Ginsburg and Sands's remark about why the five topologies mentioned in their theorem are ``minimal''.

\begin{proposition}\label{t:all_preserved}
	The following is provable in $\RCA_0$. Let $\seq{X,\topol{U},k}$ be an infinite CSC space, and let $Y$ be an infinite subspace of $X$.
	\begin{enumerate}
		\item If $X$ is $T_0$; $T_1$; Hausdorff; indiscrete; discrete; has the initial segment topology; has the final segment topology; or has the cofinite topology, then the same is true of $Y$.
		\item If $X$ is indiscrete; discrete; has the initial segment topology; has the final segment topology; or has the cofinite topology, then it does not satisfy any other property on this list.
	\end{enumerate}
\end{proposition}

\begin{proof}
	Part (1) is immediate from the definitions.	Part (2) is obvious if $X$ is indiscrete or discrete. It is also clear that if $X$ has the final segment topology then it cannot also have the cofinite topology. It thus suffices to show that if $X$ has the initial segment topology, say with witnessing homeomorphism $h$, then it cannot have the final segment or cofinal topologies. We claim that every $U \in \topol{U}$ other than $X$ is bounded, which yields the result. So fix some such $U$, and assume also that $U \neq \emptyset$. Then $U = \set{h(i) : i \leq j}$ for some $j \in \mathbb{N}$. If this set were unbounded, we could construct $x_0 < \ldots < x_j$ in $\set{h(i) : i \leq j}$, whence the map $i \mapsto h^{-1}(x_i)$ would be an injection from $j+1$ into $j$, and this contradicts $\I\Sigma^0_1$ (see \cite{DM-2022}, Proposition 6.2.7). Thus, $U$ must be bounded, which proves the claim.
\end{proof}

\section{General spaces}\label{sec:general}

The main principle we want to understand is the following.

\begin{definition}
	$\GS$ is the statement that every infinite CSC space has an infinite subspace which is indiscrete; has the initial segment topology; has the final segment topology; is discrete; or has the cofinite topology.
\end{definition}

\noindent Recall the three-part breakdown of the Ginsburg--Sands theorem, discussed at the end of \Cref{sec:GS_proof}. The first of these involved the construction, on a given topological space, of a partial order induced by the closure relation on points. We begin by formalizing the closure relation in $\RCA_0$.

\begin{definition}
	Let $\seq{X,\topol{U},k}$ be an infinite CSC space. The \emph{closure relation} on $X$ is a subset $\cl_X \subseteq X^2$ such that $\seq{y,x} \in \cl_X$ if and only if every $U \in \mathcal{U}$ that contains $y$ also contains $x$.
\end{definition}

\noindent We will write $y \in \cl_X(x)$ if $\seq{y,x} \in \cl_X$, and usually omit the subscript for ease of notation.

Unsurprisingly, $\RCA_0$ cannot in general prove that the closure relation exists.

\begin{proposition}\label{prop:closure_ACA}
	The following are equivalent over $\RCA_0$.
	\begin{enumerate}
		\item $\ACA_0$.
		\item The statement that for every CSC space, the closure relation exists.
	\end{enumerate}
\end{proposition}

\begin{proof}
	The implication from (1) to (2) is clear, because given a CSC space $\seq{X,\topol{U},k}$, the definition of the closure relation is arithmetical using $\topol{U}$ as a set parameter. For the reverse, we argue in $\RCA_0$ and assume (2). We prove that the range of a given one-to-one function $f : \mathbb{N} \to \mathbb{N}$ exists. For $x,s \in \mathbb{N}$, we define
	\[
    V_{\seq{x,s}}=\begin{cases}
        \{2x,2x+1\} &\text{if $(\forall y\leq s)[f(y)\neq x]$}, \\
        \{2x\} &\text{if $(\exists y\leq s)[f(y)=x]$}.
    \end{cases}
    \]
    Let $\seq{\mathbb{N},\topol{U},k}$ be the CSC space generated by $\seq{V_n : n \in \mathbb{N}}$. Applying (2), let $\cl$ be the closure relation on this space.
    
    Now fix $x \in \mathbb{N}$; we claim that $x \in \ran(f)$ if and only if $2x \notin \cl(2x+1)$. If $x \in \ran(f)$, say with $f(y) = x$, then $V_{\seq{x,y}}$ is a basic open set containing $2x$ but not $2x+1$. So $2x \notin \cl(2n+1)$. Conversely, suppose $2x \notin \cl(2x+1)$. Then there exists $U \in \topol{U}$ that contains $2x$ but not $2x+1$. By \Cref{prop:dorais_generate}~(2), we can fix $z,s \in \mathbb{N}$ such that $V_{\seq{z,s}}$ contains $2x$ but not $2x+1$. Since $V_{\seq{z,s}} \subseteq \set{2z,2z+1}$, it must be that $z = x$ and hence that $(\exists y \leq s)[f(y) = x]$. So $x \in \ran(f)$.
\end{proof}

As a result of the proposition, we can consider adding the closure relation to the signature of a CSC space, and looking instead at the strength of the Ginsburg--Sands theorem for these enriched spaces.

\begin{definition}
	$\GS^{\cl}$ is the restriction of $\GS$ to CSC spaces for which the closure relation exists.
\end{definition}

\noindent We can thus think of the proof-theoretic  strength of $\GS^{\cl}$ as corresponding to the proof of the Ginsburg--Sands theorem with the complexity of part (1) removed.

We next define the following further restrictions.

\begin{definition}
	\
	\begin{itemize}
		\item $\wGS$ is the statement that every infinite CSC space has an infinite subspace which is indiscrete; has the initial segment topology; has the final segment topology; or is $T_1$.
		\item $\wGS^{\cl}$ is the restriction of $\wGS^{\cl}$ to CSC spaces for which the closure relation exists.
		\item $\GST_1$ is the restriction of $\GS$ to $T_1$ CSC spaces.
	\end{itemize}	
\end{definition}

\noindent By analogy to the above, the strength of $\wGS$ corresponds to the proof of Ginsburg--Sands with the complexity of part (3) removed, and $\wGS^{\cl}$ corresponds to the proof with the complexities of both parts (1) and (3) removed. $\GST_1$ carries the complexity just of part (3). 

The obvious implications over $\RCA_0$ are that $\GS \to \GS^{\cl} \to \wGS^{\cl}$, that $\GS \to \wGS \to \wGS^{\cl}$, and that $\GS \to \GST_1$. We now analyze the missing implications, which confirm that our breakdown of Ginsburg and Sands's proof was the correct one. (It is noteworthy that, in spite of \Cref{prop:closure_ACA}, the proof of the following theorem does not make direct use of the closure relation.)

\begin{theorem}\label{thm:ACA_to_GS}
	The following are equivalent over $\RCA_0$.
	\begin{enumerate}
		\item $\ACA_0$.
		\item $\GS$.
		\item $\wGS$. 
	\end{enumerate}	
\end{theorem}

\begin{proof}
	For the implication from (1) to (2), we note that the proof of the Ginsburg--Sands theorem given in \Cref{sec:GS_proof} readily formalizes in $\ACA_0$ when applied to CSC spaces. The implication from (2) to (3) is trivial. It thus remains to show that (3) implies (1). To this end, we argue in $\RCA_0$ and assume $\wGS$. Note that $\wGS$ trivially implies $\wGS^{\cl}$, which we in turn prove is equivalent to $\GST_1$ in \Cref{T:GST1_equiv_WGSclos} below. Moreover, in \Cref{cor:GST1_to_RT22} we show that $\GST_1$ implies $\RT^2_2$, and hence $\mathsf{B}\Sigma^0_2$. Therefore, we may avail ourselves of $\mathsf{B}\Sigma^0_2$ in our argument.
	
	We prove that the range of a given one-to-one function $f : \mathbb{N} \to \mathbb{N}$ exists.
	Given $x,s \in \mathbb{N}$, using the notation from the proof of \Cref{prop:ACA_disc_to_eff_disc}, define
	\[
		V_{\seq{x,s}} = \set{x} \cup \set{w < x : \fres{f}{w}{x} = \fres{f}{w}{x+s} },
	\]
	and let $\seq{\mathbb{N},\topol{U},k}$ be the CSC space generated by $\seq{V_n : n \in \mathbb{N}}$.
	
	We claim that if a number $x$ is sufficiently large with respect to a number $w$ then $w$ belongs to every basic open set containing $x$. To see this, fix $w$ and choose $x > w$ so that $\ran(f) \res w = \fres{f}{w}{x}$, which exists by $\mathsf{B}\Sigma^0_2$. Fix $U \in \topol{U}$ containing $x$. Then $U$ is a finite intersection of sets of the form $V_{\seq{z,s}}$ for some $z$ and $s$, and since no element of $V_{\seq{z,s}}$ can be larger than $z$, we necessarily have $z \geq x$. Consider such a $V_{\seq{z,s}}$. If $x = z$, then $w \in V_{\seq{z,s}} = V_{\seq{x,s}}$ since $\fres{f}{w}{t} = \ran(f) \res w$ for all $t \geq x$ by choice of $x$. If $x < z$ then we must have $\fres{f}{x}{z} = \fres{f}{x}{z+s}$, and hence also $\fres{f}{w}{z} = \fres{f}{w}{z+s}$ since $w < x$, so $w \in V_{\seq{z,s}}$. This proves the claim.
	
	Apply $\wGS$ to $X$ to obtain an infinite subspace $Y$. By the preceding claim, we immediately know that $Y$ cannot be $T_1$ nor have the final segment topology. And since, for $w < x$, we have that $V_{\seq{w,0}}$ contains $w$ but not $x$, we know that $Y$ cannot be indiscrete. This means that $Y$ has the initial segment topology. There is thus a bijection $h : \mathbb{N} \to Y$ such that the basic open sets of $Y$ in the subspace topology are precisely $\set{h(i) : i \leq j}$ for some $j \in \mathbb{N}$. Let $Y = \set{y_0 < y_1 < \cdots}$; we claim that $h(n) = y_n$ for all $n$. To see this, first suppose that $h(n) > h(m)$ for some $n < m$, and fix $j$ such that $V_{\seq{h(m),0}} \cap Y = \set{h(i) : i \leq j}$. Then $V_{\seq{h(m),0}} \cap Y$ contains $h(n)$, contradicting the fact that this intersection only contains $h(m)$ and otherwise numbers smaller than $y_m$. It follows that $h(n) < h(m)$ for all $n < m$, whence an easy induction using the fact that $h$ is bijective shows that $h(n) = y_n$. This proves the claim. In particular, every basic open set in $X$ is closed downward in $Y$ under $\leq$.
	
	We now use $Y$ to prove the existence of $\ran(f)$. Fix $w \in \mathbb{N}$. Since $Y$ is infinite, we can find $z > x > w$ with $x,z \in Y$. We must then have that $\ran(f) \res x = \fres{f}{x}{z}$, and hence also that $\ran(f) \res w = \fres{f}{w}{z}$. If not, then there would be an $s$ such that $\fres{f}{x}{z} \neq \fres{f}{x}{z + s}$. But then $V_{\seq{z,s}}$ would be a basic open set in $X$ containing $z$ but not $x < z$, even though both belong to $Y$. This completes the proof.
\end{proof}

\begin{theorem}\label{T:CAC_equiv_GSclos}
	The following are equivalent over $\RCA_0$.
	\begin{enumerate}
		\item $\CAC$.
		\item $\wGS^{\cl}$.
	\end{enumerate}
\end{theorem}

\begin{proof}
	That (1) implies (2) follows readily from the proof of the Ginsburg--Sands theorem given in \Cref{sec:GS_proof}. More precisely, given a CSC space with its closure relation, we can construct the equivalence relation $\sim$ and the partial ordering $\leq_X$ defined in the proof, and argue as before to obtain an infinite subspace which is either indiscrete, has the initial segment or final segment topologies, or is $T_1$.
	
	For the reverse implication, assume $\wGS^{\cl}$ and let $(P,\leq_P)$ be any infinite partial order. Assume this partial ordering has no infinite antichain; we will construct an infinite chain. (In fact, our chain will either be an infinite ascending sequence under $\leq_P$, or an infinite descending sequence.)
	
	For each $p \in P$, define $V_p = \set{q \in P: p \leq_P q}$ and let $\seq{P,\topol{U},k}$ be the CSC space generated by $\seq{V_p : p \in P}$. We claim that, for $p,q \in P$, $p$ is in the closure of $q$ in this space if and only if $p \leq_P q$. Indeed, if $p$ is in the closure of $q$ then in particular $q$ must belong to $V_p$. So $p \leq_P q$. Suppose $p$ is not in the closure of $q$. Then we can fix a $U \in \topol{U}$ containing $p$ but not $q$. By \Cref{prop:dorais_generate}~(2), there is an $r \in P$ such that $V_r$ contains $p$ but not $q$, meaning $r \leq_P p$ but $r \nleq_P q$. So $p \nleq_P q$. This proves the claim, and implies that the closure relation exists. Call it $\cl$.
		
	Now apply $\wGS^{\cl}$ to obtain an infinite subspace $Y$ of $\seq{P,\topol{U},k}$. By assumption, $Y$ is not an infinite antichain under $\leq_P$. Thus we can fix some $p <_P q$ in $Y$. It follows that $p \notin U_q$, so $Y$ cannot be indiscrete as a subspace. It also cannot be $T_1$, because if it were we would have $p \notin \cl(q)$ and hence $p \nleq_P q$. Thus, $Y$ either has the initial segment or final segment topology. Consider first the case that it is the former. Then there is a bijection $h : \mathbb{N} \to Y$ such that the basic open sets of $Y$ in the subspace topology are precisely $\set{h(i) : i \leq j}$ for some $j \in \mathbb{N}$. We claim that $h(j) <_P h(i)$ for all $i < j$, which proves that $Y$ is a descending sequence under $\leq_P$. Suppose not. Then there exist $p,q \in Y$ such that $h^{-1}(q) < h^{-1}(p)$ and $p \nleq_P q$. Then $p \notin \cl(q)$, so there exists $U \in \topol{U}$ containing $p$ but not $q$. Fix $j$ so that $U \cap Y = \set{h(i) : i 
	\leq j}$. Then necessarily $h^{-1}(p) \leq j$, but this means that also $h^{-1}(q) \leq j$ and hence that $q \in U$, a contradiction. The case where $Y$ has the final segment topology is symmetric, and yields an ascending sequence under~$\leq_P$.
\end{proof}

Finally, we show that $\GST_1$ and $\GS^{\cl}$ are equivalent. In a sense, the following says that when the complexity of the closure operator is removed, the combinatorics underlying the full Ginsburg--Sands are the same as those underlying just the $T_1$ case.

\begin{theorem}\label{T:GST1_equiv_WGSclos}
	The following are equivalent over $\RCA_0$.
	\begin{enumerate}
		\item $\GST_1$.
		\item $\mathsf{GS}^{\cl}$.
	\end{enumerate}
\end{theorem}

\begin{proof}
	We argue in $\RCA_0$. To prove that (1) implies (2), assume $\GST_1$ and let a CSC space $\seq{X,\topol{U},k}$ with closure relation $\cl$ be given. By \Cref{cor:GST1_to_RT22} below, $\GST_1$ implies $\RT^2_2$ and therefore $\CAC$. Hence, by \Cref{T:CAC_equiv_GSclos}, $\GST_1$ implies $\wGS^{\cl}$. So let $Y \subseteq X$ be an infinite subspace obtained by applying $\wGS^{\cl}$ to $X$. If $Y$ is indiscrete, or homeomorphic to the initial segment or final segment topologies, then $Y$ is a solution to $X$ as an instance of $\mathsf{GS}^{\cl}$. If $Y$ is $T_1$, then we can apply $\GST_1$ to $Y$ to obtain a subspace $Z$ of $Y$ (and so, of $X$) that is discrete or has the cofinite topology.
	
	To prove that (2) implies (1), assume $\mathsf{GS}^{\cl}$ and let $\seq{X,\topol{U},k}$ be a $T_1$ CSC space. The closure relation on $X$ is then trivial, since $y \in \cl(x)$ if and only if $x = y$. We can thus apply $\mathsf{GS}^{\cl}$ to $X$ to obtain an infinite subspace $Y \subseteq X$, and since $X$ is $T_1$ this subspace cannot be indiscrete, nor be homeomorphic to either the initial segment or final segment topologies. It follows that $Y$ is discrete or has the cofinite topology, as needed.
\end{proof}

\noindent This theorem, unlike the two preceding it, does not directly establish any relationships to principles previously studied in the literature. We will do so in \Cref{sec:T1}, however, which is dedicated to a detailed study of $\GST_1$.

\section{Hausdorff spaces}\label{sec:Haus}

In this section, we turn our attention momentarily to Hausdorff spaces. Since all such spaces are $T_1$, they do not feature separately in our breakdown of the Ginsburg--Sands theorem. Moreover, since no Hausdorff space can have the cofinite topology, the subspace produced by applying $\GS$ to such a space must be discrete. We show that this can always be done computably. We begin with a technical lemma, and then give the result.

\begin{lemma}\label{lem:t2LeastPair}
	The following is provable in $\mathsf{RCA}_0$. Let $(X, \topol{U}, k)$ be a Hausdorff CSC space with $\topol{U} = \seq{U_n : n \in \mathbb{N}}$.
	\begin{enumerate}
		\item There exists a function $c: X \times X \times \mathbb{N} \to \mathbb{N}$ such that if $x, y \in X$ are distinct and $s \in \mathbb{N}$, then $c(x, y, s)$ is the least pair $\seq{m,n}$ such that $x \in U_m \setminus U_n$, $y \in U_n \setminus U_m$, and $(U_m \res s) \cap (U_n \res s) = \emptyset$.
		\item For all sufficiently large $s$, $c(x,y,s)$ equals the least pair $\seq{m,n}$ such that $x \in U_m$, $y \in U_n$, and $U_m \cap U_n = \emptyset$.
	\end{enumerate}
\end{lemma}

\begin{proof}
	Since $c$ is $\Delta^0_1$ definable, to prove (1) it suffices to prove that it is total. Fix $x \neq y$ in $X$ and $s \in \mathbb{N}$. Since $X$ is Hausdorff, there exist $U_m, U_n$ such that $x \in U_m$, $y \in U_n$, and $U_m \cap U_n = \emptyset$. In particular, the set of all pairs $\seq{m,n}$ such that $x \in U_m \setminus U_n$, $y \in U_n \setminus U_m$, and $(U_m \res s) \cap (U_n \res s) = \emptyset$ is nonempty. Now $\mathsf{I}\Sigma^0_1$ suffices to find the least element $\seq{m,n}$ of this set, and we have $c(x,y,s) = \seq{m,n}$.
	
	To prove (2), we note that the least $\seq{m,n}$ such that $x \in U_m$, $y \in U_n$, and $U_m \cap U_n = \emptyset$ exists by $\mathsf{I}\Pi^0_1$. Thus, for each $\seq{i,j} < \seq{m,n}$ there exists a $t$ such that if $x \in U_{i} \setminus U_{j}$ and $y \in U_{j} \setminus U_{i}$ then $t \in U_{i} \cap U_{j}$. By $\mathsf{B}\Sigma^0_0$ in $\RCA_0$ we can find an $s_0$ such that for each $\seq{i,j} < \seq{m,n}$ there exists such a $t$ below $s_0$. Then $c(x,y,s) = \seq{m,n}$ for all $s \geq s_0$.
\end{proof}

\begin{theorem}\label{thm:t2hasdiscrete}
	The following is provable in $\RCA_0$. Every infinite Hausdorff CSC space has an infinite discrete subspace.
\end{theorem}

\begin{proof}
	We argue in $\RCA_0$. Let $(X, \topol{U}, k)$ be an infinite Hausdorff CSC space with $\topol{U} = \seq{U_n : n \in \mathbb{N}}$. If $X$ has no limit point, then $X$ is already discrete. So, suppose $X$ has a limit point, $p$. We define a sequence $x_0 < x_1 < \cdots \in X$ by primitive recursion such that $x_i \neq p$ for all $i$. Let $x_0$ be any element of $X$ different from $p$. Now fix $s$, and suppose we have defined $x_i$ for all $i \leq s$. For each $i \leq s$, let $\seq{m_{i,s},n_{i,s}} = c(p,x_i,s)$ and let $m_s = \overline{k}(p,\set{m_{i,s}: i \leq s})$, where $\overline{k}$ and $c$ are as in \Cref{lem:t2LeastPair,lem:kStar}. In particular, $p \in U_{m_{i,s}}$ for all $i \leq s$, so $p \in U_{m_s} \subseteq \bigcap_{i \leq s} U_{m_{i,s}}$. Since $p$ is a limit point, $U_{m_s}$ is infinite, and we can let $x_{s+1}$ be any element of this set larger than $x_s$ and different from $p$.
	
	Let $Y = \set{x_0,x_1,\ldots}$. We claim that $Y$ has the discrete topology. To see this, fix $i$, and let $\seq{m,n}$ be least such that $p \in U_m$, $x_i \in U_n$, and $U_m \cap U_n = \emptyset$. By \Cref{lem:t2LeastPair}~(2), there is an $s_0 > i$ such that $c(p,x_i,s) = \seq{m,n}$ for all $s \geq s_0$. For all $s \geq s_0$, the point $x_s$ is chosen in $U_{m,s} \subseteq U_{m_{i,s}} = U_m$, which means that $U_n \cap Y \subseteq \set{x_j : j < s}$. Thus, $U_n \cap Y$ is a finite basic open set in $Y$ containing $x_i$. Since $Y$ is Hausdorff by \Cref{t:all_preserved}, and so in particular $T_1$, it follows by \Cref{lem:T_1_finite_is_singleton} that $\set{x_i}$ is open in $Y$.
\end{proof}

In light of the discussion in \Cref{sec:CSC} concerning Hausdorff spaces versus effectively Hausdorff spaces, and discrete spaces versus effectively discrete spaces, a natural question is whether \Cref{thm:t2hasdiscrete} also holds for these stronger presentations. We show that the answer is yes for Hausdorff spaces that have a limit point.

\begin{theorem}
	The following is provable in $\RCA_0$. Every infinite effectively Hausdorff CSC space which is not discrete has an infinite effectively discrete subspace.
\end{theorem}

\begin{proof}
	Let $(X, \topol{U}, k)$ be an infinite effectively Hausdorff CSC space which is not discrete. Then $X$ has a limit point, $p$. This means that for every $U \in \topol{U}$ containig $p$, there exist infinitely many $x \in U$ different from $p$. Let $\topol{U} = \seq{U_n : n \in \mathbb{N}}$, and let $e : X \times X \to \mathbb{N}$ be such that $e(x,y) = \seq{m,n}$ for all distinct $x,y \in X$, where $x \in U_m$, $y \in U_n$, and $U_m \cap U_n = \emptyset$. We construct sequences $x_0 < x_1 < \cdots \in X$, $n_0,n_1,\ldots \in \mathbb{N}$, and $m_0,m_1,\ldots \in \mathbb{N}$ such that $U_{m_0} \supseteq U_{m_1} \supseteq \cdots$, and for all $i$ and $j$: $x_j \in U_{n_i}$ if and only if $i = j$; $p \in U_{m_i}$; and $U_{m_i} \cap U_{n_i} = \emptyset$. Then $Y = \set{x_i : i \in \mathbb{N}}$ is an effectively discrete subspace of $X$ since $U_{n_i} \cap Y = \set{x_i}$ for all $i$.
	
	Let $x_0$ be any element of $X$ different from $p$, and let $n_0$ and $m_0$ be such that $e(p,x_0) = \seq{m_0,n_0}$. Now fix $s$, and suppose we have defined $x_i$, $n_i$, and $m_i$ for all $i \leq s$. Since $p \in \bigcap_{i \leq s} U_{m_i}$ we have that $p \in U_{\overline{k}(p,\set{m_i : i \leq s})} \subseteq \bigcap_{i \leq s} U_{m_i}$. And since $p$ is a limit point, we can let $x_{s+1}$ be the least element of $U_{\overline{k}(p,\set{m_i : i \leq s})}$ larger than $x_s$ and different from $p$. Say $e(p,x_{s+1}) = \seq{m,n}$. Then $p \in \bigcap_{i \leq s} U_{m_i} \cap U_m$ and $x_{s+1} \in U_{m_s} \cap U_n$, so $p \in U_{\overline{k}(p,\set{m_i : i \leq s} \cup \set{m})} \subseteq \bigcap_{i \leq s} U_{m_i} \cap U_m$, and $x_{s+1} \in U_{\overline{k}(x_i,\set{m_s,n})} \subseteq U_{m_s} \cap U_n$. Let $m_{s+1} = \overline{k}(p,\set{m_i : i \leq s} \cup \set{m})$ and $n_{s+1} = \overline{k}(x_{s+1},\set{m_s,n})$. 
	
	So, $x_{s+1} \in U_{n_{s+1}}$, $p \in U_{m_{s+1}}$, $U_{m_{s+1}} \subseteq U_{m_s}$, and $U_{m_{s+1}} \cap U_{n_{s+1}} = \emptyset$. Fix any $i \leq s$. We have $x_{s+1} \in U_{n_{s+1}} \subseteq U_{m_s}$, while $x_i \in U_{n_i}$ which is disjoint from $U_{m_i}$ and hence from $U_{m_s} \subseteq U_{m_i}$. Thus, $x_i \notin U_{n_{s+1}}$ and $x_{s+1} \notin U_{n_i}$, as needed.
\end{proof}

By contrast, and perhaps surprisingly, the previous result cannot be extended to effectively Hausdorff spaces in general.

\begin{theorem}\label{thm:Haus_to_disc_no_RCA}
	There exists an infinite computable effectively Hausdorff CSC which is discrete but has no infinite computable effectively discrete subspace.
\end{theorem}

\begin{proof}
	We construct a computable sequence $\seq{V_n : n \in \omega}$ of subsets of $\omega$ and then generate a computable CSC space $\seq{\omega,\topol{U},k}$ using \Cref{prop:dorais_generate}. Initially, we enumerate $x$ into $V_{2\seq{x,y}}$ and $V_{2\seq{x,y}+1}$ for all $y$. We then enumerate more elements into the $V_n$ in stages, maintaining throughout that $V_{2\seq{x,y}} \cap V_{2\seq{y,x}} = \emptyset$ for all $x \neq y$. By \Cref{prop:dorais_generate}~(2), we can compute the indices $m$ and $n$ of $V_{2\seq{x,y}}$ and $V_{2\seq{y,x}}$ in $\topol{U}$ uniformly computably from $x$ and $y$. Hence, the function $e : X \times X \to \mathbb{N}$ mapping $\seq{x,y}$ to $\seq{m,n}$ is computable. Since $x \in V_{2\seq{x,y}}$ and $y \in V_{2\seq{y,x}}$, this $e$ will witness that $\seq{\omega,\topol{U},k}$ is effectively Hausdorff.
	
	We aim to additionally satisfy the following requirements for all $x,e,u \in \omega$:
	\begin{eqnarray*}
		\mathcal{D}_x & : & \parbox[t]{.85\textwidth}{there exists an $n \in \omega$ such that $V_n = \set{x}$;}\\
		\mathcal{R}_{e,u} & : & \parbox[t]{.85\textwidth}{if $\Phi_e$ is the characteristic function of an infinite set $Y$ and $\Phi_u$ is total on $Y$, then there exists $x \in Y$ and $n_0,\ldots,n_r \in \omega$ such that $\Phi_u(x) \converges = \seq{n_0,\ldots,n_r}$ but $\bigcap_{i \leq r} V_{n_i} \neq \set{x}$.}
	\end{eqnarray*}
	We call these the \emph{$\mathcal{D}$-requirements} and \emph{$\mathcal{R}$-requirements} for short. The $\mathcal{D}$-requirements thus ensure that our CSC space is discrete, while the $\mathcal{R}$-requirements ensure that it has no infinite computable effectively discrete subspace. At each stage, at most one requirement will \emph{act}. If a $\mathcal{D}$-requirement acts it \emph{claims} an index $n \in \omega$, and if an $\mathcal{R}$-requirement acts it \emph{claims} finitely many indices $n_0,\ldots,n_r \in \omega$. A requirement may have its claims cancelled at a later stage, and then make new claims later still, as we describe below.	
	
	We proceed to the construction. We assign a priority ordering to the requirements in order type $\omega$, as usual. For notational convenience, for $n = 2\seq{x,y}$ we let $\overline{n}$ denote $2\seq{y,x}$. Suppose we are at stage $s \in \omega$. Say $\mathcal{D}_x$ \emph{requires attention} at stage $s$ if it does not currently have a claim. Say $\mathcal{R}_{e,u}$ \emph{requires attention} at stage $s$ if it does not currently have any claims, and if there exist $x < s$ and $\seq{n_0,\ldots,n_r} < s$ as follows:
	\begin{itemize}
		\item $\Phi_e(x)[s] \converges = 1$;
		\item $\Phi_u(x)[s] \converges = \seq{n_0,\ldots,n_r}$;
		\item no $n_i$ is claimed by any stronger $\mathcal{D}$ requirement;
		\item no even $\overline{n_i}$ with $n_i \neq \overline{n_i}$ is claimed by any stronger $\mathcal{R}$-requirement;
		\item no even $n_i \neq n_j$ satisfy $n_i = \overline{n_j}$.
	\end{itemize}
	If it exists, fix the strongest requirement with index smaller than $s$ that requires attention at stage $s$.
	
	If this is $\mathcal{D}_x$ for some $x$, choose the least $y$ so that $2\seq{x,y}+1$ is larger than all indices claimed by any requirement at any point in the construction, including claims that may since have been cancelled. Since at most one requirement claims at most finitely many numbers at any stage, $y$ exists. Now say that $\mathcal{D}_x$ \emph{claims} $2\seq{x,y}+1$. If the requirement is instead $\mathcal{R}_{e,u}$ for some $e$ and $u$, fix the least witnessing $x, \seq{n_0,\ldots,n_r} < s$. Cancel the claims of any lower priority requirements, and say that $\mathcal{R}_{e,u}$ \emph{claims} $n_0,\ldots,n_r$. In either case, say that the requirement in question \emph{acts}. Finally, for each $n$ that is claimed by some $\mathcal{R}$-requirement at this point, enumerate $s$ into $V_n$. This completes the construction.
	
	Since the only time a number $s$ is enumerated into some $V_n$ during the construction is at stage $s$, the sequence $\seq{V_n : n \in \omega}$ is computable. It follows that the CSC space $\seq{\omega,\topol{U},k}$ generated from $\seq{V_n : n \in \omega}$ is computable as well.
	
	Next, note that if $n$ is even and $n \neq \overline{n}$, then $n$ and $\overline{n}$ can never be claimed at the same time (whether by the same requirement, or by different ones). Indeed, even indices $n$ can only be claimed by $\mathcal{R}$-requirements. When an even $n$ is claimed by a new $\mathcal{R}$-requirement at some stage $s$, then $\overline{n}$ must not already be claimed by any higher priority requirement (necessarily another $\mathcal{R}$-requirement). And if $\overline{n}$ was already claimed by a lower priority requirement, then this claim is first cancelled. It follows that, at the end of each stage $s$, the number $s$ is enumerated into at most one of $V_n$ or $V_{\overline{n}}$ for every even $n$ with $n \neq \overline{n}$. In particular, if $x \neq y$ then $V_{2\seq{x,y}}$ and $V_{2\seq{y,x}}$ are disjoint. By our remarks above, $\seq{\omega,\topol{U},k}$ is effectively Hausdorff.
	
	An induction on the priority ordering shows that each requirement acts only finitely often. This is because once a requirement acts at a stage after all stronger $\mathcal{R}$-requirements have stopped acting, it claims some indices, and these claims cannot later be cancelled. Hence, the requirement can never require attention and act again.
	
	It thus remains only to show that all requirements are satisfied. Fix a requirement, and let $s_0$ be least such that every stronger requirement has stopped acting prior to stage $s_0$. Thus, either $s_0 = 0$ or some stronger $\mathcal{R}$-requirement acted at stage $s_0-1$ and thereby cancelled any claim the requirement in question may have had. Either way, then, this requirement has no claims at the start of stage $s_0$. Hence, if the requirement in question is $\mathcal{D}_x$ for some $x$, then it requires attention at the least stage $s \geq s_0$ with $s > x$, at which point it claims some $2\seq{x,y}+1$ that has never been claimed by any requirement before. By choice of $s_0$, this claim is never cancelled at any stage after $s$, and so in particular, $2\seq{x,y}+1$ is never claimed by any $\mathcal{R}$-requirement. It follows that no numbers are enumerated into $V_{2\seq{x,y}+1}$ during the construction, so $V_{2\seq{x,y}+1}$ is just the singleton $\set{x}$ and $\mathcal{D}_x$ is satisfied.
	
	Now suppose the requirement in question is $\mathcal{R}_{e,u}$ for some $e,u$. Seeking a contradiction, suppose $\Phi_e$ is the characteristic function of an infinite set $Y$, $\Phi_u$ is total on $Y$, and for all $x \in Y$ we have $x \in \bigcap_{i \leq r} V_{n_i}$ where $\Phi_u(x) \converges = \seq{n_0,\ldots,n_r}$. Let $F_0$ be the set of all indices $n$ claimed by stronger $\mathcal{D}$-requirement prior to stage $s_0$, and $F_1$ the set of all indices $n$ claimed by stronger $\mathcal{R}$-requirements. By choice of $s_0$, none of these requirements acts at any stage $s \geq s_0$, so their claims in $F_0 \cup F_1$ are permanent and none of them claims any other numbers. By our earlier observations, if $n \in F_0$ then $V_n$ is a singleton, and if $n \in F_1$ is even and $n \neq \overline{n}$ then $V_n$ and $V_{\overline{n}}$ are disjoint. Moreover, by construction, if $n \in F_1$ then $V_n$ is cofinite, so if $n$ is even and $n \neq \overline{n}$ then $V_{\overline{n}}$ is finite. Let $m$ be the maximum of all $V_n$ for $n \in F_0$ and all $V_{\overline{n}}$ for even $n \in F_1$ with $n \neq \overline{n}$. Since $Y$ is infinite, we can choose $x > \max\set{m,s_0}$ in $Y$. Then if $\Phi_u(x) \converges = \seq{n_0,\ldots,n_r}$ we have $x \in V_{n_i}$ for all $i \leq r$, and hence it cannot be that $n_i \in F_0$ or $\overline{n_i} \in F_1$ if $n_i$ is even and $n_i \neq \overline{n_i}$. We conclude that, if we take the least $s \geq s_0$ for which there exists $x,\seq{n_0,\ldots,n_r} < s$ with these properties, and such that $\Phi_e(x)[s] \converges = 1$ and $\Phi_u(x)[s] \converges = \seq{n_0,\ldots,n_r}$, then $\mathcal{R}_{e,u}$ will require attention at stage $s$. It will then claim $n_0,\ldots,n_r$, and by choice of $s_0$, these claims will be permanent. Hence, $V_{n_i}$ will be cofinite for all $i \leq r$, proving that $\bigcap_{i \leq r} V_{n_i}$ cannot be $\set{x}$ after all. The proof is complete.
\end{proof}

\begin{corollary}\label{cor:Q1_partial}
	There exists an $\omega$-model of $\RCA_0$ which does not satisfy the statement that every infinite effectively Hausdorff CSC space has an infinite effectively discrete subspace.
\end{corollary}

\section{$T_1$ spaces}\label{sec:T1}

\subsection{An alternative proof of $\GS$ for $T_1$ CSC spaces}

We now return to studying the Ginsburg--Sands theorem for $T_1$ CSC spaces. We begin with an alternative proof of the $T_1$ case. Although not directly relevant to the rest of our discussion, the proof is interesting in its own right because it is more direct than the original. The following definition and lemma will be useful.

\begin{definition}
	Let $\seq{X,\topol{U},k}$ be a $T_1$ CSC space.
	\begin{enumerate}
		\item A pair of elements $x,y \in X$ is a \emph{$T_2$ pair} if there exist $U,V \in \topol{U}$ such that $x \in U$, $y \in V$, and $U \cap V = \emptyset$.
		\item $X$ is \emph{pure $T_1$} if it contains no $T_2$ pairs. 
	\end{enumerate}
\end{definition}

\begin{lemma}\label{lem:pureT1_cof}
	The following is provable in $\RCA_0$. Let $\seq{X,\topol{U},k}$ be an infinite computable pure $T_1$ CSC space.
	\begin{enumerate}
		\item For each $U \in \topol{U}$, if $U$ is nonempty then $U$ is infinite.
		\item If $F$ is a nonempty finite set such that $U_n \neq \emptyset$ for all $n \in F$, then $\bigcap_{n \in F} U_n$ is infinite.
		\item $X$ has an infinite subspace with the cofinite topology.
	\end{enumerate}
\end{lemma}

\begin{proof}
	To prove (1), fix any $x \in U \in \topol{U}$. Since $X$ is $T_1$, for each $y \neq x$ there is a $V \in \topol{U}$ containing $y$ but not $x$. Now by \Cref{lem:T_1_finite_is_singleton}, if $U$ were finite we would have $\set{x} \in \topol{U}$, whence $x$ and $y$ would be a $T_2$ pair as witnessed by $\set{x}$ and $V$, which cannot be.
	
	To prove (2), fix a finite set $F = \set{n_0 < \cdots < n_r}$ and suppose $U_{n_i} \neq \emptyset$ for all $i \leq r$. We use $\Sigma^0_1$ induction on $j \leq r$ to show that $\bigcap_{i \leq j} U_{n_i} \neq \emptyset$. This is clear for $j = 0$. Suppose it is true for some $j < r$, fix $x \in \bigcap_{i \leq j} U_{n_i}$, and let $n = \overline{k}(x,\set{n_i : i \leq j})$, where $\overline{k}$ is as given by \Cref{lem:kStar}. Then $x \in U_n \subseteq \bigcap_{i \leq j} U_{n_i}$. Now if $U_n \cap U_{n_{j+1}}$ were empty, then every element of $U_n$ would form a $T_2$ pair with every element of $U_{n_{j+1}}$, which is impossible. Hence, $U_n \cap U_{n_{j+1}} \subseteq \bigcap_{i \leq {j+1}} U_{n_i}$ is nonempty, as desired. We conclude that $\bigcap_{n \in F} U_n \neq \emptyset$. Fix $x$ in this intersection. Then $x \in U_{\overline{k}(x,F)} \subseteq \bigcap_{n \in F} U_n \neq \emptyset$. By (1), $U_{\overline{k}(x,F)}$ is infinite, so $\bigcap_{n \in F} U_n$ is too.
	
	To prove (3), we proceed as follows. Write $\topol{U} = \seq{U_n : n \in \mathbb{N}}$. $\RCA_0$ proves that there exists a function $f : \mathbb{N} \to \mathbb{N}$ such that $\ran(f) = \set{n : U_n \neq \emptyset}$. By (2), $\bigcap_{i \leq j} U_{f(i)}$ is infinite for every $j$. Let $x_0$ be any element of $U_{f(0)}$, and having defined $x_i$ for all $i < j$, let $x_j$ be the least element of $\bigcap_{i \leq j} U_{f(i)}$ larger than all the $x_i$. Let $Y = \set{x_i : i \in \omega}$. Then for each nonempty $U_n$, say with $f(i) = n$, we have that $x_j \in U_{f(i)} = U_n$ for all $j \geq i$. Thus, $Y$ is an infinite computable subspace of $X$ with the cofinite topology, as was to be shown.
\end{proof}



\begin{proof}[Alternative proof of $\GST_1$]
	We argue in $\ACA_0$. Let $\seq{X,\topol{U},k}$ be an infinite $T_1$ CSC space. Using arithmetical comprehension, define the following $2$-coloring of $[X]^2$: for $x < y$, 
	\[
		c(x,y) = 
		\begin{cases}
			0 & \text{if $x,y$ is a $T_2$ pair},\\
			1 & \text{otherwise}.
		\end{cases}
	\]
	Applying $\RT^2_2$, let $H \subseteq X$ be an infinite homogeneous set for $c$. If $H$ is homogeneous for color $0$, then $H$ forms a an infinite Hausdorff subspace of $X$. By \Cref{thm:t2hasdiscrete}, $H$ (and hence $X$) has an infinite discrete subspace. If, instead, $H$ is homogeneous for color $1$, then $H$ is an infinite pure $T_1$ subspace of $X$. By \Cref{lem:pureT1_cof}~(3), $H$ has an infinite subspace with the cofinite topology.	
\end{proof}

Note that the proof above featured a non-effective use of $\RT^2_2$; the coloring $c$ being only computable in the double jump of the space $\seq{X,\topol{U},k}$. Our next goal, therefore, is to understand the relationship between $\GST_1$ and $\RT^2_2$ more precisely.

\subsection{$\GST_1$ and $\RT^2_2$}

We introduce the following notion of ``stability'' for CSC spaces, and shortly show that it behaves much like stability for colorings of pairs.

\begin{definition}
	A CSC space $\seq{X,\topol{U},k}$ is \emph{stable} if for every $x \in X$, either $\set{x} \in \topol{U}$, or every $U \in \topol{U}$ containing $x \in U$ is cofinite. $\SGST_1$ is the restriction of $\GST_1$ to stable CSC spaces.
\end{definition}

\noindent Trivially, $\GST_1$ implies $\SGST_1$ over $\RCA_0$.

\begin{proposition}\label{prop:SGST1+COH_to_GST1}
	Over $\RCA_0$, $\SGST_1 + \COH \to \GST_1$.
\end{proposition}

\begin{proof}
	We argue in $\RCA_0 + \SGST_1 + \COH$. Let $\seq{X,\topol{U},k}$ be any $T_1$ CSC space. Apply $\COH$ to $\topol{U} = \seq{U_n : n \in \mathbb{N}}$, obtaining an infinite $\topol{U}$-cohesive set $Y$. Thus, for each $n$, either $Y \subseteq^* \overline{U_n}$ or $Y \subseteq^* U_n$. We claim that $Y$ is a stable subspace of $X$. Indeed, fix any $x \in Y$. If there is an $n$ such that $x \in U_n$ and $Y \subseteq^* \overline{U_n}$ then $Y \cap U_n$ is finite. Because $Y$ is $T_1$ (being a subspace of $X$), it follows by \Cref{lem:T_1_finite_is_singleton} that $\set{x} \in \topol{U} \res Y$. Otherwise, for every $n$ with $x \in U_n$ we have that $Y \subseteq^* U_n$, so $Y \cap U_n =^* Y$. This proves the claim. Now apply $\SGST_1$ to $Y$ to obtain a subspace $Z$ which is either discrete or has the cofinite topology.
\end{proof}

We now prove the converse, thereby obtaining a decomposition of $\GST_1$ into $\SGST_1 + \COH$, analogous to the Cholak--Jockusch--Slaman decomposition of $\RT^2_2$ into $\SRT^2_2 + \COH$. It is enough to show that $\GST_1$ implies $\COH$, which we obtain as follows.

\begin{theorem}\label{T:ADS_ured_GST1}
	Over $\RCA_0$, $\GST_1 \to \ADS$.	
\end{theorem}

\begin{proof}
	We argue in $\RCA_0$. Fix an infinite linear order $\seq{L,\leq_L}$. We define a $T_1$ CSC space with the property that any infinite discrete subspace is an ascending sequence under $\leq_L$, and any infinite subspace with the cofinite topology is a descending sequence. This yields $\ADS$.
	
	For each $x \in L$, let
	\[
		V_{\seq{x,x}} = \set{w \in L: w \leq_L x},
	\]
	and for each $y \neq x$ in $L$, let
	\[
		V_{\seq{x,y}} = \set{w \in L: w \leq_L x} \smallsetminus \set{y}.
	\]
	Let $\seq{L,\topol{U},k}$ be the CSC space on $L$ generated by $\seq{V_n : n \in \mathbb{N}}$. Note that $X$ is $T_1$: if $x \neq y$, then $V_{\seq{x,y}}$ contains $x$ but not $y$, and $V_{\seq{y,x}}$ contains $y$ but not $x$.
	
	Now, consider any infinite subspace $Y$ of $X$ with either the discrete or cofinite topology. We consider the two cases separately.
	
	If $Y$ has the discrete topology, then for every $x \in Y$ there is a $U \in \topol{U}$ such that $U \cap Y = \set{x}$. Then there is a finite set $F$ such that $U = \bigcap_{y \in F} V_{\seq{x,y}}$. (For if $x \in V_{\seq{z,y}}$ for some $z$ then necessarily $x \leq_L z$ and so $V_{\seq{x,y}} \subseteq V_{\seq{z,y}}$.) It follows that $|\set{w \in Y : w \leq_L x}| \leq |F|$, so $x$ has only finitely many $\leq_L$-predecessors in $Y$. Hence, $Y$ is an ascending sequence under $\leq_L$.
	
	If $Y$ has the cofinite topology, we have instead that for every $x \in Y$ and every $U \in \topol{U}$, if $x \in U$ then $U \cap Y =^* Y$. In particular, $V_{\seq{x,x}} \cap Y =^* Y$, meaning that almost every $z \in Y$ satisfies $z \leq_L x$. In other words, $x$ has only finitely many $\leq_L$-successors in $Y$. We conclude that $Y$ is a descending sequence under $\leq_L$, as was to be shown.
\end{proof}

\begin{corollary}\label{GST_1_decomp}
	Over $\RCA_0$, $\GST_1 \iff \SGST_1 + \COH$.
\end{corollary}

\begin{proof}
	As discussed in \Cref{sec:intro}, $\ADS$ implies $\COH$ by \cite[Proposition 2.10]{HS-2007}, so by the preceding theorem $\GST_1$ implies $\COH$. Since $\GST_1$ implies $\SGST_1$, the result now follows by \Cref{prop:SGST1+COH_to_GST1}.
\end{proof}

To gain further insights, we now introduce a definition that will allow us to formulate a principle that behaves with respect to $\SGST_1$ as $\D^2_2$ does with respect to $\SRT^2_2$.

\begin{definition}
	Fix $n \geq 1$.
	\begin{itemize}
		\item $\Sigma^0_n\SubP$ is the scheme consisting of all formulas of the form
		\[
			(\exists Y)[Y \text{ is infinite } \wedge ((\forall x)[x \in Y \to \varphi(x)] \vee (\forall x)[x \in Y \to \neg \varphi(x)])],
		\]
		where $\varphi(x)$ is a $\Sigma^0_n$ formula in the language of second-order arithmetic.
		\item $\Delta^0_n\SubP$ is the scheme consisting of all formulas of the form
		\begin{align*}
		(\forall x)[\varphi(x) \iff \psi(x)] \to\\
		(\exists Y)[Y \text{ is infinite } \wedge ((\forall x)[x \in Y \to \varphi(x)] \vee (\forall x)[x \in Y \to \neg \varphi(x)])],
		\end{align*}
		where $\varphi(x)$ and $\psi(x)$ are $\Sigma^0_n$ formulas in the language of second-order arithmetic.
	\end{itemize}
\end{definition}

\noindent As discussed in the introduction, $\D^2_2$ is equivalent over $\RCA_0$ to the principle $\Delta^0_2\SubP$. We get the following similar and surprising characterization of $\SGST_1$.

\begin{proposition}\label{prop:sig2_sgst_equiv}
	Over $\RCA_0$, $\SGST_1 \iff \Sigma^0_2$-$\mathsf{Subset}$.
\end{proposition}

\begin{proof}
	We argue in $\RCA_0$. First, assume $\SGST_1$. Let a $\Sigma^0_2$ formula $\varphi(x)$ be given. Fix a $\Delta^0_0$ formula $\theta(x,y,z)$ such that $(\forall x)[\varphi(x) \iff (\exists y)(\forall z)\theta(x,y,z)]$. For all $x,y,s \in \mathbb{N}$, define
	\[
		V_{\seq{x,y,s}} = \set{x} \cup \set{w > \max \set{x,s}: (\exists z \leq w)[\neg \theta(x,y,z)]},
	\]
	and observe the following fact about all $x$ and $y$. If $(\forall z)\theta(x,y,z)$ holds, then $V_{\seq{x,y,s}} = \set{x}$ for all $s$. Otherwise, there is a least $z$ such that $\neg \theta(x,y,z)$ holds since $\theta$ is $\Delta^0_0$, and $V_{\seq{x,y,s}} = \set{x} \cup \set{w > \max\set{x,s} : w \geq z}$ for all $s$. In particular, for all $x$ and $y$ either $V_{\seq{x,y,s}} = \set{x}$ for all $s$, or $V_{\seq{x,y,s}}$ is a cofinite subset of $\mathbb{N}$ for all $s$. Moreover, there is a $y$ such that $V_{\seq{x,y,s}} = \set{x}$ for some (equivalently, all) $s$ if and only if $\varphi(x)$ holds.
	
	Let $\seq{\mathbb{N},\topol{U},k}$ be the CSC space generated by $\seq{V_n : n \in \mathbb{N}}$. Fix any $x \in \mathbb{N}$ such that $\neg \varphi(x)$ holds, and any $U \in \topol{U}$ such that $x \in U$. We claim that $U$ is cofinite. Write $U = \bigcap_{n \in F} V_n$ for some nonempty finite set $F$, and for each $n \in F$ write $n = \seq{x_n,y_n,s_n}$. If $x_n = x$, then $V_{\seq{x_n,y_n,s_n}}$ is cofinite by the argument above because $\neg \varphi(x)$ holds. If $x_n \neq x$ then $V_{\seq{x_n,y_n,s_n}} \neq \set{x_n}$ since it contains $x$, hence it is also cofinite. Thus, each $V_n$ for $n \in F$ is cofinite. For each $n \in F$, fix the least $z_n > x_n$ in $V_n$; the set $D = \set{z_n : n \in F}$ exists by $\Delta^0_1$ comprehension and is nonempty. Furthermore, as observed above, $V_n = \set{x_n} \cup \set{w : w \geq z_n}$. Hence, every $w \geq \max D$ belongs to $\bigcap_{n \in F} V_n = U$, which proves the claim.
	
	We conclude that $X$ is stable: if $\varphi(x)$ holds then $V_{\seq{x,y,0}} = \set{x}$ for any $y$ such that $(\exists y)(\forall z)\theta(x,y,z)$ holds; if $\neg \varphi(x)$ then every $U \in \topol{U}$ containing $x$ is cofinite. Moreover, $X$ is $T_1$: given $x \neq y$, we have that $V_{x,0,y}$ contains $x$ but not $y$, while $V_{y,0,x}$ contains $y$ but not $x$. Let $Y$ be any infinite subspace of $X$ with either the discrete or cofinite topology. By construction, in the former case we have that $\varphi(x)$ holds for all $x \in Y$; in the latter case, we have that $\neg \varphi(x)$ holds for all $x \in Y$. Hence, $\Sigma^0_2$-$\mathsf{Subset}$ holds.
	
	Conversely, assume $\Sigma^0_2$-$\mathsf{Subset}$ and let $\seq{X,\topol{U},k}$ be any infinite stable $T_1$ CSC space. Write $\topol{U} = \seq{U_n : n \in \mathbb{N}}$, and let $\varphi(x)$ be the $\Sigma^0_2$ formula $(\exists n)[U_n = \set{x}]$. Notice that by stability of $X$, $\neg \varphi(x)$ holds if and only if every $U \in \topol{U}$ containing $x$ is cofinite. Let $Y$ be any infinite set such that either $\varphi(x)$ holds for all $x \in Y$, or $\neg \varphi(x)$ holds for all $x \in Y$. In the former case, we have that $Y$ forms a discrete subspace of $X$; in the latter, it is a subspace with the cofinite cofinite topology. So, $\SGST_1$ holds. 
\end{proof}

\begin{corollary}\label{SGST1_to_SRT22}
	Over $\RCA_0$, $\SGST_1$ implies $\SRT^2_2$.
\end{corollary}

\begin{proof}
	Immediate, since $\Sigma^0_2\SubP$ obviously implies $\Delta^0_2\SubP$.
\end{proof}

\begin{corollary}\label{cor:GST1_to_RT22}
	Over $\RCA_0$, $\GST_1$ implies $\RT^2_2$.
\end{corollary}

\begin{proof}
	By the preceding corollary, $\GST_1$ implies $\SRT^2_2$. Since $\GST_1$ implies $\COH$, the conclusion now follows from that fact that $\SRT^2_2 + \COH \to \RT^2_2$.
\end{proof}

\noindent We will prove in \Cref{sec:separating} that this implication is strict.

\Cref{GST_1_decomp} has the following additional consequence.

\begin{corollary}\label{cor:GST_1_no_ACA}
	There exists an $\omega$-model satisfying $\RCA_0 + \GST_1$ but not $\WKL_0$.
\end{corollary}

\begin{proof}
	Liu \cite[Theorem 1.5 and Corollary 1.6]{Liu-2012} established that $\COH$ admits \emph{PA avoidance}, and $\RT^1_2$ admits \emph{strong PA avoidance}. This means that for any set $X \not\gg \emptyset$ (i.e., for any set $X$ not of PA degree), every $X$-computable instance of $\COH$ has a solution $Y$ such that $X \oplus Y \not\gg \emptyset$, and every instance of $\RT^1_2$ (computable from $X$ or not) has a solution $H$ such that $X \oplus H \not\gg \emptyset$. We can combine these facts to produce an $\omega$-model of $\Sigma^0_2\SubP + \COH$, and hence (by \Cref{prop:SGST1+COH_to_GST1,prop:sig2_sgst_equiv}) of $\GST_1$, that contains no set of PA degree, and therefore cannot satisfy $\WKL_0$. For completeness, we give the details. We define a sequence $X_0 \Tred X_1 \Tred \cdots$ of subsets of $\omega$, and define (the second-order part of) an $\omega$-model $\mathcal{M}$ to be $\set{Y : (\exists s)(Y \Tred X_s)}$. Let $X_0 = \emptyset$, and suppose inductively that we have defined $X_s \not\gg \emptyset$ for some $s \in \omega$. First, suppose $s = 2\seq{e,t}$, so that $e,t \leq s$ by standard conventions. If $\Phi_e^{X_t}$ defines a family $\vec{R}$ of subsets of $\omega$, then also $\vec{R} \Tred X_s$, and we appeal to PA avoidance of $\COH$ to fix an infinite $\vec{R}$-cohesive set $Y$ such that $X_s \oplus Y \not\gg\emptyset$. We then let $X_{s+1} = X_s \oplus Y$. If $\Phi_e^{X_t}$ does not define a family, we simply let $X_{s+1} = X_s$. Now, suppose $s = 2\seq{e,t}+1$. We consider $W_e^{X_t'}$, the $e$th $\Sigma^0_2$ set relative to $X_t$, and regard it as the instance $c : \omega \to 2$ of $\RT^1_2$ defined by $c(x) = 1$ if and only if $x \in W_e^{X_t'}$. We appeal to strong PA avoidance of $\RT^1_2$ to fix an infinite homogeneous set $H$ for $c$ such that $X_s \oplus H \not\gg \emptyset$, noting that $H$ is an infinite subset of either $W_e^{X_t'}$ or its complement. We then let $X_{s+1} = X_s \oplus H$. This completes the construction of $\mathcal{M}$. Since $X_s \not\gg \emptyset$ for all $s$, $\mathcal{M} \not\models \WKL_0$. To verify that $\mathcal{M}$ satisfies $\COH$, consider any instance $\vec{R}$ of $\COH$ in $\mathcal{M}$. Then $\vec{R} = \Phi_e^{X_t}$ for some $e$ and $t$, and a solution to $\vec{R}$ is thus computable from $X_{2\seq{e,t}}$ and so belongs to $\mathcal{M}$. Next, consider any instance of $\Sigma^0_2\SubP$ in $\mathcal{M}$, meaning a set of the form $W_e^{X_t'}$ for some $e$ and $t$. Then a solution to this instance is computable from $X_{2\seq{e,t}+1}$.
\end{proof}

In particular, it follows that $\GST_1$ also does not imply $\ACA_0$. In fact, by using strong cone avoidance of $\RT^1_2$ in place of strong PA avoidance in the proof above, we can deduce that $\GST_1$ admits cone avoidance. For another proof of this result, which provides further insights into the computability-theoretic aspects of $T_1$ CSC spaces, see Benham \cite{Benham-TA}.

\subsection{Arithmetical bounds}

In the previous section, we established that $\GST_1$ has at least the same proof-theoretic strength as $\RT^2_2$ relative to $\RCA_0$. We now show that, in terms of the arithmetical hierarchy, $\GST_1$ is actually less complex than $\RT^2_2$. Indeed, Jockusch \cite[Theorem 5.1]{Jockusch-1972} constructed a computable instance of $\RT^2_2$ with no $\Delta^0_2$ solutions. By contrast, $\GST_1$ always admits $\Delta^0_2$ solutions.

\begin{theorem}\label{thm:Delta2_solns}
	Let $\seq{X,\topol{U},k}$ be an infinite computable $T_1$ CSC space. Then $X$ has an infinite $\Delta^0_2$ discrete subspace, or an infinite computable subspace with the cofinite topology.
\end{theorem}

\begin{proof}
	Write $\topol{U} = \seq{U_n : n \in \omega}$. The proof breaks into three scenarios for finding the required subspace. First, consider the $\Sigma^0_2$ set $S = \set{x \in X : (\exists n)[x \in U_n \wedge U_n \text{ is finite}]}$. By \Cref{lem:T_1_finite_is_singleton}, $x \in S$ if and only if $\set{x} \in \topol{U}$, so $S$ is a discrete subspace of $X$. If $S$ is infinite, we can fix an infinite $\Delta^0_2$ subset $Y$ of $S$, and this then forms an infinite $\Delta^0_2$ discrete subspace of $X$.
	
	So suppose $S$ is finite. By replacing $X$ with $X \setminus S$, and each $U \in \topol{U}$ with $U \setminus S$, we may assume without loss of generality that every nonempty element of $\topol{U}$ is infinite. (Since $S$ is finite, this change actually results in a computable subspace of $X$.) Now, suppose there exists a nonempty $U \in \topol{U}$ such that $U$ contains no $T_2$ pairs. Then $U$ forms an infinite computable pure $T_1$ subspace of $X$, so by \Cref{lem:pureT1_cof}~(3) it has an infinite computable subspace $Y$ with the cofinite topology, and this is of course also a subspace of $X$.
	
	For the remaining case, then, we can assume every nonempty $U \in \topol{U}$ contains a $T_2$ pair. We use $\emptyset'$ to construct an infinite $\Delta^0_2$ discrete subspace of $X$. The key point we will use is that, given $m,n \in \omega$, asking whether $U_n = \emptyset$, whether $U_n \cap U_m = \emptyset$, or whether $U_n \subseteq U_m$ can all be answered by $\emptyset'$, uniformly in $n$ and $m$. This has the following consequences. First, given $n$ such that $U_n \neq \emptyset$, $\emptyset'$ can uniformly find a $T_2$ pair $x,y \in U_n$. And second, suppose we are given $i$ and $x,y \in U_i$ forming a $T_2$ pair. Then $\emptyset'$ can uniformly find indices $m$ and $n$ such that $x \in U_m$, $y \in U_n$, $U_m \cap U_n = \emptyset$, and $U_m,U_n \subseteq U_i$. To see this note, $\emptyset'$ first uniformly finds $s$ and $t$ such that $x \in U_s$, $y \in U_t$, $U_s \cap U_t = \emptyset$. Then $m = k(x,i,s)$ and $n = k(x,i,t)$ have the desired properties. Indeed, we have $x \in U_n \subseteq U_i \cap U_s$ and $y \in U_m \subseteq U_i \cap U_t$, and $U_n \cap U_m \subseteq U_s \cap U_t = \emptyset$, as needed.
	
	We now construct $\Delta^0_2$ sequences of points $x_0 < x_1 < \cdots$ and $y_0 < y_1 < \cdots$ in $X$, and sequences of indices $m_0,m_1,\ldots$ and $n_0,n_1,\ldots$ such that, for all $i$:
	\begin{itemize}
		\item $x_i \in U_{m_i}$;
		\item $y_i \in U_{n_i}$;
		\item $U_{m_i} \cap U_{n_i} = \emptyset$;
		\item $U_{m_{i+1}} \subseteq U_{n_i}$.
	\end{itemize}
	In particular, $x_i \in U_{m_j}$ if and only if $i = j$.
	Hence, if we set $Y = \set{x_i : i \in \omega}$, then $Y$ forms the desired $\Delta^0_2$ discrete subspace of $X$.
	
	To define our sequences, we proceed by recursion. Let $x_0,y_0$ be an arbitrary $T_2$ pair in $X$, and let $m_0$ and $n_0$ be such that $x_0 \in U_{m_0}$, $y_0 \in U_{n_0}$, and $U_{m_0} \cap U_{n_0} = \emptyset$. Now fix $i$, and suppose $x_i$, $y_i$, $m_i$, and $n_i$ have been defined and satisfy the above properties. Since $U_{n_i}$ is nonempty, $\emptyset'$ can uniformly find a $T_2$ pair $x,y \in U_{n_i}$. It can then uniformly find indices $m_{i+1}$ and $n_{i+1}$ such that $x \in U_{m_{i+1}}$, $y \in U_{n_{i+1}}$, $U_{m_{i+1}} \cap U_{n_{i+1}} = \emptyset$, and $U_{m_{i+1}},U_{n_{i+1}} \subseteq U_{n_i}$. Now $U_{m_{i+1}}$ and $U_{n_{i+1}}$ are nonempty, hence infinite, so we can fix the least $x_{i+1} > x_{i}$ in $U_{m_{i+1}}$, and the least $y_{i+1} > y_{i}$ in $U_{n_{i+1}}$. Clearly, $x_{i+1}$, $y_{i+1}$, $m_{i+1}$, and $n_{i+1}$ maintain the above properties. This completes the construction, and finishes the proof.
\end{proof}

In the parlance of computable reducibility (see \cite{DM-2022}, Chapter 4), it follows that $\RT^2_2 \ncred \GST_1$. This is somewhat unusual. Most implications over $\RCA_0$ between $\Pi^1_2$ statements are in fact formalizations of computable reductions.

We note, too, that \Cref{thm:Delta2_solns} is not symmetric.

\begin{proposition}
	There exists an infinite computable $T_1$ CSC space having no computable discrete subspace, and no $\Delta^0_2$ subspace with the cofinite topology.
\end{proposition}

\begin{proof}
	Fix any $\Sigma^0_2$ set with no infinite computable subset, and no infinite $\Delta^0_2$ subset in its complement.	By the proof of \Cref{prop:sig2_sgst_equiv}, we can view this an infinite computable instance of $\GST_1$ (in fact, a stable one) with the desired properties.
\end{proof}

In spite of the curious distinction between $\GST_1$ and $\RT^2_2$ highlighted by \Cref{thm:Delta2_solns}, there are other aspects in which $\GST_1$ and $\RT^2_2$ behave quite similarly. One of these is PA avoidance, which holds for $\RT^2_2$ by Liu's theorem (discussed in the introduction) and for $\GST_1$ by \Cref{cor:GST_1_no_ACA} above. Another is the following. Cholak, Jockusch, and Slaman \cite[Theorem 3.1]{CJS-2001} showed that every computable instance of $\RT^2_2$ has a solution $H$ which is low$_2$ (i.e., satisfies $H'' \Tred \emptyset''$). In fact, they offered two proofs of this result, one by controlling the first jump of $H$ during its construction (\cite{CJS-2001}, Section 4), the other by controlling its second jump (\cite{CJS-2001}, Section 5). We can adapt the latter proof to obtain the analogous result for $\GST_1$.

\begin{theorem}\label{thm:Sig2_low2}
	Every $\Sigma^0_2$ set has an infinite low$_2$ subset $Y$ in it or its complement.	
\end{theorem}

\begin{proof}
	In \cite[Section 5.2]{CJS-2001}, it is shown that if $A$ is any $\Delta^0_2$ set then there exists an infinite subset of $A$ or $\overline{A}$ which is low$_2$.	The authors comment in the proof (Section 5.2.1 on p.~19) that, except for one moment, their argument also works for $A$ being $\Delta^0_3$. This moment is when they consider a certain sentence of the form
	\begin{equation}\label{CJS_eqn}
		(\exists z)[R^A(z) \wedge (\exists Z)P(z,Z,D,L,S)],
	\end{equation}
	which they need to be $\Sigma^0_2$ uniformly in $D$, $S$, and $L$. Here, $D$ and $S$ are finite sets specified by canonical indices, $L$ is a low set specified by a $\Delta^{0,\emptyset'}_1$ index for its jump, and $P$ is a $\Pi^0_1$ predicate. Thus, $(\exists Z)P(z,Z,D,L,S)$ is itself uniformly $\Pi^{0,L}_1$. The predicate $R^A(z)$, on the other hand, asserts that a certain finite sequence of finite sets coded by $z$ is contained in $A$, and as such is $\Delta^0_2$. But notice that if $A$ is $\Sigma^0_2$, rather of $\Delta^0_2$, then \eqref{CJS_eqn} is still $\Sigma^0_2$, so the proof still works. The rest of the argument from \cite{CJS-2001} then goes through as before.
\end{proof}

\begin{corollary}\label{cor:GST_1_low_2_solutions}
	Let $\seq{X,\topol{U},k}$ be an infinite computable $T_1$ CSC space. Then $X$ has an infinite low$_2$ subspace which is discrete or has the cofinite topology. 
\end{corollary}

\begin{proof}
	Fix $\seq{X,\topol{U},k}$. Jockusch and Stephan \cite[Theorem 2.1]{JS-1993} showed that every computable instance of $\COH$ has a low$_2$ solution. So let $Y$ be an infinite low$_2$ $\mathcal{U}$-cohesive set. Then $Y$ is stable as a subspace of $X$, and we let $Y$ be the $\Sigma^0_2(Y)$ set of all $x \in Y$ such that $(\exists U \in \topol{U})[U = \set{x}]$. By the proof of \Cref{prop:sig2_sgst_equiv}, every infinite subset of $S$ is a discrete subspace of $Y$ (and hence of $X$) and every infinite subset of $Y \setminus S$ is a subspace with the cofinite topology. Apply \Cref{thm:Sig2_low2} to find an infinite subset $Z$ of $S$ or $Y \setminus S$ which is low$_2$ over $Y$. Then $Z$ is itself low$_2$, which completes the proof.
\end{proof}

\begin{corollary}\label{low_2_model}
	There exists an $\omega$-model of $\GST_1$ consisting entirely of low$_2$ sets.
\end{corollary}

\begin{proof}
	Iterate \Cref{cor:GST_1_low_2_solutions}, as usual.
\end{proof}

\section{Separating $\GST_1$ from $\RT^2_2$}\label{sec:separating}

This section is dedicated to the proof of the following theorem.

\begin{theorem}\label{RT22_not_implies_Sig2}
	There exists an $\omega$-model satisfying $\RCA_0 + \RT^2_2$ but not $\Sigma^0_2\SubP$.
\end{theorem}

\noindent We design a property preserved by each of $\D^2_2$ and $\COH$, but not by $\Sigma^0_2\SubP$. Such a property in general needs to be designed very carefully, so as to strike the right balance between the principle(s) meant to preserve it and the principle meant not to. At the very least, this property should imply not computing a solution to some fixed instance $S$ of the ``stronger'' principle (in our case, $\Sigma^0_2\SubP$), but usually it has to entail quite a bit more than that. This is because we must be able to show that each instance of the ``weaker'' principles (in our case, $\D^2_2$ and $\COH$) that has this property also has a solution with this property, and for this, knowing merely that the given instance does not compute a solution to $S$ may not be enough. Thus, the basic tension is between increasing the complexity of the property to the point of being able to extract useful information from an instance that enjoys it, but not to the point that we can no longer construct a solution that does as well. For a general account of preservation arguments, and many examples of how they are used to obtain separations over $\RCA_0$, see Patey \cite{Patey-2017b}.

Our  argument divides into the following three lemmas. The first expresses a so-called ``fairness property'', while the second and third express that this property is preserved under $\D^2_2$ and $\COH$. We recall that if $S, U \subseteq \omega$ are sets, then $S$ is \emph{$U$-immune} if $U$ computes no infinite sequence of $S$; $S$ is \emph{$U$-hyperimmune} if $U$ computes no infinite sequence $x_0 < x_1 < \cdots$ of numbers such that $[x_n,x_{n+1}) \cap S \neq \emptyset$ for all $n$. Say $\Phi^U_e$ is a \emph{non-hyperimmunity witness} for $S$ if it is the characteristic function of such a sequence. Thus, a set $S$ is $U$-hyperimmune if and only if $\Phi^U_e$ is not a non-hyperimmunity witness for $S$, for all $e$. Being $U$-hyperimmune implies being $U$-immune. If $S$ and $\overline{S}$ are both $U$-immune, or both $U$-hyperimmune, then $S$ is \emph{$U$-bi-immune}, respectively, \emph{$U$-bi-hyperimmune}. If $S$ computes no $U$-hyperimmune set then it is said to be of \emph{hyperimmune-free degree} relative to $U$, and this is equivalent to every $S$-computable function being dominated by a $U$-computable function. (See, e.g., \cite[Section 2.17]{DH-2010} for details on these notions.)

\begin{lemma}\label{sep_ground}
	There exists a $\Sigma^0_2$ set $S$, a low set $L_0 \gg \emptyset$, and a set $P_0 \gg \emptyset'$ such that $S$ is $L_0$-hyperimmune and $\overline{S}$ is $P_0$-hyperimmune.
\end{lemma}

\begin{lemma}\label{P_2_sep_iterate}
	Let $X$, $S$, $L$, and $P$ be sets as follows:
	\begin{itemize}
		\item  $S$ is $\Sigma^0_2$ (nota bene, not merely $\Sigma^0_2$ relative to $X$);
		\item $L \gg X$;
		\item $P \gg L'$;
		\item $S$ is $L$-hyperimmune;
		\item $\overline{S}$ is $P$-hyperimmune.
	\end{itemize}
 	For every set $A \Tred X'$, there exists an infinite set $G$ contained in $A$ or $\overline{A}$ along with sets $L_G$ and $P_G$ as follows:
 	\begin{itemize}
 		\item $L_G \gg X \oplus G$;
 		\item $P_G \gg L_G'$;
 		\item $S$ is $L_G$-hyperimmune;
 		\item $\overline{S}$ is $P_G$-hyperimmune.
 	\end{itemize}
\end{lemma}

We have an analogous result for $\COH$.

\begin{lemma}\label{P_2_sep_COH_iterate}
	Let $X$, $S$, $L$, and $P$ be sets as follows:
	\begin{itemize}
		\item  $S$ is $\Sigma^0_2$;
		\item $L \gg X$;
		\item $P \gg L'$;
		\item $S$ is $L$-hyperimmune;
		\item $\overline{S}$ is $P$-hyperimmune.
	\end{itemize}
 	For every family of sets $\vec{R} \Tred X$, there exists an infinite $\vec{R}$-cohesive set $G$ along with sets $L_G$ and $P_G$ as follows:
 	\begin{itemize}
 		\item $L_G \gg X \oplus G$;
 		\item $P_G \gg L_G'$;
 		\item $S$ is $L_G$-hyperimmune;
 		\item $\overline{S}$ is $P_G$-hyperimmune.
 	\end{itemize}
\end{lemma}

Observe that, since $G \Tred L_G \Tred P_G$ in both of the above lemmas, the set $S$ will be $(X \oplus G)$-bi-immune (in fact, $(X \oplus G)$-bi-hyperimmune). If we were dealing with a single set $A$, or a single family $\vec{R}$, then we could in fact find a set $G$ with this property more simply, and without needing to also construct $L_G$ and $P_G$. But this would only help show that $\Sigma^0_2\SubP \ncred \D^2_2$ and $\Sigma^0_2\SubP \ncred \COH$, respectively. We need these auxiliary sets in order to be able to iterate the property and build an $\omega$-model separation. We show how this is done, and then move on to the lemmas.

\begin{proof}[Proof of \Cref{RT22_not_implies_Sig2}]
	We construct an $\omega$-model by applying \Cref{sep_ground} and then iterating and dovetailing \Cref{P_2_sep_iterate,P_2_sep_COH_iterate}. To this end, we define a chain of sets $X_0 \Tred X_1 \Tred \cdots$ and let our model be $\mathcal{M} = \set{Y : (\exists s)(Y \Tred X_s)}$, as usual. We proceed by induction.
	
	Let $X_0 = \emptyset$, and let $S$, $L_0$, and $P_0$ be as given by \Cref{sep_ground}. Notice that, if we take $X = X_0$, $L = L_0$, and $P = P_0$, then $X$, $S$, $L$, and $P$ satisfy the hypotheses of \Cref{P_2_sep_iterate,P_2_sep_COH_iterate}.
	
	Next, fix $s \in \omega$, and suppose that for each $t \leq s$ we have defined $X_t$, $L_t$, and $P_t$ such that, if we take $X = X_t$, $L = L_t$, and $P = P_t$, then $X$, $S$, $L$, and $P$ satisfy the hypotheses of \Cref{P_2_sep_iterate,P_2_sep_COH_iterate}. Say $s$ is $2\seq{e,t}$ or $2\seq{e,t}+1$, so that $e,t \leq s$ by standard conventions. In the first case, we check if $\Phi_e^{X_t'}$ defines a set $A$, and in the second, we check if $\Phi_e^{X_t}$ defines a family of sets $\vec{R}$. If not, we let $X_{s+1} = X_s$, $L_{s+1} = L_s$, and $P_{s+1} = P_s$ in either case. Otherwise, we have $A \Tred X_s'$ or $\vec{R} \Tred X_s$, respectively. We can thus apply \Cref{P_2_sep_iterate} in the first case, or \Cref{P_2_sep_COH_iterate} in the second, to obtain sets $G$, $L_G$, and $P_G$. (So, in the first case, we have that $G$ is an infinite subset of $A$ or $\overline{A}$, and in the second, that $G$ is an infinite $\vec{R}$-cohesive set.) In either case, as noted above, we have that $X_s \oplus G$ does not compute any infinite subset of $S$ or $\overline{S}$. We now set $X_{s+1} = X_s \oplus G$, $L_{s+1} = L_G$ and $P_{s+1} = P_G$, and notice that, if we take $X = X_{s+1}$, $L = L_{s+1}$, and $P = P_{s+1}$, then we again satisfy the hypotheses of \Cref{P_2_sep_iterate,P_2_sep_COH_iterate}. So the inductive conditions are maintained.
	
	Let $\mathcal{M}$ be the model specified above. By construction, if $A$ is any instance of $\D^2_2$ in $\mathcal{M}$ then $A = \Phi_e^{X_t'}$ for $e$ and $t$, and then some solution $G$ to $A$ is computable from $X_{2\seq{e,t}}$ and so belongs to $\mathcal{M}$. Similarly, if $\vec{R}$ is any instance of $\COH$ in $\mathcal{M}$ then $\vec{R} = \Phi_e^{X_t}$ for some  for $e$ and $t$, and now a solution $G$ to $\vec{R}$ is computable $X_{2\seq{e,t}+1}$. Hence, $\mathcal{M}$ satisfies $\D^2_2 + \COH$, and therefore $\RT^2_2$. On the other hand, by induction, no $X_s$ computes any infinite subset of $S$ or $\overline{S}$, so $S$ is a computable instance of $\Sigma^0_2\SubP$ with no solution in $\mathcal{M}$. Hence, $\mathcal{M}$ does not satisfy $\Sigma^0_2\SubP$.
\end{proof}

\begin{corollary}
	There exists an $\omega$-model satisfying $\RCA_0 + \RT^2_2$ but not $\GST_1$.
\end{corollary}

\subsection{Proving \Cref{sep_ground}}

For the remainder of this section, we follow the convention that if $\Phi^F(x) \converges$ for some Turing functional $\Phi$, finite set $F$, and $x \in \omega$, then the use of this computation is bounded by $\max F$. For a set $U$, let $T^2_U \subseteq 2^{<\omega}$ be the standard $U$-computable tree whose paths are precisely the DNC$_2$ functions relative to $U$. By well-known results of Jockusch and Soare \cite{JS-1972} and Solovay (unpublished), $X \gg U$ if and only if $X \Tequiv f$ for some $f \in [T^2_U]$; see \cite[Theorem 2.22.2]{DH-2010}.

Our original proof of \Cref{sep_ground} built $S$, $L_0$, and $P_0$ directly, via a slightly more complicated construction than the one that follows. We thank the anonymous referee for suggesting the more streamlined version below.

\begin{proof}[Proof of \Cref{sep_ground}]
	We let $L_0$ be any low set of PA degree, and $P_0 \gg L_0'$ any set of hyperimmune-free degree relative to $L_0'$. These exist, respectively, by applying the low basis theorem to $T^2_\emptyset$, and the hyperimmune-free basis theorem, relativized to $L_0'$, to $T^2_{L_0'}$ (Jockusch and Soare \cite[Theorems 2.2 and 2.4]{JS-1972}). We construct a $\Sigma^0_2$ set $S$ such that $S$ is $L_0$-hyperimmune and $\overline{S}$ is $L_0'$-hyperimmune. Since $P_0$ is of hyperimmune-free degree relative to $L_0'$, this implies that $\overline{S}$ is also $P_0$-hyperimmune, as needed. (To see this, fix any $P_0$-computable sequence $x_0 < x_1 < \cdots$. Since $P_0$ is of hyperimmune-free degree relative to $L_0'$, there is a function $g \Tred L_0'$ such that $g(n) \geq x_n$ for all $n$. Define a function $h$ by $h(0) = x_0$ and $h(n+1) = g(h(n)+1)$ for all $n$. Then $h \Tred g \Tred L_0'$, and we have $h(n+1) \geq x_{h(n)+1} > x_{h(n)} \geq h(n)$ for all $n$. Since $\overline{S}$ is $L_0'$-hyperimmune, there exists an $n$ such that $[h(n),h(n+1)) \subseteq S$, and hence also $[x_{h(n)},x_{h(n)+1}) \subseteq S$.)
	
	We enumerate $S$ by stages, using $L_0'$ as an oracle. Thus $S$ will be $\Sigma^0_1(L_0')$, and so $\Sigma^0_2$ since $L_0$ is low. We aim to satisfy the following requirements for all $e \in \omega$:
	\begin{eqnarray*}
		\Lreq_e & : &  \text{$\Phi^{L_0}_e$ is not the characteristic function of an infinite set or}\\
		& & \text{not a non-hyperimmunity witness for $S$};\\
		\Preq_e & : &  \text{$\Phi^{L_0'}_e$ is not the characteristic function of an infinite set or}\\
		& & \text{not a non-hyperimmunity witness for $\overline{S}$}. 
	\end{eqnarray*}
	We call requirements of the former kind the \emph{$\Lreq$ requirements}, and those of the latter kind the \emph{$\Preq$ requirements}. It is clear that satisfying all $\Lreq$ and $\Preq$ requirements ensures that $S$ will be $L_0$-hyperimmune and $\overline{S}$ will be $L_0'$-hyperimmune.
	
	We assign a priority ordering to the requirements of type $\omega$, with $\Preq_e$ stronger than $\Lreq_e$, and $\Lreq_e$ stronger than $\Preq_{e+1}$, for all $e \in \omega$. Requirements will \emph{act} at certain stages, never more than one requirement per stage, in a way that we will make precise below. When some $\Lreq_e$ acts, it will impose a restraint $r_e$ that may be cancelled later and redefined later still; this will be used to control which elements can be enumerated into $S$ by lower priority $\Preq$ requirements. At each stage $s$, we also let $m_s$ be the maximum of all restraints (including cancelled or redefined ones) imposed at any point in the construction prior to stage $s$, as well as all numbers enumerated into $S$ so far.
	
	We say a requirement $\Lreq_e$ \emph{requires attention} at stage $s$ if there exist numbers $x_1 > x_0 \geq m_s$ such that $\Phi_e^{L_0}(x_0) \converges = \Phi_e^{L_0}(x_1) \converges = 1$. We say a requirement $\Preq_e$ \emph{requires attention} at stage $s$ if there exist numbers $x_0 < x_1 \leq s$, larger than all active restraints at stage $s$ imposed by any $\Lreq$ requirements stronger than $\Preq_e$, such that $\Phi_e^{L_0'}(x_0)[s] \converges = \Phi_e^{L_0'}(x_1)[s] \converges = 1$. Notice that, in either case, determining if a given requirement requires attention at a given stage is computable in $L'_0$, uniformly in the index of the requirement and the stage number.
	
	Initially, we let $S$ be empty and let no restraints be imposed. (In particular, $m_0 = 0$.) Now suppose we are at stage $s$. If it exists, fix the highest priority requirement of index $e \leq s$ that requires attention at stage $s$ and that has not acted since the last stage, if any, at which a higher priority $\Preq$ requirement acted. If this requirement is $\Lreq_e$, we fix the least witnessing $x_1 > x_0 \geq m_s$ and impose the restraint $r_e = x_1$. We then say that $\mathcal{L}_e$ has \emph{acted}. If, instead, the requirement is $\Preq_e$, we fix the least witnessing $x_0 < x_1$ and enumerate every element of the interval $[x_0,x_1)$ into $S$. We then say $\Preq_e$ has acted, and cancel the restraints of any lower priority $\Lreq$ requirements.
	
	This completes the construction. We claim that each requirement acts at most finitely often. Suppose not, and fix the highest priority requirement that acts at infinitely many stages. Let $s < t$ be stages at which this requirement acts, but such that every stronger requirement has stopped acting before stage $s$. Then this requirement must require attention at stage $t$, and moreover, it must not have acted since the last stage at which some stronger $\Preq$ requirement did. In particular, the requirement must not have acted at stage $s$, a contradiction. So the claim is proved.
	
	Now, a restraint $r_e$ is only ever cancelled if some $\Preq$ requirement stronger than $\Lreq_e$ acts, and it is only defined if $\Lreq_e$ itself acts. The preceding claim thus implies that each $r_e$ is either undefined from some stage on or reaches a final value.
	
	It remains to verify that all requirements are satisfied. Consider any requirement, and let $s$ be least such that every higher priority $\Preq$ requirement has stopped acting prior to stage $s$. If we are dealing with $\Lreq_e$, and if $\Phi_e^{L_0}$ is the characteristic function of an infinite set, then $\Lreq_e$ will require attention at stage $s$. This means that for some $x_1 > x_0 \geq m_s$ with $\Phi_e^{L_0}(x_0) \converges = \Phi^{L_0}_e(x_1) \converges = 1$, $\Lreq_e$ will now impose the restraint $r_e = x_1$, and this restraint will be permanent. Only $\Preq$ requirements enumerate elements into $S$, and the only such requirements that can act after stage $s$ must respect $r_e$. Hence, every number in the interval $[x_0,x_1)$ will be permanently kept out of $S$, and so belong to $\overline{S}$, as needed.
	
	Now suppose we are dealing with $\Preq_e$, and that $\Phi_e^{L_0'}$ is the characteristic function of an infinite set. Any $\Lreq$ requirement stronger than $\Preq_e$ that acts at or after stage $s$ will do so before $\Preq_e$ can act. This is because the action of $\Preq_e$ cannot change any such $\Lreq$ requirement from not requiring attention to requiring attention. Let $t \geq s$ be least such that all such $\Lreq$ requirements have acted prior to stage $t$. In particular, any restraints they impose are now permanent, by choice of $s$. By hypothesis, there is a $u \geq t$ such that $\Preq_e$ acts at stage $u$, which means that for some $x_1 > x_0$ such that $\Phi_e^{L_0'}(x_0) \converges = \Phi^{L_0'}_e(x_1) \converges = 1$ every element of $[x_0,x_1)$ is enumerated into $S$. This completes the proof.
\end{proof}

\subsection{Proving \Cref{P_2_sep_iterate}}

Throughout this section, let $X$, $S$, $L$, and $P$ be fixed as in the statement of \Cref{P_2_sep_iterate}, along with a set $A \Tred X'$. For ease of notation, we write $A^0$ for $A$ and $A^1$ for $\overline{A}$.

Notice that if $A^0$ or $A^1$ has an infinite subset $U$ so that $L \gg X \oplus U$, then we can simply take $G = U$, $L_G = L$, and $P_G = P$ to satisfy the conclusion of \Cref{P_2_sep_iterate}. Going forward, we thus assume that neither $A^0$ nor $A^1$ has any such subset.

We will use different forcing notions to build each of the sets $G$, $L_G$, and $P_G$. The underlying sets of conditions will be called \emph{$G$-conditions}, \emph{$L_G$-conditions}, and \emph{$P_G$-conditions}, respectively, but when it is clear which kind of condition we are referring to we will simply use \emph{conditions} for short. All three sets will be constructed simultaneously. The forcing for $L_G$ will be defined in terms of that for $G$, and the forcing notion for $P_G$ will be defined in terms of that for $L_G$. We thus define the $G$-conditions, then the $L_G$-conditions, and finally the $P_G$-conditions, in order. The $G$-conditions will be be based on Mathias forcing, in the style originally developed by Hummel and Jockusch \cite{HJ-1999} and Cholak, Jockusch, and Slaman \cite{CJS-2001}, which is standard.

\subsubsection{Forcing notion for $G$}

We begin by defining the $G$-conditions and proving some of their properties.

\begin{definition}[$G$-conditions]
\
	\begin{enumerate}
		\item A \emph{$G$-condition} is a tuple $(D^0,D^1,Y)$ such that for each $i < 2$, $D^i$ is a finite subset of $A^i$, $Y$ is an infinite set with $\max D^i < \min Y$, and $L \gg X \oplus Y$.
		\item A condition $(E^0,E^1,Z)$ \emph{extends} $(D^0,D^1,Y)$ if for each $i < 2$, $D^i \subseteq E^i \subseteq D^i \cup Y$, and if $Z \subseteq Y$. If, in addition, $Y \setminus Z$ is finite then $(E^0,E^1,Z)$ is a \emph{finite extension} of $(D^0,D^1,Y)$.
		\item An \emph{index} for $(D^0,D^1,Y)$ is a number $\seq{d_0,d_1,y}$ such that for each $i < 2$, $d_i$ is a canonical index for $D^i$, and $\Phi^L_y$ is a DNC$_2$ function relative to $X \oplus Y$.
	\end{enumerate}
\end{definition}

\noindent Notice that $(\emptyset, \emptyset, \omega)$ is a condition.

Our proof will have the feature that the conditions we will be interested in will all be \emph{constructible}, which we make precise in \Cref{def:constructible} below. In short, this means that each condition will be obtained from another in one of two ways. The first way is finite extension, and the second way is implicit in the statement of \Cref{lem:Pi01_extension} below. The next two lemmas show that, if we extend conditions in these two ways, we can keep track of indices along the way.

\begin{lemma}\label{lem:finite_extension}
	Let $(E^0,E^1,Z)$ be a finite extension of $(D^0,D^1,Y)$. Then an index for $(E^0,E^1,Z)$ can be found uniformly computably from an index for $(D^0,D^1,Y)$ and canonical indices for $E^0$, $E^1$, and $Y \setminus Z$.
\end{lemma}

\begin{proof}
	From an index for $(D^0,D^1,Y)$, we first extract a $y \in \omega$ such that $\Phi^L_y$ is DNC$_2$ relative to $X \oplus Y$. From a canonical index for $Y \setminus Z$, we can uniformly find a computable function $h$ such that $\Phi^{X \oplus Y}_{h(e)}(x) = \Phi^{X \oplus Z}_e(e)$ for all $e$ and $x$. Now let $g$ be a computable function so that $\Phi^L_{g(e)}(x) = \Phi^L_e(h(x))$ for all $e$ and $x$. Then for all $e$ we have that
	\[
		\Phi^L_{g(y)}(e) = \Phi^L_y(h(e)) \neq \Phi_{h(e)}^{X \oplus Y}(h(e)) = \Phi_e^{X \oplus Z}(e),
	\]
	so $\Phi^L_{g(y)}$ is DNC$_2$ relative to $X \oplus Z$. Thus, the canonical indices of $E^0$ and $E^1$, together with $g(y)$, form an index for $(E^0,E^1,Z)$.
\end{proof}

\begin{lemma}\label{lem:Pi01_extension}
	Let $(D^0,D^1,Y)$ be a condition and let $\mathcal{C}$ be a non-empty $\Pi^0_1(X \oplus Y)$ class of $2$-partitions of $Y$. There exists $\seq{Y_0,Y_1} \in \mathcal{C}$ and an $i < 2$ such that $(D^0,D^1,Y_i)$ is an extension of $(D^0,D^1,Y)$, and such that this $i$ and an index for this extension can be found uniformly $P$-computably from an index for $(D^0,D^1,Y)$ and an index for $\mathcal{C}$ as a $\Pi^0_1(X \oplus Y)$ class.
\end{lemma}

\begin{proof}
	First, from an index for $(D^0,D^1,Y)$, fix $y \in \omega$ such that $\Phi^L_y$ is DNC$_2$ relative to $X \oplus Y$. Let $\mathcal{D}$ be the class of all $\seq{f_0,f_1,Y_0,Y_1}$ such that $\seq{Y_0,Y_1} \in \mathcal{C}$ and, for each $i < 2$, $f_i$ is DNC$_2$ over $X \oplus Y_i$. This is a non-empty $\Pi^0_1(X \oplus Y)$ class, whose index as such can be found uniformly computably from such an index for $\mathcal{C}$. Thus, $\Phi^L_y$ can uniformly compute an element $\seq{f_0,f_1,Y_0,Y_1} \in \mathcal{D}$. In particular, this produces indices $y_i$ for each $i < 2$ such that $\Phi_{y_i}^L = f_i$. Now, since $Y$ is infinite, there is at least one $i < 2$ such that $Y_i$ is infinite, and since $P \gg L'$ it follows that $P$ can uniformly compute such an $i$ (see \cite{DM-2022}, Theorem 2.8.25~(4)). Thus, the canonical indices for $D^0$ and $D^1$, together with $y_i$, form an index for $(D^0,D^1,Y_i)$.
\end{proof}

\noindent We think of \Cref{lem:Pi01_extension} as specifying a $P$-computable procedure that, given indices for $(D^0,D^1,Y)$ and $\mathcal{C}$, produces an index for the extension $(D^0,D^1,Y_i)$. In this case, we say that $(D^0,D^1,Y_i)$ is \emph{obtained from $(D^0,D^1,Y)$ and $\mathcal{C}$ via} \Cref{lem:Pi01_extension}.

\begin{definition}\label{def:constructible}
	A condition $(E^0,E^1,Z)$ is a \emph{constructible extension} of $(D^0,D^1,Y)$ if it is either a finite extension of $(D^0,D^1,Y)$, or else if there is a non-empty $\Pi^0_1(X \oplus Y)$ class $\mathcal{C}$ such that $(E^0,E^1,Z)$ is obtained from $(D^0,D^1,Y)$ and $\mathcal{C}$ via \Cref{lem:Pi01_extension}.
\end{definition}

We now add the following definition and basic facts about it, which we will use in our main construction.

\begin{definition}\label{force_G_inf}
	Let $(D^0,D^1,Y)$ be a condition, and fix $i < 2$ and $e \in \omega$. $(D^0,D^1,Y)$ \emph{forces $|\G| \geq e$ on the $i$ side} if $|D^i| \geq e$.
\end{definition}

\noindent It is easy to see that if $(D^0,D^1,Y)$ forces $|\G| \geq e$ on the $i$ side then so does every extension. The following lemma is obvious, but we include it here for easy reference later.

\begin{lemma}\label{lem:force_e_in_G_truth}
	Let $(D^0,D^1,Y)$ be any condition that, for some $i < 2$ and $e \in \omega$, forces $|\G| \geq e$ on the $i$ side. If $U$ is any set satisfying $D^i \subseteq U \subseteq D^i \cup Y$ then $|U| \geq e$.
\end{lemma}

\noindent The second lemma is a density fact that will allow us to build infinite subsets of both $A^0$ and $A^1$, although only one of these sets will end up working as our set $G$.

\begin{lemma}\label{G_infinite}
	Let $(D^0,D^1,Y)$ be a condition, and fix $i < 2$ and $e \in \omega$. Then there is a finite (hence constructible) extension of $(D^0,D^1,Y)$ that forces $|\G| \geq e$ on the $i$ side.
\end{lemma}

\begin{proof}
	By hypothesis, $A^{1-i}$ has no infinite subset $U$ such that $L \gg X \oplus U$. Since $L \gg X \oplus Y$, it follows that $Y \cap A^i$ is infinite. Hence, there exists $F^i \subseteq Y \cap A^i$ such that $|D^i \cup F^i| \geq e$. Let $E^i = D^i \cup F^i$, $E^{1-i} = D^{1-i}$, and let $Z = Y \setminus [0,\max F^0 \cup F^1]$. Then $(E^0,E^1,Z)$ is the desired finite extension.
\end{proof}

\subsubsection{Forcing notion for $L_G$}\label{sec:LG_conditions}

To define the $L_G$-conditions, we first need to define the following finite structures.

\begin{definition}\label{def:TFHJ}
	Let $F,H,J \subseteq \omega$ be finite.
	\begin{enumerate}
		\item $T_F$ is the set of all $\sigma \in 2^{<\omega}$ of length at most $\max F$ such that $\sigma(e) \neq \Phi^{X \oplus F}_{e,|\sigma|}(e)$ for every $e < |\sigma|$.
		\item $T_{F,H,J}$ is the set of all $\sigma \in T_F$ satisfying the following:
			\begin{itemize}
				\item for all $\seq{e,k} \in H$, $\neg (\exists\, x_0,x_1 \leq |\sigma|)[x_1 > x_0 > k \wedge \Phi_e^{\sigma}(x_0) \converges = \Phi_e^{\sigma}(x_1) \converges = 1]$;
				\item for all $e \in J$, $\Phi_e^\sigma(e) \diverges$.
			\end{itemize}
		\item $\sigma \in T_{F,H,J}$ has \emph{maximal length} if it has length $\max F$.
		\item $T_{F,H,J}$ \emph{looks extendible} if it contains at least one string of maximal length.
	\end{enumerate}
\end{definition}

Note that $T_F = T_{F,\emptyset,\emptyset}$. The following lemma collects some structural facts about this definition.

\begin{lemma}\label{lem:T_F_structural}
	\
	\begin{enumerate}
		\item For all finite $F,H,J \subseteq \omega$, $T_{F,H,J}$ is a tree.
		\item If $F_0,F_1,H,J \subseteq \omega$ are finite, with $F_0 \subseteq F_1$ and $\min F_1 \setminus F_0 > \max F_0$, then $T_{F_0,H,J} \subseteq T_{F_1,H,J}$ and $T_{F_1,H,J}$ and $T_{F_0,H,J}$ contain the same strings of length at most $\max F_0$.
		\item If $F,H_0,H_1,J_0,J_1 \subseteq \omega$ are finite, with $H_0 \subseteq H_1$ and $J_0 \subseteq J_1$, then $T_{F,H_1,J_1} \subseteq T_{F,H_0,J_0}$.
		\item Suppose $F_0 \subseteq F_1 \subseteq \cdots$ are finite, with $\min F_{n+1} \setminus F_n > \max F_n$ for all $n$ and $\lim_n \max F_n = \infty$. Let $U = \bigcup_n F_n$. For all finite $H,J \subseteq \omega$ such that $T_{F_n,H,J}$ looks extendible for all $n$, we have that $\bigcup_{n} T_{F_n,H,J}$ is an infinite subtree of $T^2_{X \oplus U}$.
	\end{enumerate}
\end{lemma}

\begin{proof}
	Parts (1), (2), and (3) are clear from the definition.	For (4), we begin by observing that $\bigcup_n T_{F_n,\emptyset,\emptyset} = T^2_{X \oplus U}$. Hence, $\bigcup_n T_{F_n,H,J} \subseteq T^2_{X \oplus U}$ by (3). Since $\lim_n \max F_n = \infty$ and each $T_{F_n,H,J}$ contains a string of length $\max F_n$, it follows that $\bigcup_n T_{F_n,H,J}$ is infinite.
\end{proof}

We now define the conditions we will use to build the set $L_G$.

\begin{definition}[$L_G$-conditions]\
	\begin{enumerate}
		\item An \emph{$L_G$-condition} is a tuple $(D^0,D^1,Y,H^0,J^0,H^1,J^1)$ such that $(D^0,D^1,Y)$ is a $G$-condition, and for each $i < 2$, $H^i$ and $J^i$ are finite subsets of $\omega$ such that $T_{D^i \cup F,H^i,J^i}$ looks extendible for every finite set $F \subseteq Y$.
		\item A condition $(E^0,E^1,Z,I^0,K^0,I^1,K^1)$ \emph{extends} $(D^0,D^1,Y,H^0,J^0,H^1,J^1)$ if $(E^0,E^1,Z)$ extends $(D^0,D^1,Y)$ as $G$-conditions, and if for each $i < 2$, $H^i \subseteq I^i$ and $J^i \subseteq K^i$. It is a \emph{constructible} extension if $(E^0,E^1,Z)$ is a constructible extension of $(D^0,D^1,Y)$ as $G$-conditions.
	\end{enumerate}
\end{definition}

\noindent Notice that if $(D^0,D^1,Y)$ is any $G$-condition then $(D^0,D^1,Y,\emptyset,\emptyset,\emptyset,\emptyset)$ is an $L_G$-condition, so the above notion is non-trivial. We also have the following compatibility fact with respect to $G$-conditions.

\begin{lemma}\label{G_to_LG_extensions}
	Let $p = (D^0,D^1,Y,H^0,J^0,H^1,J^1)$ be a condition. If $(E^0,E^1,Z)$ is a (constructible) extension of $(D^0,D^1,Y)$ as $G$-conditions, then the tuple\break$(E^0,E^1,Z,H^0,J^0,H^1,J^1)$ is a (constructible) extension of $p$.
\end{lemma}

\begin{proof}
	We have only to verify that $(E^0,E^1,Z,H^0,J^0,H^1,J^1)$ is a condition. To this end, fix $i < 2$ and a finite set $F \subseteq Z$. We must show that $T_{E^i \cup F,H^i,J^i}$ looks extendible. By definition of $G$-extension, we have that $E^i = D^i \cup F'$ for some finite set $F' \subseteq Y$. Let $F'' = F' \cup F$. Since $Z \subseteq Y$, we have that $F'' \subseteq Y$. Hence, because $p$ is a condition, we have that $T_{D^i \cup F'',H^i,J^i}$ looks extendible. But $T_{D^i \cup F'',H^i,J^i} = T_{E^i \cup F,H^i,J^i}$.
\end{proof}

\begin{definition}\label{def:L_G_forcing_types}
	Let $p = (D^0,D^1,Y,H^0,J^0,H^1,J^1)$ be a condition, and fix $i < 2$ and $e \in \omega$.
	\begin{enumerate}
		\item $p$ \emph{forces $e \in L_\G'$ on the $i$ side} if $T_{D^i,H^i,J^i \cup \set{e}}$ does not look extendible.
		\item $p$ \emph{forces $e \notin L_\G'$ on the $i$ side} if $e \in J^i$.
		\item $p$ \emph{forces $\Phi_e^{L_\G}$ is not a non-hyperimmunity witness for $S$ on the $i$ side} if for some $k \in \omega$ and every $\sigma \in T_{D^i,H^i,J^i}$ of maximal length there exist $x_1 > x_0 > k$ such that $\Phi^\sigma_e(x_0) \converges = \Phi^\sigma_e(x_1) \converges = 1$ and $[x_0,x_1) \subseteq \overline{S}$.
		\item $p$ \emph{forces $\Phi_e^{L_\G}$ is finite on the $i$ side} if $\seq{e,k} \in H^i$ for some $k \in \omega$.
	\end{enumerate}
\end{definition}

\noindent The key properties of this definition for the purposes of our main construction are the following.

\begin{lemma}\label{lem:L_G_forcing_monotone}
	Let $p = (D^0,D^1,Y,H^0,J^0,H^1,J^1)$ be a condition forcing any of the statements in \Cref{def:L_G_forcing_types}. If $q = (E^0,E^1,Z,I^0,K^0,I^1,K^1)$ is any extension of $p$ then it forces the same statement.
\end{lemma}

\begin{proof}
	For parts (2) and (4) of \Cref{def:L_G_forcing_types}, the result is immediate from the definition of extension. For part (1), suppose $p$ forces $e \in L_\G'$ on the $i$ side. By \Cref{lem:T_F_structural}~(2), $T_{E^i,I^i,K^i \cup \set{e}}$ agrees with $T_{D^i,I^i,K^i \cup \set{e}}$ on the strings of length at most $\max D_i$, and by \Cref{lem:T_F_structural}~(3), $T_{D^i,I^i,K^i \cup \set{e}} \subseteq T_{D^i,H^i,J^i \cup \set{e}}$. Now since $T_{E^i,I^i,K^i \cup \set{e}}$ is a tree, if it looked extendible (meaning, if it contained a string of length $\max E^i \geq \max D^i$) then it would in particular contain a string of length $\max D^i$. But then $T_{D^i,I^i,K^i \cup \set{e}}$ would contain such a string, and hence so would $T_{D^i,H^i,J^i \cup \set{e}}$ and therefore look extendible. It follows that $T_{E^i,I^i,K^i \cup \set{e}}$ does not look extendible, so $q$ forces $e \in L_\G'$ on the $i$ side. For part (3) of \Cref{def:L_G_forcing_types}, the proof is similar.
\end{proof}

\begin{lemma}\label{lem:force_e_in_LG_truth}
	Let $p = (D^0,D^1,Y,H^0,J^0,H^1,J^1)$ be a condition, and fix $i < 2$ and $e \in \omega$. Suppose $f \in 2^{\omega}$ satisfies $f \res \max D^i \in T_{D^i,H^i,J^i}$.
	\begin{enumerate}
		\item If $p$ forces $e \in L_\G'$ on the $i$ side then $e \in f'$.
		\item If $p$ forces $e \notin L_\G'$ on the $i$ side then $\Phi_e^{f \res \max D_i}(e) \diverges$.
		\item If $p$ forces $\Phi_e^{L_\G}$ is not a non-hyperimmunity witness for $S$ on the $i$ side then there exist $x_1 > x_0$ such that $\Phi^f_e(x_0) \converges = \Phi^f_e(x_1) \converges = 1$ and $[x_0,x_1) \subseteq \overline{S}$.
		\item If $p$ forces $\Phi_e^{L_\G}$ is finite on the $i$ side then for any $k$ such that $\seq{e,k} \in H^i$ we have that $\neg (\exists\, x_0,x_1 \leq \max D^i)[x_1 > x_0 > k \wedge \Phi_e^{f \res \max D^i}(x_0) \converges = \Phi_e^{f \res \max D^i}(x_1) \converges = 1]$.
	\end{enumerate}
\end{lemma}

\begin{proof}
	All parts follow directly from the definitions.
\end{proof}

We now collect several density facts about $L_G$-conditions that we will use in our main construction.

\begin{lemma}\label{lem:J0orJ1}
	Let $(D^0,D^1,Y,H^0,J^0,H^1,J^1)$ be a condition, and fix $e_0,e_1 \in \omega$. There exists a constructible extension of $(D^0,D^1,Y,H^0,J^0,H^1,J^1)$ that forces one of the following:
	\begin{enumerate}
		\item $e_0 \in L_\G'$ on the $0$ side;
		\item $e_0 \notin L_\G'$ on the $0$ side;
		\item $e_1 \in L_\G'$ on the $1$ side;
		\item $e_1 \notin L_\G'$ on the $1$ side.
	\end{enumerate}
\end{lemma}

\begin{proof}
	Let $\mathcal{C}$ be the class of all $2$-partitions $\seq{Y_0,Y_1}$ of $Y$ such that
	\[
		(\forall i < 2)(\forall F \subseteq Y_i \text{ finite})[T_{D^i \cup F,H^i,J^i \cup \set{e_i}} \text{ looks extendible}].
	\]
	Then $\mathcal{C}$ is a $\Pi^0_1(X \oplus Y)$ class, so if it is non-empty we can apply \Cref{lem:Pi01_extension} to find a $\seq{Y_0,Y_1} \in \mathcal{C}$ and an $i < 2$ such that $(D^0,D^1,Y_i)$ is a constructible extension of $(D^0,D^1,Y)$. Then the $L_G$-condition obtained from $(D^0,D^1,Y,H^0,J^0,H^1,J^1)$ by replacing $Y$ with $Y_i$ and $J^i$ with $J^i \cup \set{e_i}$ is a constructible extension of $(D^0,D^1,Y,H^0,J^0,H^1,J^1)$ that forces $e_i \notin L_\G'$ on the $i$ side.
	
	Suppose, then, that $\mathcal{C}$ is empty. By compactness of Cantor space, there is an $\ell \in \omega$ such that for every $2$-partition $\seq{Y_0,Y_1}$ of $Y$ there exists an $i < 2$ and an $F \subseteq Y_i \res \ell$ for which $T_{D^i \cup F,H^i,J^i \cup \set{e_i}}$ does not look extendible. Fix such an $i$ and $F$ for the $2$-partition $\seq{Y \cap A^0,Y \cap A^1}$. Then the condition obtained from $(D^0,D^1,Y,H^0,J^0,H^1,J^1)$ by replacing $D^i$ with $D^i \cup F$ and $Y$ by $Y \setminus [0,\max F]$ is a finite extension of $(D^0,D^1,Y,H^0,J^0,H^1,J^1)$ that forces $e_i \in L_\G'$ on the $i$ side.
\end{proof}

\begin{lemma}\label{lem:H0orH1}
	Let $(D^0,D^1,Y,H^0,J^0,H^1,J^1)$ be a condition, and fix $e_0,e_1 \in \omega$. There exists a constructible extension of $(D^0,D^1,Y,H^0,J^0,H^1,J^1)$ that forces one of the following:
	\begin{enumerate}
		\item $\Phi_{e_0}^{L_\G}$ is not a non-hyperimmunity witness for $S$ on the $0$ side;
		\item $\Phi_{e_0}^{L_\G}$ is finite on the $0$ side;
		\item $\Phi_{e_1}^{L_\G}$ is not a non-hyperimmunity witness for $S$ on the $1$ side;
		\item $\Phi_{e_1}^{L_\G}$ is finite on the $1$ side.
	\end{enumerate}
\end{lemma}

\begin{proof}
	For each $k$, let $\mathcal{C}_k$ be the class of all $2$-partitions $\seq{Y_0,Y_1}$ of $Y$ such that
	\[
		(\forall i < 2)(\forall F \subseteq Y_i \text{ finite})[T_{D^i \cup F,H^i \cup \set{\seq{e_i,k}},J^i} \text{ looks extendible}]
	\]
	which is a $\Pi^0_1(X \oplus Y)$ class. If there is a $k$ for which $\mathcal{C}_k$ is non-empty we apply \Cref{lem:Pi01_extension} to find a $\seq{Y_0,Y_1} \in \mathcal{C}_k$ and an $i < 2$ such that $(D^0,D^1,Y_i)$ is a constructible extension of $(D^0,D^1,Y)$. The $L_G$-condition obtained by replacing $Y$ with $Y_i$ and $H^i$ with $H^i \cup \set{\seq{e_i,k}}$ is then a constructible extension of $(D^0,D^1,Y,H^0,J^0,H^1,J^1)$ that forces $\Phi_{e_i}^{L_\G}$ is finite on the $i$ side.
	
	Suppose next that $\mathcal{C}_k$ is empty for every $k$. This means that for every $k$, there is an $\ell$ such that for every $2$-partition $\seq{Y_0,Y_1}$ of $Y$ there is an $i < 2$ and a finite set $F \subseteq Y_i \res \ell$ such that $T_{D^i \cup F,H^i \cup \set{\seq{e_i,k}},J^i}$ does not look extendible. Let $\ell_k$ be the least such $\ell$. Since $(D^0,D^1,Y,H^0,J^0,H^1,J^1)$ is a condition and $F \subseteq Y_i \subseteq Y$, the tree $T_{D^1 \cup F,H^1,J^1}$ looks extendible, and so contains a string $\sigma$ of maximal length. Since no such $\sigma$ belongs to $T_{D^i \cup F,H^i \cup \set{\seq{e_i,k}},J^i}$, it follows from the definition that there exist $x_0,x_1 \leq |\sigma|$ such that $x_1 > x_0 > k$ and $\Phi_{e_i}^{\sigma}(x_0) \converges = \Phi_{e_1}^{\sigma}(x_1) \converges = 1$. Let $x^k_0$ be the minimum of all such $x_0$, and $x^k_1$ the maximum of all the such $x_1$, across both $i < 2$, all finite sets $F \subseteq Y_i \res \ell_k$ as above, and all $\sigma \in T_{D^i \cup F,H^i,J^i}$ of maximal length.
	
	Now, $\ell_k$ can be found uniformly $(X \oplus Y)$-computably from $k$, and hence so can the finite collection of all the finite sets $F$ above. It follows that $X \oplus Y$ can uniformly compute the numbers $x^k_0$ and $x^k_1$ for each $k$. Since $S$ is $L$-hyperimmune, and since $L \Tabove X \oplus Y$ and $k < x^k_0 < x^k_1$ for all $k$, there must be a $k$ such that $[x^k_0,x^k_1) \subseteq \overline{S}$. Fix this $k$, and fix $i < 2$ and $F \subseteq (Y \cap A^i) \res \ell_k$ such that $T_{D^i \cup F,H^i \cup \set{\seq{e_i,k}},J^i}$ does not look extendible. Then the condition obtained from $(D^0,D^1,Y,H^0,J^0,H^1,J^1)$ by replacing $D^i$ by $D^i \cup F$ and $Y$ by $Y \setminus [0,\max F]$ is a finite extension of $(D^0,D^1,Y,H^0,J^0,H^1,J^1)$ that forces $\Phi_{e_i}^{L_\G}$ is not a non-hyperimmunity witness for $S$ on the $i$ side.
\end{proof}

\begin{lemma}\label{lem:J0orH1}
	Let $(D^0,D^1,Y,H^0,J^0,H^1,J^1)$ be a condition, and fix $e_0,e_1 \in \omega$. There exists a constructible extension of $(D^0,D^1,Y,H^0,J^0,H^1,J^1)$ that forces one of the following:
	\begin{enumerate}
		\item $e_0 \in L_\G'$ on the $0$ side;
		\item $e_0 \notin L_\G'$ on the $0$ side;
		\item $\Phi_{e_1}^{L_\G}$ is not a non-hyperimmunity witness for $S$ on the $1$ side;
		\item $\Phi_{e_1}^{L_\G}$ is finite on the $1$ side.
	\end{enumerate}
\end{lemma}

\begin{proof}
	For each $k$, let $\mathcal{C}_k$ be the class of all $2$-partitions $\seq{Y_0,Y_1}$ of $Y$ such that
	\[
		(\forall F \subseteq Y_0 \text{ finite})[T_{D^0 \cup F,H^0,J^0 \cup \set{e_0}} \text{ looks extendible}]
	\]
	and
	\[
		(\forall F \subseteq Y_1 \text{ finite})[T_{D^1 \cup F,H^1 \cup \set{\seq{e_1,k}},J^1} \text{ looks extendible}],
	\]
	which is a $\Pi^0_1(X \oplus Y)$ class. If there is a $k$ for which $\mathcal{C}_k$ is non-empty we apply \Cref{lem:Pi01_extension} to find a $\seq{Y_0,Y_1} \in \mathcal{C}_k$ and an $i < 2$ such that $(D^0,D^1,Y_i)$ is a constructible extension of $(D^0,D^1,Y)$. The $L_G$-condition obtained by replacing $Y$ with $Y_i$ and either $J^0$ with $J^0 \cup \set{e_0}$ if $i = 0$, or $H^1$ with $H^1 \cup \set{\seq{e_1,k}}$ if $i = 1$, is then a constructible extension of $(D^0,D^1,Y,H^0,J^0,H^1,J^1)$. If $i = 0$ this extension forces $e_0 \notin L_\G'$ on the $0$ side, and if $i = 1$ it forces $\Phi_{e_1}^{L_\G}$ is finite on the $1$ side.
	
	Now suppose $\mathcal{C}_k$ is empty for every $k$. This means there is an $\ell_k$ such that for every $2$-partition $\seq{Y_0,Y_1}$ of $Y$, one of the following holds:
	\begin{enumerate}
		\item[(a)] there is a finite set $F \subseteq Y_0 \res \ell_k$ such that $T_{D^0 \cup F,H^0,J^0 \cup \set{e_0}}$ does not look extendible;
		\item[(b)] there is a finite set $F \subseteq Y_1 \res \ell_k$ such that $T_{D^1 \cup F,H^1 \cup \set{\seq{e_1,k}},J^1}$ does not look extendible.
	\end{enumerate}
	In particular, for each $k$, (a) or (b) holds for the $2$-partition $\seq{Y \cap A^0,Y \cap A^1}$. If there is a $k$ such that (a) holds for $\seq{Y \cap A^0,Y \cap A^1}$, we can fix the corresponding $F \subseteq Y_0 \res \ell$ and proceed as in the proof of \Cref{lem:J0orJ1} to obtain a finite extension of $(D^0,D^1,Y,H^0,J^0,H^1,J^1)$ that forces $e_0 \in L_\G'$ on the $0$ side. If (b) holds for every $k$, then we can instead proceed as in the proof of \Cref{lem:H0orH1} and obtain a finite extension of $(D^0,D^1,Y,H^0,J^0,H^1,J^1)$ that forces $L^\G_{e_1}$ is not a non-hyperimmunity witness for $S$ on the $1$ side.
\end{proof}

\begin{lemma}\label{lem:H0orJ1}
	Let $(D^0,D^1,Y,H^0,J^0,H^1,J^1)$ be a condition, and fix $e_0,e_1 \in \omega$. There exists a constructible extension of $(D^0,D^1,Y,H^0,J^0,H^1,J^1)$ that forces one of the following:
	\begin{enumerate}
		\item $\Phi_{e_0}^{L_G}$ is not a non-hyperimmunity witness for $S$ on the $0$ side;
		\item $\Phi_{e_0}^{L_G}$ is finite on the $0$ side;
		\item $e_1 \in L_G'$ on the $1$ side;
		\item $e_1 \notin L_G'$ on the $1$ side.
	\end{enumerate}
\end{lemma}

\begin{proof}
	The proof is symmetric to that of \Cref{lem:J0orH1}.
\end{proof}

\subsubsection{Forcing notion for $P_G$.}\label{sec:P_G_conditions}

We now come to our final forcing notion, which will be used to construct the set $P_G$. Here, we first define a notion of \emph{precondition}, in terms of which we then define the actual conditions.

\begin{definition}[$P_G$-preconditions]
	\
	\begin{enumerate}
		\item A \emph{$P_G$-precondition} is a tuple $(p,b^0,b^1)$ such that $p$ is an $L_G$-condition, $b^0,b^1 \in \omega$, and for each $i < 2$ and each $e < b^i$, $p$ either forces $e \in L_\G'$ or $e \notin L_\G'$ on the $i$ side.
		\item A precondition $(q,c^0,c^1)$ \emph{extends} $(p,b^0,b^1)$ if $q$ extends $p$ as $L_G$-conditions, and if for each $i < 2$, $c^i \geq b^i$. It is a \emph{constructible} extension if $q$ is a constructible extension of $p$ as $L_G$-conditions.
	\end{enumerate}
\end{definition}

\noindent Notice that if $p$ is any $L_G$-condition then $(p,0,0)$ is a $P_G$-precondition. For precondition $(p,b^0,b^1)$ and each $i < 2$, let $F^i_{p,b^i} = \set{e < b^i : p \text{ forces } e \in L_\G' \text{ on the } i \text{ side}}$. Given a finite set $V$, recall that $T_{F^i_{p,b^i},V,\emptyset}$ is the tree of all $\sigma \in T_{F^i_{p,b^i}}$ such that for all $\seq{e,k} \in V$,  $\neg (\exists\, x_0,x_1 \leq |\sigma|)[x_1 > x_0 > k \wedge \Phi_e^{\sigma}(x_0) \converges = \Phi_e^{\sigma}(x_1) \converges = 1]$.

\begin{definition}[$P_G$-conditions]\
	\begin{enumerate}
		\item A \emph{$P_G$-condition} is a tuple $(p,b^0,b^1,V^0,V^1)$ such that $(p,b^0,b^1)$ is a $P_G$-precondition and for each $i < 2$, $V^i$ is a finite set for which there is no sequence of preconditions
		\[
			(p,b^0,b^1) = (p_0,b^0_0,b^0_0),\ldots,(p_n,b^0_n,b^1_n)
		\]
		such that $(p_{m+1},b^0_{m+1},b^1_{m+1})$ is a constructible extension of $(p_m,b^0_m,b^1_m)$ as preconditions for all $m < n$ and $T_{F^i_{p_n,b_n^i},V^i,\emptyset}$ does not look extendible.
		\item A condition $(q,c^0,c^1,W^0,W^1)$ \emph{extends} $(p,b^0,b^1,V^0,V^1)$ if $q$ extends $p$ as $L_G$-conditions, and if for each $i < 2$, $c^i \geq b^i$ and $V^i \subseteq W^i$. It is a \emph{constructible} extension if $q$ is a constructible extension of $p$ as $L_G$-conditions.
	\end{enumerate}
\end{definition}

\noindent We have the follow compatibility lemma with respect to $L_G$-conditions.

\begin{lemma}\label{LG_to_PG_extensions}
	Let $(p,b^0,b^1,V^0,V^1)$ be a condition. If $q$ is a constructible extension of $p$ as $L_G$-conditions, and if $c^0 \geq b^0$ and $c^1 \geq b^1$ are such that $(q,c^0,c^1)$ is a precondition, then the tuple $(q,c^0,c^1,V^0,V^1)$ is a constructible extension of $(p,b^0,b^1,V^0,V^1)$.
\end{lemma}

\begin{proof}
	We have only to verify that $(q,c^0,c^1,V^0,V^1)$ is a condition. Suppose not. Since $(q,c^0,c^1)$ is a precondition, this means that there is an $i < 2$ and a sequence of preconditions
	\[
		(q,c^0_0,c^0_0) = (p_0,b^0_0,b^0_0),\ldots,(p_n,b^0_n,b^1_n)
	\]
	such that $(p_{m+1},b^0_{m+1},b^1_{m+1})$ is a constructible extension of $(p_m,b^0_m,b^1_m)$ for all $m < n$ and $T_{F^i_{p_n,b^i_n},V^i,\emptyset}$ does not look extendible. Since $q$ is a constructible extension of $p$, it follows that $(q,c^0_0,c^0_0)$ is a constructible extension of $(p,b^0,b^1)$. But this means we can prepend $(p,b^0,b^1)$ to the above sequence to get one witnessing that $(p,b^0,b^1,V^0,V^1)$ is not a condition, a contradiction.
\end{proof}

\begin{definition}\label{def:P_G_forcing_types}
	Let $(p,b^0,b^1,V^0,V^1)$ be a condition, and fix $i < 2$ and $e \in \omega$.
	\begin{enumerate}
		\item $(p,b^0,b^1,V^0,V^1)$ \emph{forces $\Phi^{P_\G}_e$ is not a non-hyperimmunity witness for $\overline{S}$ on the $i$ side} if for some $k \in \omega$ and every $\sigma \in T_{F^i_{p,b^i},V^i,\emptyset}$ of maximal length there exist $x_1 > x_0 > k$ such that $\Phi^\sigma_e(x_0) \converges = \Phi^\sigma_e(x_1) \converges = 1$ and $[x_0,x_1) \subseteq S$.
		\item $(p,b^0,b^1,V^0,V^1)$ \emph{forces $\Phi^{P_\G}_e$ is finite on the $i$ side} if $\seq{e,k} \in V^i$ for some $k \in \omega$.
	\end{enumerate}
\end{definition}

The following lemmas are analogues for $P_G$-conditions of \Cref{lem:L_G_forcing_monotone,lem:force_e_in_LG_truth}, and are proved similarly.

\begin{lemma}\label{lem:P_G_forcing_monotone}
	Let $(p,b^0,b^1,V^0,V^1)$ be a condition forcing any of the statements in \Cref{def:P_G_forcing_types}. If $(q,c^0,c^1,W^0,W^1)$ is any extension of $(p,b^0,b^1,V^0,V^1)$ then it forces the same statement.
\end{lemma}

\begin{lemma}\label{lem:force_e_in_PG_truth}
	Let $(p,b^0,b^1,V^0,V^1)$ be a condition, and fix $i < 2$ and $e \in \omega$. Suppose $f \in 2^{\omega}$ satisfies $f \res \max F^i_{p,b^i} \in T_{F^i_{p,b^i},V^i,\emptyset}$.
	\begin{enumerate}
		\item If $(p,b^0,b^1,V^0,V^1)$ forces $\Phi_e^{P_\G}$ is not a non-hyperimmunity witness for $\overline{S}$ on the $i$ side then there exist $x_1 > x_0$ such that $\Phi^f_e(x_0) \converges = \Phi^f_e(x_1) \converges = 1$ and $[x_0,x_1) \subseteq \overline{S}$.
		\item If $(p,b^0,b^1,V^0,V^1)$ forces $\Phi_e^{P_\G}$ is finite on the $i$ side then for any $k$ such that $\seq{e,k} \in V^i$ we have that $\neg (\exists\, x_0,x_1 \leq \max D^i)[x_1 > x_0 > k \wedge \Phi_e^{f \res \max F^i_{p,b^i}}(x_0) \converges = \Phi_e^{f \res \max F^i_{p,b^i}}(x_1) \converges = 1]$.
	\end{enumerate}
\end{lemma}

We now have our final density fact.

\begin{lemma}\label{P_G_density}
	Let $(p,b^0,b^1,V^0,V^1)$ be a condition, and fix $i < 2$ and $e \in \omega$. Then one of the following holds:
	\begin{enumerate}
		\item there is a constructible extension of $(p,b^0,b^1,V^0,V^1)$ that forces $\Phi^{P_\G}_e$ is finite on the $i$ side;
		\item there is a sequence of conditions
		\[
			(p,b^0,b^1,V^0,V^1) = (p_0,b^0_0,b^1_0,V^0_0,V^1_0),\ldots,(p_n,b^0_n,b^1_n,V^0_n,V^1_n)
		\]
		in which each condition is a constructible extension of the previous and such that $(p_n,b^0_n,b^1_n,V^0_n,V^1_n)$ forces $\Phi^{P_\G}_e$ is not a non-hyperimmunity witness for $\overline{S}$ on the $i$ side.
	\end{enumerate}
\end{lemma}

\begin{proof}
	Suppose (1) does not hold. Then in particular, for every $k$, replacing $V^i$ by $V^i \cup \set{\seq{e,k}}$ in $(p,b^0,b^1,V^0,V^1)$ does not produce a condition. This means there is a sequence of preconditions
	\[
		(p,b^0,b^1) = (p_0,b^0_0,b^1_0),\ldots,(p_n,b^0_n,b^1_n)
	\]
	in which each condition is a constructible extension of the previous and for every $\sigma \in T_{F^i_{p_n,b^i_n},V^i,\emptyset}$ of maximal length there exist $x_1 > x_0 > k$ such that $\Phi^\sigma_e(x_0) \converges = \Phi^\sigma_e(x_1) \converges = 1$.
	
	Note that since $P \gg L'$ and $L \gg X \oplus Y$, it is possible for $P$ to search through all indices of infinite $\Pi^0_1(X \oplus Y)$ classes. (For this, it suffices just that $P \Tabove (X \oplus Y)'$.) Hence, given $k$, it follows by \Cref{lem:finite_extension,lem:Pi01_extension} that $P$ can uniformly search for a sequence of preconditions as above. By assumption, this search must always succeed. Let $(p_n,b^0_n,b^1_n)$ be the last condition in this sequence. Then $P$ can find all the $x_1 > x_0 > k$ as above, across all $\sigma \in T_{F^i_{p_n,b^i_n},V^i,\emptyset}$, and in particular can find the minimum $x^k_0$ of all such $x_0$, and the maximum $x^k_1$ of all such $x_1$. Thus, $x^k_1 > x^k_0 > k$.
	
	Since $\overline{S}$ is $P$-hyperimmune, there must be a $k$ such that $[x^k_0,x^k_1) \subseteq S$. Fix this $k$, and fix the corresponding sequence $(p,b^0,b^1) = (p_0,b^0_0,b^1_0),\ldots,(p_n,b^0_n,b^1_n)$ found by $P$ in the above search. Then $(p_n,b^0_n,b^1_n,V^0,V^1)$ forces $\Phi^{P_\G}_e$ is not a non-hyperimmunity witness for $\overline{S}$ on the $i$ side, so we can take our sequence witnessing (2) to be $(p_0,b^0_0,b^1_0,V^0,V^1),\ldots,(p_n,b^0_n,b^1_n,V^0,V^1)$.
\end{proof}

\subsubsection{Putting it all together}

This brings us to the main construction.

\begin{proof}[Proof of \Cref{P_2_sep_iterate}] We begin by building a sequence of $P_G$-conditions
\[
	(p_0,b^0_0,b^1_0,V^0_0,V^1_0),(p_1,b^0_1,b^1_1,V^0_1,V^1_1),\ldots
\]
in which each condition is a constructible extension of the previous. This sequence will be suitably generic, as we describe below. For each $m$, write $p_m = (D^0_m,D^1_m,Y_m,H^0_m,J^0_m,H^1_m,J^1_m)$.

\subsubsection*{Construction of generic sequence.} We proceed in stages. At stage $s$, we define $(p_m,b^0_m,b^1_m,V^0_m,V^1_m)$ for all $m \leq n_s$.

At stage $0$, let $n_0 = 0$, let $p_0 = (\emptyset,\emptyset,\omega,\emptyset,\emptyset,\emptyset,\emptyset)$, and let $b^0_0 = b^1_0 = 0$ and $V^0_0 = V^1_0 = \emptyset$.

Now suppose we are at stage $s + 1$, so that $(p_{n_s},b^0_{n_s},b^1_{n_s},V^0_{n_s},V^1_{n_s})$ is defined. We define $n_{s+1}$, and $(p_m,b^0_m,b^1_m,V^0_m,V^1_m)$ for all $n_s+1 \leq m \leq n_{s+1}$. To do this, we break into cases based on the congruence class of $s$ modulo $8$.
\begin{itemize}
	\item If $s = 8e+i$ for some $e \in \omega$ and $i < 2$, let $n_{s+1} = n_s + 1$, and let $(D^0_{n_{s+1}},D^1_{n_{s+1}},Y_{n_{s+1}})$ be the condition obtained by applying \Cref{G_infinite} with $(D^0,D^1,Y) = (D^0_{n_{s}},D^1_{n_{s}},Y_{n_{s}})$, $i$, and $e$. Let $H^0_{n_{s+1}} = H^0_{n_s}$, $J^0_{n_{s+1}} = J^0_{n_s}$, $H^1_{n_{s+1}} = H^1_{n_s}$, and $J^1_{n_{s+1}} = J^1_{n_s}$. This defines $p_{n_{s+1}}$.
	\item If $s = 8\seq{e_0,e_1} + 2$, let $n_{s+1} = n_s + 1$, and let $p_{n_{s+1}}$ be the condition obtained by applying \Cref{lem:J0orJ1} with $(D^0,D^1,Y,H^0,J^0,H^1,J^1) = p_{n_s}$ and $e_0,e_1$. 
	\item If $s = 8\seq{e_0,e_1} + 3$, let $n_{s+1} = n_s + 1$, and let $p_{n_{s+1}}$ be the condition obtained by applying \Cref{lem:H0orH1} with $(D^0,D^1,Y,H^0,J^0,H^1,J^1) = p_{n_s}$ and $e_0,e_1$.
	\item If $s = 8\seq{e_0,e_1} + 4$, let $n_{s+1} = n_s + 1$, and let $p_{n_{s+1}}$ be the condition obtained by applying \Cref{lem:J0orH1} with $(D^0,D^1,Y,H^0,J^0,H^1,J^1) = p_{n_s}$ and $e_0,e_1$.
	\item If $s = 8\seq{e_0,e_1} + 5$, let $n_{s+1} = n_s + 1$, and let $p_{n_{s+1}}$ be the condition obtained by applying \Cref{lem:H0orJ1} with $(D^0,D^1,Y,H^0,J^0,H^1,J^1) = p_{n_s}$ and $e_0,e_1$.
\end{itemize}
In each of the above cases, for each $i < 2$, if $p_{n_{s+1}}$ forces $b^i_{n_s} \in L_\G'$ or $b^i_{n_s} \notin L_\G'$ on the $i$ side, let $b^i_{n_{s+1}} = b^i_{n_s}+1$, and otherwise let $b^i_{n_{s+1}} = b^i_{n_s}$. Let $V^i_{n_{s+1}} = V^i_{n_s}$.
\begin{itemize}
	\item If $s = 8e+6+i$ for some $e \in \omega$ and $i < 2$, apply \Cref{P_G_density} with $(p,b^0,b^1,V^0,V^1) = (p_{n_s},b^0_{n_s},b^1_{n_s},V^0_{n_s},V^1_{n_s})$, $i$, and $e$. If we are in case (1) of the lemma, then we obtain a single constructible extension. We set $n_{s+1} = n_s+1$, and let $(p_{n_{s+1}} ,b^0_{n_{s+1}},b^1_{n_{s+1}},V^0_{n_{s+1}},V^1_{n_{s+1}})$ be this extension. If we are in case (2), we obtain instead a sequence of conditions, each a constructible extension of the previous, beginning with $(p_{n_s},b^0_{n_s},b^1_{n_s},V^0_{n_s},V^1_{n_s})$. Say this sequence has length $n+1$. We set $n_{s+1}=n_s + n$, and for each $m < n$, let $(p_{n_s+m},b^0_{n_s+m},b^1_{n_s+m},V^0_{n_s+m},V^1_{n_s+m})$ be the $(m+1)$-st condition in the sequence.
\end{itemize}
For the stages $s \equiv i \mod 8$, $i < 2$, the fact that we obtain a constructible extension follows by \Cref{G_to_LG_extensions,LG_to_PG_extensions}. For the stages $s \equiv i \mod 8$, $2 \leq i \leq 5$, it follows by  \Cref{LG_to_PG_extensions}. For the other stages it is clear.

This completes the construction.

\begin{lemma}\label{lem:pick_a_side}
	There is an $i < 2$ such that for every $e \in \omega$ the following hold:
	\begin{enumerate}
		\item there is an $m \in \omega$ such that $p_m$ forces $e \in L_\G'$ or $e \notin L_\G'$ on the $i$ side;
		\item there is an $m \in \omega$ such that $p_m$ forces $\Phi^{L_\G}_e$ is not a non-hyperimmunity witness for $S$, or $\Phi^{L_\G}_e$ is finite, on the $i$ side.
	\end{enumerate}
\end{lemma}

\begin{proof}
	Suppose the result is false for $i = 0$.	 First, suppose that there is an $e_0$ such that no $p_m$ forces $e_0 \in L_\G'$ or $e_0 \notin L_\G'$ on the $0$ side. Fix $e \in \omega$. Then the condition $p_{n_s}$ defined at stage $s = 8\seq{e_0,e}+2$ (which appeals to \Cref{lem:J0orJ1} with $e_0,e$) must force $e \in L_\G'$ or $e \notin L_\G'$ on the $1$ side, while the condition $p_{n_s}$ defined at stage $s = 8\seq{e_0,e}+4$ (which appeals to \Cref{lem:J0orH1} with $e_0,e$) must force $\Phi^{L_\G}_e$ is not a non-hyperimmunity witness for $S$, or $\Phi^{L_\G}_e$ is finite, on the $1$ side.
	
	Suppose instead that there is an $e_0$ such that no $p_m$ forces $\Phi^{L_\G}_e$ is not a non-hyperimmunity witness for $S$, or $\Phi^{L_\G}_e$ is finite, on the $0$ side. Then the condition $p_{n_s}$ defined at stage $s = 8\seq{e_0,e}+3$ (which appeals to \Cref{lem:H0orH1} with $e_0,e$) must force $\Phi^{L_\G}_e$ is not a non-hyperimmunity witness for $S$, or $\Phi^{L_\G}_e$ is finite, on the $1$ side, while the condition $p_{n_s}$ defined at stage $s = 8\seq{e_0,e}+5$ (which appeals to \Cref{lem:H0orJ1} with $e_0,e$) must force $e \in L_\G'$ or $e \notin L_\G'$ on the $1$ side.
\end{proof}

\noindent For the remainder of this section, we let $i < 2$ as above be fixed.

\subsubsection*{Construction of $G$.} Let $G = \bigcup_m D^i_m$. For each $e \in \omega$, the $G$-condition $(D^0_m,D^1_m,Y_m)$ defined as part of $p_m$ at stage $s = 8e+i$ (which appeals to \Cref{G_infinite} with $e$ and $i$) forces $|\G| \geq e$ on the $i$ side. By \Cref{lem:force_e_in_G_truth}, it follows that $|G| \geq e$ for all $e$ and hence that $G$ is infinite. By definition of $G$-conditions, $G \subseteq A^i$.

\subsubsection*{Construction of $L_G$.} 

For each $m$, let $T_m = \bigcup_n T_{D^i_n,H^i_m,J^i_m}$. Observe that for all $m$ and $n$, $T_{D^i_n,H^i_m,J^i_m}$ looks extendible. If $m \leq n$, this is because $T_{D^i_n,H^i_n,J^i_n} \subseteq T_{D^i_n,H^i_m,J^i_m}$ by \Cref{lem:T_F_structural}~(3), and $T_{D^i_n,H^i_n,J^i_n}$ looks extendible on account of $(D^0_n,D^1_n,Y_n,H^0_n,J^0_n,H^1_n,J^1_n)$ being an $L_G$-condition. On the other hand, if $m > n$, it follows by \Cref{lem:T_F_structural}~(2) that $T_{D^i_n,H^i_m,J^i_m} \subseteq T_{D^i_m,H^i_m,J^i_m}$, and that these trees agree on all strings of length at most $\max D^i_n$. Since $T_{D^i_m,H^i_m,J^i_m}$ looks extendible, it in particular  contains a string of length $\max D^i_n$, and so $T_{D^i_n,H^i_m,J^i_m}$ does too.
 
We conclude, using \Cref{lem:T_F_structural}~(4), that $T_m$ is an infinite subtree of $T^2_{X \oplus G}$. Now, by \Cref{lem:T_F_structural}~(3), $T_0 \supseteq T_1 \supseteq \cdots$, and by compactness of Cantor space we know that $\bigcap_m [T_m] \neq \emptyset$. Let $L_G$ be any element of this intersection. Now $L_G \in T^2_{X \oplus G}$, so $L_G \gg X \oplus G$. 
 
Since $L_G \res \max D^i_m \in T_m$ for all $m$, \Cref{lem:T_F_structural}~(2) implies that $L_G \res \max D^i_m \in T_{D^i_m,H^i_m,J^i_m}$. Fix $e$, and fix $m$ such that $p_m$ forces $\Phi^{L_\G}_e$ is not a non-hyperimmunity witness for $S$, or $\Phi^{L_\G}_e$ is finite, on the $i$ side. In the former case, it follows by \Cref{lem:force_e_in_LG_truth}~(3) that there exist $x_1 > x_0$ such that $\Phi^{L_G}_e(x_0) \converges = \Phi^{L_G}_e(x_1) \converges = 1$ and $[x_0,x_1) \subseteq \overline{S}$. In the latter case, it follows by \Cref{lem:force_e_in_LG_truth}~(4) that for any $k$ such that $\seq{e,k} \in H^i_m$, there are no $x_1 > x_0 > k$ such that $\Phi_e^{L_G \res \max D^i_m}(x_0) \converges = \Phi_e^{L_G \res \max D^i_m}(x_1) \converges = 1$. Since, by \Cref{lem:L_G_forcing_monotone}, forcing is preserved under extension, we can take $m$ arbitrarily large with this property, which means there are no $x_1 > x_0 > k$ such that $\Phi_e^{L_G}(x_0) \converges = \Phi_e^{L_G}(x_1) \converges = 1$. Thus, either $\Phi^{L_G}_e$ is not a non-hyperimmunity witness for $S$ or $\Phi^{L_G}$ is not the characteristic function of an infinite set. Since $e$ was arbitrary, this implies that $S$ is $L_G$-hyperimmune.

\subsubsection*{Construction of $P_G$} By an analogous argument to that in the preceding paragraph, but using \Cref{lem:force_e_in_LG_truth}~(1) and (2), we conclude that $e \in L_G'$ if and only if there is an $m$ such that $p_m$ forces $e \in L_\G'$ on the $i$ side.
By construction, $\lim_m b^i_m = \infty$. Thus, using the notation of \Cref{sec:P_G_conditions},
we see that $\bigcup_m F^i_{p_m,b^i_m} = L_G'$.

Recall that, for all $m$, the tree $T_{F^i_{p_m,b^i_m},V^i_m,\emptyset}$ looks extendible (i.e., contains a string of length $\max F^i_{p_n,b^i_n}$) on account of $(p_m,b^0_m,b^1_m,V^0_m,V^1_m)$ being a $P_G$-condition. We can then argue as in the construction of $T_m$ above that $T_{F^i_{p_n,b^i_n},V^i_m,\emptyset}$ looks extendible for all $m$ and $n$. Thus, again just like before, $S_m = \bigcup_n T_{F^i_{p_n,b^i_n},V^i_m,\emptyset}$ is an infinite subtree of $\bigcup_n T_{F^i_{p_n,b^i_n}} = T^2_{X \oplus L_G'}$. We have $S_0 \supseteq S_1 \supseteq \cdots$, and we let $P_G$ be any element of $\bigcap_m [S_m]$. Then $P_G \gg X \oplus L_G'$ and so $P_G \gg L_G'$ (in fact, these are equivalent facts since $L_G \Tabove X$).

For every $m$, we have $P_G \res \max F^i_{p_m,b^i_m} \in S_m$ and hence, by \Cref{lem:T_F_structural}~(2),  $P_G \res \max F^i_{p_m,b^i_m} \in T_{F^i_{p_m,b^i_m},V^i_m,\emptyset}$. Fix $e$. The condition $(p_{n_s},b^0_{n_s},b^1_{n_s},V^0_{n_s},V^1_{n_s})$ defined at stage $s = 8e+6+i$ either forces $\Phi^{P_\G}_e$ is not a non-hyperimmunity witness for $\overline{S}$, or $\Phi^{P_\G}_e$ is finite, on the $i$ side. In the former case, it follows by \Cref{lem:force_e_in_PG_truth}~(1) that there exist $x_1 > x_0$ such that $\Phi^{P_G}_e(x_0) \converges = \Phi^{P_G}_e(x_1) \converges = 1$ and $[x_0,x_1) \subseteq S$. In the latter case, it follows by \Cref{lem:force_e_in_PG_truth}~(2) that for any $k$ such that $\seq{e,k} \in H^i_{n_s}$, there are no $x_1 > x_0 > k$ such that $\Phi_e^{P_G \res \max F^i_{p_{n_s},b^i_{n_s}}}(x_0) \converges = \Phi_e^{P_G \res \max F^i_{p_{n_s},b^i_{n_s}}}(x_1) \converges = 1$. By \Cref{lem:P_G_forcing_monotone}, this fact remains forced by all conditions with indices $m \geq n_s$, and $\seq{e,k} \in H^i_m$ for this same $k$. So there are no $x_1 > x_0 > k$ such that $\Phi_e^{P_G}(x_0) \converges = \Phi_e^{P_G}(x_1) \converges = 1$. We conclude that either $\Phi^{P_G}_e$ is not a non-hyperimmunity witness for $\overline{S}$ or $\Phi^{P_G}$ is not the characteristic function of an infinite set. Since $e$ was arbitrary, this means $\overline{S}$ is $P_G$-hyperimmune. The proof is complete.
\end{proof}

\subsection{Proving \Cref{P_2_sep_COH_iterate}}

The proof of \Cref{P_2_sep_COH_iterate} is very similar to that of \Cref{P_2_sep_iterate}, with only some modifications to the forcing notions. Essentially, instead of needing to work with two ``sides'' as we did before, we can now work just with one. This makes the overall argument simpler, although the underlying combinatorics remain the same. We therefore outline only the changes below, and omit most of the details.

Throughout, fix $X$, $S$, $L$, $P$, and $\vec{R} \Tred X$ as in the statement of \Cref{P_2_sep_COH_iterate}. (We could assume, if we wished, that there is no infinite $\vec{R}$-cohesive set $U$ such that $L \gg X \oplus U$, since then we could simply take $G = U$, $L_G = L$, and $P_G = P$. But we do not need to make this assumption.) We begin with the modified $G$-conditions.

\begin{definition}[Modified $G$-conditions]
	A \emph{$G$-condition} is a tuple $(D,Y)$ such that $D$ is a finite set, $Y$ is an infinite set with $\max D < \min Y$, and $L \gg X \oplus Y$.
\end{definition}

\noindent \emph{Extensions} and \emph{finite extensions} can be defined for the above modification in the obvious way. It is then straightforward to adapt the statements of \Cref{lem:finite_extension,lem:Pi01_extension} to this setting. (For \Cref{lem:Pi01_extension}, we still look at classes of $2$-partitions of $Y$, so the only change is still only in terms of what conditions we are looking at.) With these in hand, we obtain analogues of \emph{constructible extensions} and \emph{indices}.

We modify \Cref{force_G_inf} by no longer specifying the $0$ side or $1$ side, and add to it, as follows.

\begin{definition}
	Let $(D,Y)$ be a condition, and fix $e \in \omega$.
	\begin{enumerate}
		\item $(D,Y)$ \emph{forces $|\G| \geq e$} if $|D| \geq e$.
		\item $(D,Y)$ \emph{forces $\G \subseteq^* R_e$} if $Y \subseteq R_e$.
		\item $(D,Y)$ \emph{forces $\G \subseteq^* \overline{R_e}$} if $Y \subseteq \overline{R_e}$.
	\end{enumerate}
\end{definition}

\noindent Analogues of \Cref{G_infinite,lem:force_e_in_G_truth} go through as before. (For the latter, we do not need any additional hypotheses on $\vec{R}$-cohesive, as we did before.) We also have the following lemmas. The first of these is clear.

\begin{lemma}
	Let $(D,Y)$ be a condition that, for some $e \in \omega$, forces $\G \subseteq^* R_e$ or $\G \subseteq^* \overline{R_e}$. If $U$ is any set satisfying $D \subseteq U \subseteq D \cup Y$ then $U \subseteq^* R_e$ or $U \subseteq^* \overline{R}_e$, respectively.
\end{lemma}

\begin{lemma}
	Let $(D,Y)$ be a condition and fix $e \in \omega$. Then there is a constructible extension $(E,Z)$ of $(D,Y)$ that forces $\G \subseteq^* R_e$ or $\G \subseteq^* \overline{R_e}$.
\end{lemma}

\begin{proof}
	Let $\mathcal{C}$ be the class of all $2$-partitions $\seq{Y_0,Y_1}$ of $Y$ such that $Y_0 \subseteq R_e$ and $Y_1 \subseteq \overline{R}_e$. Then $\mathcal{C}$ is a $\Pi^0_1(X \oplus Y)$ class, and it is nonempty since, in fact, $\mathcal{C} = \set{\seq{Y \cap R_e,Y \cap \overline{R}_e}}$. We can thus apply the modified version of \Cref{lem:Pi01_extension} to obtain the desired extension.
\end{proof}

We move on to the modified $L_G$-conditions.

\begin{definition}[Modified $L_G$-conditions]
	An \emph{$L_G$-condition} is a tuple $(D,Y,H,J)$ such that $(D,Y)$ is a $G$-condition, and $H$ and $J$ are finite subsets of $\omega$ such that $T_{D \cup F,H,J}$ looks extendible for every finite set $F \subseteq Y$.
\end{definition}

\noindent The definition of \emph{extensions} and \emph{constructible extensions} of $L_G$-conditions are again clear. The analogue of \Cref{G_to_LG_extensions} in this setting is straightforward.

\begin{definition}
	Let $p = (D,Y,H,J)$ be a condition, and fix $e \in \omega$.
	\begin{enumerate}
		\item $p$ \emph{forces $e \in L_\G'$} if $T_{D,H,J \cup \set{e}}$ does not look extendible.
		\item $p$ \emph{forces $e \notin L_\G'$} if $e \in J^i$.
		\item $p$ \emph{forces $\Phi_e^{L_\G}$ is not a non-hyperimmunity witness for $S$} if for some $k \in \omega$ and every $\sigma \in T_{D,H,J}$ of maximal length there exist $x_1 > x_0 > k$ such that $\Phi^\sigma_e(x_0) \converges = \Phi^\sigma_e(x_1) \converges = 1$ and $[x_0,x_1) \subseteq \overline{S}$.
		\item $p$ \emph{forces $\Phi_e^{L_\G}$ is finite} if $\seq{e,k} \in H$ for some $k \in \omega$.
	\end{enumerate}
\end{definition}

\noindent We can now formulate and prove analogues of \Cref{lem:force_e_in_LG_truth,lem:L_G_forcing_monotone} as before. In place of \Cref{lem:J0orJ1,lem:H0orH1,lem:J0orH1,lem:H0orJ1}, we have only the following two.

\begin{lemma}
	Let $(D,Y,H,J)$ be a condition, and fix $e \in \omega$. There exists a constructible extension of $(D,Y,H,J)$ forcing $e \in L_\G'$ or $e \notin L_\G'$.
\end{lemma}

\begin{lemma}
	Let $(D,Y,H,J)$ be a condition, and fix $e \in \omega$. There exists a constructible extension of $(D,Y,H,J)$ forcing $\Phi_e^{L_\G}$ is not a non-hyperimmunity witness for $S$ or $\Phi_e^{L_\G}$ is finite.
\end{lemma}

\noindent To prove the former, we replace the class $\mathcal{C}$ by the class of all $2$-partitions $\seq{Y_0,Y_1}$ of $Y$ such that
\[
	(\forall i < 2)(\forall F \subseteq Y_i \text{ finite})[T_{D \cup F,H,J \cup \set{e}} \text{ looks extendible}],
\]
and to prove the latter we replace, for each $k$, the class $\mathcal{C}_k$ by the class of all $2$-partitions $\seq{Y_0,Y_1}$ of $Y$ such that
\[
	(\forall i < 2)(\forall F \subseteq Y_i \text{ finite})[T_{D \cup F,H \cup \set{\seq{e,k}},J} \text{ looks extendible}].
\]
The rest of the proofs are then exactly as before, mutatis mutandis.

Finally, we define the modified $P_G$-preconditions and $P_G$-conditions.

\begin{definition}[Modified $P_G$-preconditions]
	A \emph{$P_G$-precondition} is a pair $(p,b)$ such that $p$ is an $L_G$ condition, $b \in \omega$, and each $e < b$, $p$ either forces $e \in L_\G'$ or $e \notin L_\G'$.
\end{definition}

\noindent \emph{Extensions} and and \emph{constructible extensions} are then defined, and for each precondition $(p,b)$ we let $F_{p,b} = \set{e < b : p \text{ forces } e \in L_\G'}$. 

\begin{definition}[Modified $P_G$-conditions]
	A \emph{$P_G$-condition} is a tuple $(p,b,V)$ such that $(p,b)$ is a $P_G$-precondition and $V$ is a finite set for which there is no sequence of preconditions
	\[
		(p,b) = (p_0,b_0),\ldots,(p_n,b_n)
	\]
	such that $(p_{m+1},b_{m+1})$ is a constructible extension of $(p_m,b_m)$ as preconditions for all $m < n$ and $T_{F_{p_n,b_n},V,\emptyset}$ does not look extendible.
\end{definition}

\noindent The analogue of \Cref{LG_to_PG_extensions} is straightforward.

We modify \Cref{def:P_G_forcing_types} as follows.

\begin{definition}
	Let $(p,b,V)$ be a condition, and $e \in \omega$.
	\begin{enumerate}
		\item $(p,b,V)$ \emph{forces $\Phi^{P_\G}_e$ is not a non-hyperimmunity witness for $\overline{S}$} if for some $k \in \omega$ and every $\sigma \in T_{F_{p,b},V,\emptyset}$ of maximal length there exist $x_1 > x_0 > k$ such that $\Phi^\sigma_e(x_0) \converges = \Phi^\sigma_e(x_1) \converges = 1$ and $[x_0,x_1) \subseteq S$.
		\item $(p,b,V)$ \emph{forces $\Phi^{P_\G}_e$ is finite} if $\seq{e,k} \in V$ for some $k \in \omega$.
	\end{enumerate}
\end{definition}

\noindent \Cref{P_G_density,lem:force_e_in_PG_truth,lem:P_G_forcing_monotone} can now be appropriately modified for the modified definitions above, and can be proved entirely similarly.

We can now combine all of the modified definitions and lemmas to build a sequence
\[
	(p_0,b_0,V_0),(p_1,b_1,V_1),\ldots
\]
of $P_G$-conditions, each a constructible extension of the previous one. Write $p_m = (D_m,Y_m,H_m,J_m)$. At stage $s+1$, we break into cases based on the congruence class of $s$ modulo $5$, instead of $8$. (We have half as many cases, since we are only working on one ``side'', but one additional case from needing to force cohesiveness.)
\begin{itemize}
	\item If $s = 5e$, we let $n_{s+1} = n_s+1$, and ensure that $	(D_{n_s+1},Y_{n+1})$ forces $|\G| \geq e$.
	\item If $s = 5e + 1$, we let $n_{s+1} = n_s+1$, and ensure that $	(D_{n_s+1},Y_{n+1})$ forces $\G \subseteq^* R_e$ or $\G \subseteq^* \overline{R}_e$.
	\item If $s = 5e + 2$, we let $n_{s+1} = n_s + 1$, and ensure that $p_{n_{s+1}}$ forces $e \in L_\G'$ or $e \notin L_\G'$.
	\item If $s = 5e + 3$, we let $n_{s+1} = n_s + 1$, and ensure that $p_{n_{s+1}}$ forces $\Phi_e^{L_\G}$ is not a non-hyperimmunity witness for $S$ or $\Phi_e^{L_\G}$ is finite.
	\item If $s = 5e + 4$, we apply the modified version of \Cref{P_G_density} to obtain $n_{s+1}$ and the finite sequence $(p_{n_s+1},b_{n_s+1},V_{n_s+1}),\ldots,(p_{n_{s+1}},b_{n_{s+1}},V_{n_{s+1}})$ whose last member forces forces $\Phi^{P_\G}_e$ is not a non-hyperimmunity witness for $\overline{S}$ or $\Phi^{P_\G}_e$ is finite.	
\end{itemize}

\noindent Note that we have no need for any analogue of \Cref{lem:pick_a_side}. The construction of $G$, $L_G$, and $P_G$ can proceed much as it did before.

\section{Summary and questions}\label{sec:qs}

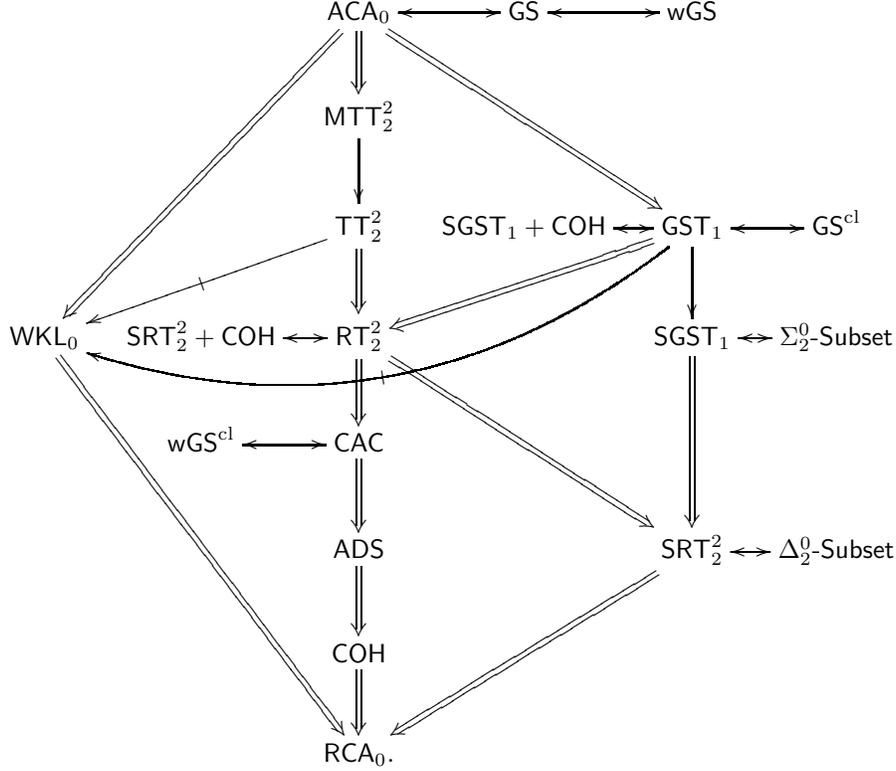
\begin{figure}
\[
\xymatrix @R=2pc @C=1pc{
& & \ACA_0 \ar@2[d] \ar@2[ddrr] \ar@2[dddll] \ar@{<->}[r] & \GS \ar@{<->}[r] & \wGS\\
& & \mathsf{MTT}^2_2 \ar@1[d]\\
& & \TT^2_2 \ar|@{|}[dll] \ar@2[d] & \SGST_1 + \COH \ar@{<->}[r] & \GST_1 \ar@/^3pc/|@{|}[dllll] \ar@2[dll] \ar@1[d] \ar@{<->}[r] & \GS^{\cl}\\
\WKL_0 \ar@2[ddddrr] & \SRT^2_2 + \COH \ar@{<->}[r]& \RT^2_2 \ar@2[d] \ar@2[ddrr] & & \SGST_1 \ar@2[dd] \ar@{<->}[r] & \Sigma^0_2\SubP\\
& \wGS^{\cl} \ar@{<->}[r] & \CAC \ar@2[d]\\
& & \ADS \ar@2[d] & & \SRT^2_2 \ar@{<->}[r] \ar@2[ddll] & \Delta^0_2\SubP\\
& & \COH \ar@2[d]\\
& & \RCA_0.
}
\]
\caption{The state of affairs in the reverse mathematics zoo surrounding Ramsey's theorem for pairs. Arrows indicate implications over $\RCA_0$; double arrows are strict; interrupted arrows are non-implications.}\label{fig:zoo}
\end{figure}

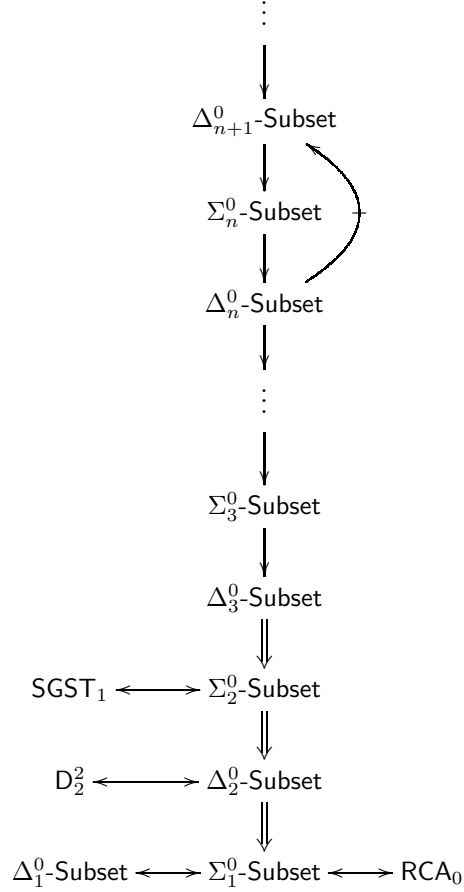
\begin{figure}[h]
\[
\xymatrix @R=1.5pc @C=1.5pc{
& \raisebox{0pt}[0.9\height][0.3\height]{\vdots} \ar@1[d]\\
& \Delta^0_{n+1}\SubP \ar@1[d]\\
& \Sigma^0_n\SubP \ar@1[d]\\
& \Delta^0_n\SubP \ar@1[d] \ar@/_3pc/|@{|}[uu]\\
& \raisebox{0pt}[0.9\height][0.3\height]{\vdots} \ar@1[d]\\
& \Sigma^0_3\SubP \ar@1[d]\\
& \Delta^0_3\SubP \ar@2[d]\\
\SGST_1 & \Sigma^0_2\SubP \ar@2[d] \ar@{<->}[l]\\
\D^2_2 & \Delta^0_2\SubP \ar@2[d] \ar@{<->}[l]\\
\Delta^0_1\SubP & \Sigma^0_1\SubP \ar@{<->}[l] \ar@{<->}[r] & \RCA_0
}
\]
\caption{The hierarchy of subset principles, with $n \geq 4$ arbitrary. Arrows indicate implications over $\RCA_0$; double arrows are strict; interrupted arrows are non-implications.}\label{fig:GammaSubP}
\end{figure}

We summarize our main reverse mathematical results, and how they fit into the literature, in \Cref{fig:zoo}. Our investigation leaves several questions unanswered, and raises some new ones. The first of these concerns the one case of the Ginsburg--Sands theorem for Hausdorff spaces we were not able to fully characterize.

\begin{question}
	What is the strength of the statement that every infinite effectively Hausdorff CSC space has an infinite effectively discrete subspace?
\end{question}

\noindent By \Cref{cor:Q1_partial}, the statement is not provable in $\RCA_0$. But it seems difficult to push this up. The methods we employed in \Cref{prop:ACA_disc_to_eff_disc}, for example, to code the jump into the effectively discrete subspaces of a discrete space $X$ seem incompatible with also making $X$ effectively Hausdorff.

Questions remain about the precise location of $\GST_1$ within the interval of principles strictly in-between $\ACA_0$ and $\RT^2_2$. In particular, we have the following.

\begin{question}
	Does $\GST_1$ imply $\mathsf{MTT}^2_2$? Does it imply $\TT^2_2$? Do either of $\mathsf{MTT}^2_2$ or $\TT^2_2$ imply $\GST_1$?	
\end{question}

With regard to \Cref{cor:GST_1_low_2_solutions}, which employed a modification of the second jump control proof from \cite{CJS-2001} to show that $\GST_1$ admits low$_2$ solutions, we can ask whether the same result can also be obtained by first jump control. More precisely, in \cite[Section 4]{CJS-2001} this method was employed to show that, given any $P \gg \emptyset'$ and any computable $c : [\omega]^2 \to 2$, there exists an infinite homogeneous set $H$ for $c$ satisfying $H' \Tred P$.

\begin{question}
	Given $P \gg \emptyset'$ and an infinite computable $T_1$ CSC space $\seq{X,\topol{U},k}$, does $X$ have an infinite subspace $Y$ which is either discrete or has the cofinite topology, and which satisfies $Y' \Tred P$?
\end{question}

\noindent We conjecture the answer is no, a proof of which would yield another separation of $\GST_1$ from $\RT^2_2$. The problem is in how $P$ could be used to produce a solution to an arbitrary computable instance of $\Sigma^0_2\SubP$. In the case of the proof for $\RT^2_2$, it is crucial that $P$ be able to find, for a given $\Delta^0_2$ set $A$, arbitrarily large elements of both $A$ and $\overline{A}$. This is impossible if $A$ is merely $\Sigma^0_2$.

The second jump control proof in \cite{CJS-2001} was also used to establish a number of conservativity results, perhaps most notably that $\RCA_0 + \mathsf{I}\Sigma^0_2 + \RT^2_2$ is $\Pi^1_1$-conservative over $\RCA_0 + \mathsf{I}\Sigma^0_2$ (\cite{CJS-2001}, Theorem 10.2). This had broadly the same outline as the low$_2$ proof, but differed in one important respect. Recall that this proof only required $A$ to be $\Delta^0_2$ (actually, $\Sigma^0_2$) when building a subset of $A$; it worked equally well for $\Delta^0_3$ (and hence in particular, $\Pi^0_2$) sets when building a subset of $\overline{A}$. This is what made it possible for us to adapt the proof and obtain \Cref{thm:Sig2_low2}. The conservativity proof, by contrast, requires $A$ to be $\Delta^0_2$ for both parts, and thus does not similarly lift to $\Sigma^0_2$ sets.

\begin{question}
	Is $\RCA_0 + \mathsf{I}\Sigma^0_2 + \Sigma^0_2\SubP$ $\Pi^1_1$-conservative over $\RCA_0 + \mathsf{I}\Sigma^0_2$?
\end{question}

\noindent The standard way to produce an affirmative answer would be to show that every model of $\RCA_0 + \mathsf{I}\Sigma^0_2$ is an $\omega$-submodel of $\RCA_0 + \mathsf{I}\Sigma^0_2 + \Sigma^0_2\SubP$. If this could be proved, then it could be combined with the corresponding result for $\COH$ to show that $\RCA_0 + \mathsf{I}\Sigma^0_2 + \GST_1$ is $\Pi^1_1$-conservative over $\RCA_0 + \mathsf{I}\Sigma^0_2$ as well.

Along the same lines, we can ask the following.

\begin{question}
	Does $\Sigma^0_2\SubP$ or $\GST_1$ imply $\mathsf{I}\Sigma^0_2$ over $\RCA_0$?	
\end{question}

\noindent It was shown by Chong, Slaman, and Yang \cite[Corollary 4.2]{CSY-2017} that $\RT^2_2$ does not imply $\mathsf{I}\Sigma^0_2$.

We can also ask whether the decomposition of $\GST_1$ into $\SGST_1 + \COH$ is strict, which harkens back to the famous $\SRT^2_2$ vs.\ $\COH$ problem.

\begin{question}\label{Q:SGST1_COH}
	Does $\Sigma^0_2\SubP$ (or equivalently, $\SGST_1$) imply $\GST_1$? Equivalently, does $\Sigma^0_2\SubP$ imply $\COH$?	
\end{question}

\noindent In their aformentioned paper \cite{MP-2021}, Monin and Patey construct an $\omega$-model satisfying $\D^2_2$ but not $\COH$ by establishing that $\D^2_2$ admits \emph{jump PA avoidance}. This means that for every set $X$ such that $X' \not\gg \emptyset'$, every $\Delta^{0,X}_2$ set $A$ has an infinite subset $Y$ in it or its complement such that $(X \oplus Y)' \not\gg \emptyset'$. They leave open whether a ``strong'' version of this property holds, i.e., whether the same is true of every $A$ (not just those that are $\Delta^0_2$ in $X$). An affirmative answer would, in particular, imply that the answer to \Cref{Q:SGST1_COH} is no.

Finally, we can ask some more general questions about the $\Gamma\SubP$ principles. By \Cref{RT22_not_implies_Sig2}, $\Delta^0_2\SubP$ does not imply $\Sigma^0_2\SubP$, but our proof does not relativize to show that $\Delta^0_n\SubP$ does not imply $\Sigma^0_n\SubP$ for any higher $n$.

\begin{question}\label{Q:GSub_impl1}
	Is there an $n \geq 3$ such that $\Delta^0_n\SubP$ implies $\Sigma^0_n\SubP$?
\end{question}

\noindent It is also true that $\Sigma^0_2\SubP$ does not imply $\Delta^0_3\SubP$. Indeed, by relativizing the main result of Downey, Hirschfeldt, Lempp, and Solomon \cite{DHLS-2001} it follows that for all $n \geq 2$ there is a $\Delta^0_n$ set with no infinite subset in it or its complement that is low over $\emptyset^{(n-2)}$ (see also the discussion in \cite[p.~1372]{DHLS-2001}). Taking $n = 3$ implies, in particular, that there is a $\Delta^0_3$ set with no infinite low$_2$ subset in it or its complement. Hence, $\Delta^0_3\SubP$ fails in the model of $\GST_1$ (and also $\Sigma^0_2\SubP$) constructed in \Cref{low_2_model} above. Again, we do not know if this holds more generally.

\begin{question}\label{Q:GSub_impl2}
	Is there an $n \geq 3$ such that $\Sigma^0_n\SubP$ implies $\Delta^0_{n+1}\SubP$?
\end{question}

\noindent Monin and Patey \cite[Theorem 1.4]{MP-2021b} have shown that for every $n \geq 2$, every $\Delta^0_n$ set has an infinite low$_n$ subset in it or its complement. Combined with the aforementioned result from \cite{DHLS-2001}, this implies that $\Delta^0_n\SubP$ does not imply $\Delta^0_{n+1}\SubP$ for all $n \geq 2$. Thus, we at least know that \Cref{Q:GSub_impl1} and \Cref{Q:GSub_impl2} cannot both be answered affirmatively for the same $n$. The hierarchy of $\Gamma\SubP$ principles is summarized in \Cref{fig:GammaSubP}.

We end with a somewhat open ended question.

\begin{question}
	Are there other natural principles from topology or combinatorics or other areas that can be characterized in terms of the $\Gamma\SubP$ principles? Is there such a principle equivalent to $\Sigma^0_n\SubP + \COH$ for some $n \geq 3$?
\end{question}


\begin{thebibliography}{10}

\bibitem{Benham-TA}
Heidi Benham.
\newblock {P}h.{D}. thesis -- {U}niversity of {C}onnecticut, to appear.

\bibitem{CJS-2001}
Peter~A. Cholak, Carl~G. Jockusch, and Theodore~A. Slaman.
\newblock On the strength of {R}amsey's theorem for pairs.
\newblock {\em J. Symbolic Logic}, 66(1):1--55, 2001.

\bibitem{CLY-2010}
C.~T. Chong, Steffen Lempp, and Yue Yang.
\newblock On the role of the collection principle for {$\Sigma\sp 0\sb
  2$}-formulas in second-order reverse mathematics.
\newblock {\em Proc. Amer. Math. Soc.}, 138(3):1093--1100, 2010.

\bibitem{CLLY-2021}
C.~T. Chong, Wei Li, Lu~Liu, and Yue Yang.
\newblock The strength of {R}amsey's theorem for pairs over trees: {I}. {W}eak
  {K}\"{o}nig's lemma.
\newblock {\em Trans. Amer. Math. Soc.}, 374(8):5545--5581, 2021.

\bibitem{CSY-2014}
C.~T. Chong, Theodore~A. Slaman, and Yue Yang.
\newblock The metamathematics of stable {R}amsey's theorem for pairs.
\newblock {\em J. Amer. Math. Soc.}, 27(3):863--892, 2014.

\bibitem{CSY-2017}
C.~T. Chong, Theodore~A. Slaman, and Yue Yang.
\newblock The inductive strength of {R}amsey's {T}heorem for {P}airs.
\newblock {\em Adv. Math.}, 308:121--141, 2017.

\bibitem{CHM-2009}
Jennifer Chubb, Jeffry~L. Hirst, and Timothy~H. McNicholl.
\newblock Reverse mathematics, computability, and partitions of trees.
\newblock {\em J. Symb. Log.}, 74(1):201--215, 2009.

\bibitem{ACDMP-2022}
Paul-Elliot~Angl{\`e}s d'Auriac, Peter~A. Cholak, Damir~D. Dzhafarov, Benoit
  Monin, and Ludovic Patey.
\newblock {\em Milliken's tree theorem and its applications: a
  computability-theoretic perspective}.
\newblock Memoirs of the American Mathematical Society. American Mathematical
  Society, to appear.

\bibitem{DeLapo-TA}
Andrew DeLapo.
\newblock {P}h.{D}. thesis -- {U}niversity of {C}onnecticut, to appear.

\bibitem{Dorais-2011}
Franccois~G. Dorais.
\newblock Reverse mathematics of compact countable second-countable spaces.
\newblock {\em arXiv: Logic}, 2011.

\bibitem{DH-2010}
Rodney~G. Downey and Denis~R. Hirschfeldt.
\newblock {\em Algorithmic randomness and complexity}.
\newblock Theory and Applications of Computability. Springer, New York, 2010.

\bibitem{DHLS-2001}
Rodney~G. Downey, Denis~R. Hirschfeldt, Steffen Lempp, and Reed Solomon.
\newblock A {$\Delta\sp 0\sb 2$} set with no infinite low subset in either it
  or its complement.
\newblock {\em J. Symbolic Logic}, 66(3):1371--1381, 2001.

\bibitem{DHR-2022}
Damir~D. Dzhafarov, Denis~R. Hirschfeldt, and Sarah Reitzes.
\newblock Reduction games, provability and compactness
\newblock {\em J. Mathematical Logic}, 22(3):Paper No. 2250009, 37, 2022.

\bibitem{Dzhafarov-zoo}
Damir~D. Dzhafarov.
\newblock The {RM} {Z}oo, website: {\tt http://rmzoo.uconn.edu}, 2015.

\bibitem{DJ-2009}
Damir~D. Dzhafarov and Carl~G. Jockusch, Jr.
\newblock {R}amsey's theorem and cone avoidance.
\newblock {\em J. Symbolic Logic}, 74(2):557--578, 2009.

\bibitem{DM-2022}
Damir~D. Dzhafarov and Carl Mummert.
\newblock {\em {R}everse {M}athematics: {P}roblems, {R}eductions, and
  {P}roofs}.
\newblock Theory and Applications of Computability. Springer Nature, 2022.

\bibitem{DP-2017}
Damir~D. Dzhafarov and Ludovic Patey.
\newblock Coloring trees in reverse mathematics.
\newblock {\em Adv. Math.}, 318:497--514, 2017.

\bibitem{GS-1979}
John Ginsburg and Bill Sands.
\newblock Minimal infinite topological spaces.
\newblock {\em Am. Math. Mon.}, 86:574--576, 1979.

\bibitem{HL-1966}
J.~D. Halpern and H.~L\"{a}uchli.
\newblock A partition theorem.
\newblock {\em Trans. Amer. Math. Soc.}, 124:360--367, 1966.

\bibitem{Hirschfeldt-2014}
Denis~R. Hirschfeldt.
\newblock {\em Slicing the Truth: On the Computable and Reverse Mathematics of
  Combinatorial Principles}.
\newblock Lecture notes series / Institute for Mathematical Sciences, National
  University of Singapore. World Scientific Publishing Company Incorporated,
  2014.

\bibitem{HS-2007}
Denis~R. Hirschfeldt and Richard~A. Shore.
\newblock Combinatorial principles weaker than {R}amsey's theorem for pairs.
\newblock {\em J. Symbolic Logic}, 72(1):171--206, 2007.

\bibitem{Hirst-1987}
Jeffry~L. Hirst.
\newblock {\em Combinatorics in Subsystems of Second Order Arithmetic}.
\newblock PhD thesis, The Pennsylvania State University, 1987.

\bibitem{HJ-1999}
Tamara Hummel and Carl~G. Jockusch, Jr.
\newblock Generalized cohesiveness.
\newblock {\em J. Symbolic Logic}, 64(2):489--516, 1999.

\bibitem{Hunter-2008}
James Hunter.
\newblock {\em Higher-order reverse topology}.
\newblock PhD thesis, University of Wisconsin - Madison, 2008.

\bibitem{Jockusch-1972}
Carl~G. Jockusch, Jr.
\newblock {R}amsey's theorem and recursion theory.
\newblock {\em J. Symbolic Logic}, 37:268--280, 1972.

\bibitem{JS-1972}
Carl G. Jockusch, Jr. and Robert I. Soare.
\newblock {$\Pi \sp{0}\sb{1}$} classes and degrees of theories.
\newblock {\em Trans. Amer. Math. Soc.}, 173:33--56, 1972.

\bibitem{JS-1993}
Carl G. Jockusch, Jr. and Frank Stephan.
\newblock A cohesive set which is not high.
\newblock {\em Math. Logic Quart.}, 39(4):515--530, 1993.

\bibitem{MS-2006}
Steffen Lempp and Carl Mummert.
\newblock Filters on computable posets.
\newblock {\em Notre Dame J. Formal Logic}, 47(4):479--485, 2006.

\bibitem{LST-2013}
Manuel Lerman, Reed Solomon, and Henry Towsner.
\newblock Separating principles below {R}amsey's theorem for pairs.
\newblock {\em J. Math. Log.}, 13(2):1350007, 44, 2013.

\bibitem{Liu-2012}
Lu~Liu.
\newblock {$RT^2_2$ does not imply $WKL_0$}.
\newblock {\em J. Symbolic Logic}, 77(2):609--620, 2012.

\bibitem{MM-1995}
D.~J. Marron and T.~B.~M. McMaster.
\newblock A note on minimal infinite subspaces of a product space.
\newblock {\em Irish Math. Soc. Bull.}, (34):26--29, 1995.

\bibitem{Mileti-2004}
Joseph~R. Mileti.
\newblock {\em Partition Theorems and Computability Theory}.
\newblock PhD thesis, University of Illinois at Urbana-Champaign, 2004.

\bibitem{Milliken-1979}
Keith~R. Milliken.
\newblock A {R}amsey theorem for trees.
\newblock {\em J. Combin. Theory Ser. A}, 26(3):215--237, 1979.

\bibitem{Milliken-1981}
Keith~R. Milliken.
\newblock A partition theorem for the infinite subtrees of a tree.
\newblock {\em Trans. Amer. Math. Soc.}, 263(1):137--148, 1981.

\bibitem{Milliken-1975}
Keith~Robert Milliken.
\newblock Some results in Ramsey theory.
\newblock ProQuest LLC, Ann Arbor, MI, 1975.
\newblock Thesis (Ph.D.)--University of California, Los Angeles.

\bibitem{MP-2021}
Benoit Monin and Ludovic Patey.
\newblock {$\sf{SRT}^2_2$} does not imply {$\sf{RT}^2_2$} in {$\omega$}-models.
\newblock {\em Adv. Math.}, 389:Paper No. 107903, 32, 2021.

\bibitem{MP-2021b}
Benoit Monin and Ludovic Patey.
\newblock The weakness of the pigeonhole principle under hyperarithmetical reductions.
\newblock {\em J. Math. Log.}, 21(3), 2150013, 2021.

\bibitem{Mummert-2005}
Carl Mummert.
\newblock {\em On the reverse mathematics of general topology}.
\newblock ProQuest LLC, Ann Arbor, MI, 2005.
\newblock Thesis (Ph.D.)--The Pennsylvania State University.

\bibitem{NS-2019}
Dag Normann and Sam Sanders.
\newblock On the mathematical and foundational significance of the uncountable.
\newblock {\em J. Math. Log.}, 19(1):1950001, 40, 2019.

\bibitem{PK-1978}
J.~B. Paris and L.~A.~S. Kirby.
\newblock {$\Sigma \sb{n}$}-collection schemas in arithmetic.
\newblock In {\em Logic {C}olloquium '77 ({P}roc. {C}onf., {W}roc\l aw, 1977)},
  volume~96 of {\em Stud. Logic Foundations Math.}, pages 199--209.
  North-Holland, Amsterdam-New York, 1978.

\bibitem{Patey-2016b}
Ludovic Patey.
\newblock The strength of the tree theorem for pairs in reverse mathematics.
\newblock {\em J. Symb. Log.}, 81(4):1481--1499, 2016.

\bibitem{Patey-2017b}
Ludovic Patey.
\newblock Iterative forcing and hyperimmunity in reverse mathematics.
\newblock {\em Computability}, 6(3):209--221, 2017.

\bibitem{Sanders-2020}
Sam Sanders.
\newblock Reverse mathematics of topology: dimension, paracompactness, and
  splittings.
\newblock {\em Notre Dame J. Form. Log.}, 61(4):537--559, 2020.

\bibitem{Sanders-2021}
Sam Sanders.
\newblock Nets and reverse mathematics: a pilot study.
\newblock {\em Computability}, 10(1):31--62, 2021.

\bibitem{SS-1995}
David Seetapun and Theodore~A. Slaman.
\newblock On the strength of {R}amsey's theorem.
\newblock {\em Notre Dame J. Formal Logic}, 36(4):570--582, 1995.
\newblock Special Issue: Models of arithmetic.

\bibitem{Shafer-2020}
Paul Shafer.
\newblock The strength of compactness for countable complete linear orders.
\newblock {\em Computability}, 9(1):25--36, 2020.

\bibitem{Shore-2010}
Richard~A. Shore.
\newblock Reverse mathematics: the playground of logic.
\newblock {\em Bull. Symbolic Logic}, 16(3):378--402, 2010.

\bibitem{Simpson-2009}
Stephen~G. Simpson.
\newblock {\em Subsystems of second order arithmetic}.
\newblock Perspectives in Logic. Cambridge University Press, Cambridge, second
  edition, 2009.

\bibitem{Soare-2016}
Robert~I. Soare.
\newblock {\em Turing Computability: Theory and Applications}.
\newblock Springer Publishing Company, Incorporated, 1st edition, 2016.

\end{thebibliography}
\end{document}